%% file: main.tex
\documentclass[11pt]{amsart}

\input{preamble}
\title{Correlated double ramification cycle formula} 
\author{Thomas Blomme, Francesca Carocci}

\address{Université de Neuchâtel, rue \'Emile Argan 11, Neuchâtel 2000, Suisse}
\email{thomas.blomme@unine.ch}

\address{Università di Roma Tor Vergata, Via della Ricerca Scientifica 1, Roma 00133, Italy}
\email{carocci@mat.uniroma2.it}

\subjclass[2020]{14N35; 14N10; 14J26}
\keywords{Enumerative geometry, double ramification cycle, Gromov-Witten invariants, Correlated invariants}

\begin{document}

\begin{abstract}
We prove a refinement of Pixton's formula for the double ramification cycle with target variety which takes into account the correlator of a rubber map introduced in \cite{blomme2024correlated}. To do so, we need to: reinterpret the correlator in terms of (logarithmic) roots of the trivial line bundle;  refine the natural stratification of the boundary of  moduli of maps to keep track of how torsion line bundles on the target variety pull-back. 
We apply the refined DR cycle formula to compute 0-correlated invariants for genus g curve with points and a $\lambda$-class insertion. 
\end{abstract}
\maketitle

\setcounter{tocdepth}{2}
\tableofcontents


\input{introductionv2}

\input{sec-abelian}

\input{sec-log-pic}

\input{sec-splitting-moduli}

\input{sec-moduli-roots}

\input{sec-correlate-DRv2}
\input{sec-explicit-computation}

\appendix

\bibliographystyle{alpha}
\bibliography{biblio}

\end{document}

%% file: preamble.tex
\usepackage[utf8]{inputenc}
\usepackage[english]{babel}
\usepackage{amsmath}
\usepackage{amsthm}
\usepackage{amssymb}
\usepackage{mathrsfs}
\usepackage{dsfont}
\usepackage{array}
\usepackage{graphicx} 
\usepackage[left=3cm,right=3cm,top=3cm,bottom=3cm]{geometry}
\usepackage{hyperref}
\usepackage{tikz-cd}
\usetikzlibrary{shapes}
\usetikzlibrary[patterns]
\usetikzlibrary{decorations.pathreplacing}
\usepackage{caption}
\usepackage{subcaption}
\usepackage{adjustbox}
 \usepackage{comment}
\usepackage{enumitem}
\setitemize{label=$\circ$}

\usepackage{eucal} 

\newcommand{\bcd}{\begin{center}\begin{tikzcd}}
\newcommand{\ecd}{\end{tikzcd}\end{center}}

\newcommand{\vir}[1]{[#1]^\mathrm{vir}}

\newcommand{\ev}{\mathrm{ev}}
\newcommand{\pt}{\mathrm{pt}}

\newcommand{\Pic}{\operatorname{Pic}}
\newcommand{\Alb}{\operatorname{Alb}}
\newcommand{\Hom}{\operatorname{Hom}}
\newcommand{\Aut}{\operatorname{Aut}}
\newcommand{\DR}{\operatorname{DR}}
\newcommand{\LogPic}{\operatorname{LogPic}}
\newcommand{\TroPic}{\operatorname{TroPic}}
\newcommand{\TroJac}{\operatorname{TroJac}}
\newcommand{\Spec}{\operatorname{Spec}}

\newcommand{\Gm}{\mathbb G_m}
\newcommand{\Gmlog}{\mathbb G_{m,\rm{log}}}
\newcommand{\Gmtrop}{\mathbb G_{m,\rm{trop}}}

\newcommand{\gp}{\mathrm{gp}}
\newcommand{\Mbargp}{\overline{M}^\mathrm{gp}}
\newcommand{\Trop}{\mathrm{Trop}}

\newcommand{\rmH}{\mathrm{H}}

\newcommand{\PP}{\mathbb{P}}
\newcommand{\NN}{\mathbb{N}}
\newcommand{\ZZ}{\mathbb{Z}}
\newcommand{\RR}{\mathbb{R}}
\newcommand{\CC}{\mathbb{C}}

\newcommand{\QQ}{\mathbb{Q}}

\newcommand{\GGG}{\mathscr{G}}

\newcommand{\TTT}{\mathscr{T}}

\newcommand{\bfw}{ {\boldsymbol a} }

\newcommand{\tfk}{\mathfrak{t}}

\newcommand{\sfV}{\mathsf{V}}

\newcommand{\sfH}{\mathsf{H}}
\newcommand{\sfL}{\mathsf{L}}
\newcommand{\sfE}{\mathsf{E}}

\newcommand{\A}{\mathcal{A}}

\renewcommand{\P}{\mathcal{P}}
\renewcommand{\L}{\mathcal{L}}

\newcommand{\M}{\mathcal{M}}
\newcommand{\R}{\mathcal{R}}

\newcommand{\T}{\mathcal{T}}

\newcommand{\C}{\mathcal{C}}

\renewcommand{\H}{\mathcal{H}}
\renewcommand{\O}{\mathcal{O}}

\newcommand{\U}{\mathcal{U}}

\renewcommand{\leq}{\leqslant}
\renewcommand{\geq}{\geqslant}
\renewcommand{\tilde}{\widetilde}
\renewcommand{\bar}{\overline}

\newcommand{\gen}[1]{\langle #1 \rangle}

\newtheorem{theo}{Theorem}[section]
\newtheorem*{theom}{Theorem}
\newtheorem{prop}[theo]{Proposition}
\newtheorem{coro}[theo]{Corollary}
\newtheorem{lem}[theo]{Lemma}

\theoremstyle{definition}
\newtheorem{defi}[theo]{Definition}

\theoremstyle{remark}
\newtheorem{remark}[theo]{Remark}
\newenvironment{rem}[1]{
    \begin{remark}#1}{
    \xqed{\blacklozenge}\end{remark}
}

\theoremstyle{remark}
\newtheorem{example}[theo]{Example}
\newenvironment{expl}[1]{
    \begin{example}#1}{
    \xqed{\lozenge}\end{example}
}

\theoremstyle{remark}
\newtheorem{observation}[theo]{Observation}
\newenvironment{obs}[1]{
    \begin{observation}#1}{
    \xqed{\blacklozenge}\end{observation}
}

\newcommand{\xqed}[1]{
    \leavevmode\unskip\penalty9999 \hbox{}\nobreak\hfill
    \quad\hbox{\ensuremath{#1}}}

\def\floor (#1) at (#2,#3); {
    \node[draw,ellipse, minimum width=1cm, minimum height = 0.6 cm] (#1) at (#2,#3) {$\bullet$} ;
}

\def\ufloor (#1) at (#2,#3) (#4); {
    \node[draw,ellipse, minimum width=1cm, minimum height = 0.6 cm] (#1) at (#2,#3) {\scriptsize #4} ;
}

\def\marked (#1) to (#2) pos=#3 in=#4 out=#5; {
   \draw (#1) to[out=#5,in=#4] node[pos=#3] {$\bullet$} (#2) ;
}

\def\leftedge (#1) to (#2) pos=#3 in=#4 out=#5 w=#6; {
   \draw (#1) to[out=#5,in=#4] node[midway,left] {$#6$} (#2) ;
}
\def\leftmarked (#1) to (#2) pos=#3 in=#4 out=#5 w=#6; {
   \draw (#1) to[out=#5,in=#4] node[pos=#3] {$\bullet$} node[midway,left] {$#6$} (#2) ;
}

\def\wlmarked (#1) to (#2) pos=#3 in=#4 out=#5 w=#6; {
   \draw (#1) to[out=#5,in=#4] node[pos=#3] {$\bullet$} node[midway,left] {$#6$} (#2) ;
}

\def\rightedge (#1) to (#2) pos=#3 in=#4 out=#5 w=#6; {
   \draw (#1) to[out=#5,in=#4] node[midway,right] {$#6$} (#2) ;
}
\def\rightmarked (#1) to (#2) pos=#3 in=#4 out=#5 w=#6; {
   \draw (#1) to[out=#5,in=#4] node[pos=#3] {$\bullet$} node[midway,right] {$#6$} (#2) ;
}

\def\doublemarked (#1) to (#2) pos=#3 in=#4 out=#5; {
   \draw (#1) to[out=#5,in=#4] node[pos=#3] {$\bullet$} (#2) ;
}

\usepackage{todonotes}

\makeatletter
\def\l@subsection{\@tocline{2}{0pt}{2.5pc}{5pc}{}}
\makeatother

\makeatletter
\renewcommand{\l@section}{\@tocline{1}{0pt}{10pt}{1pc}{\bfseries}}
\makeatother

%% file: introductionv2.tex
\section{Introduction}

    \subsection{Setting, DR with target, correlation and goal}

Let $X$ be a smooth projective variety and let $L\in\Pic^0(X).$ 
We consider the $\PP^1$-bundle $Y=\PP_X(\O\oplus L)$, together with its boundary divisor $D^\pm=D^+ + D^-=\PP_X(L)\oplus\PP_X(\O)$.

\subsubsection{Relative Gromov-Witten invariants}
We denote by $\M(Y|D^\pm)=\M_{g,m}(Y|D^\pm,\beta,\bfw)$ the moduli of stable log maps $F:(C,\mathbf{p},\mathbf{q})\to Y$ with the following discrete data:
\begin{itemize}
    \item  $g,n,m$ are positive integers;
    \item  $\beta\in H_2(X,\ZZ)$ is an effective homology class;
    \item $\bfw=(a_1,\dots,a_n)$ is a $n$-tuple of non-zero integers with $\sum a_i=0$. 
\end{itemize}

The space $Y$ is endowed with a $\CC^*$-action scaling along the fibers. Let $R\M(Y|D^\pm)$ denote the moduli space of \textit{rubber logarithmic maps} to $(Y,D^\pm)$ with same combinatorial data as before (see for example \cite{janda2020double,marcuswiselog} for definitions). 
There is a natural morphism $\M(Y|D^\pm)\to R\M(Y|D^\pm)$ whose fibers are the lifts of a rubber map to a  relative one. The computation of certain relative Gromov-Witten invariants may be carried out on the rubber space. This is the case of the invariant with point insertions when $m=1$, computed for in \cite{blomme2024correlated}. 

\subsubsection{DR-cycle and Pixton's formula} The rubber space comes with a natural forgetful map to the moduli space of stable maps $\M_{g,n+m}(X,\beta)$, shortly denoted by $\M(X)$:
$$\epsilon\colon R\M(Y|D^\pm) \to \M(X).$$
The push-forward of the virtual fundamental class by $\epsilon$ is  the so-called \textit{double ramification cycle with target $X$}, as defined in \cite{janda2020double}:
\[\DR_{g,\bfw}(X,\beta,L)=\epsilon_*\vir{R\M(Y|D^\pm)}.\]
When $\beta$ and $L$ are clear, we shortly denote the DR-cycle by $\DR_{g,\bfw}(X)$.
Double ramification cycle have a recent yet extremely rich history, and a vast literature on the topic is available  \cite{janda2017DRcurves,janda2020double,bae2023pixton,holmes2021extending,holmes2025logDR}.
In \cite[Section 0]{janda2020double}, the authors provide ax explicit formula,  known as the Pixton formula, expressing the DR-cycle as a combination of \emph{decorated strata classes}, generalizing the previously known case of target a point\cite{janda2017DRcurves}.

In the more modern approach to the study of the DR classes, both results can be recovered as an application of the 
 \textit{universal DR operational class} introduced and studied in  \cite{bae2023pixton}. In [ibid.], the authors define an operational class in the  operational Chow of the universal Picard stack over the moduli space of pre-stable curves:
\[\DR^{\mathrm{op}}_{g,m,\bfw}\in \operatorname{CH}^g_{\mathrm{op}}(\mathfrak Pic_{g,m+n}),\]
which also admits a Pixton's formula expression. Even if the proof requires the results of \cite{janda2020double}, after the  facts, all the previous instances of a DR formula can be recovered as an application of the \textit{universal DR}.

In the case of DR with target variety, the relation between universal DR goes as follows. Given $\L$ a line bundle on a smooth projective variety $X$ there is a natural morphism:

\[
    \Phi:\M_{g,m+n}(X,\beta)\to\mathfrak{Pic}_{g,n+m},\qquad
    (f\colon C\to X)\mapsto f^*\L.
\]
Then:
\[\DR_{g,\bfw,m}(X,\beta,L)=\Phi^*(\DR^{\mathrm{op}}_{g,\bfw,m})\cap \vir{\M_{g,n+m}(X,\beta)}.\]

\subsubsection{Towards a Correlated  DR cycle}

Assume that $\delta$ divides the g.c.d of the tangency orders $a_i$. 
We can then define a map $\kappa^{\delta}\colon \M(Y|D^\pm)\to \Alb(X)$ as follows: let $f:C\to X$ be stable map coming from a log stable map $F:C\to Y$, we define the \textit{correlator} to be
\[\kappa^\delta(F:C\to Y) = \sum_i \frac{a_i}{\delta}a_X\circ f(p_i),\]
where $a_X\colon X\to\Alb(X)$ is the Albanese map and the sum is performed in the Albanese variety of $X$. 
In \cite{blomme2024correlated}, it is proven that 
$\kappa^\delta$  takes values in the $\Alb(X)[\delta]$-torsor $T^\delta(L,\beta)$ of $\delta$-roots of $\varphi_{\beta}(L)$, a certain element in $\Alb(X)$ which only depends on $L$ and $\beta$. 
Throughout this paper, $-[\delta]$ means $\delta$-torsion. 

\begin{obs}
When the domain is a smooth curve $C$ we can reinterpret $\kappa^{\delta}$ as follows. Let $f_*\colon\Pic^0(C)\xrightarrow{\cong}\Alb(C)\to \Alb(X) $ be the composition of the inverse of the Abel-Jacobi map and the dual to the pull-back. Then $\kappa^\delta(F:C\to Y)=f_*\O(\sum\frac{a_i}{\delta}p_i)$, where the line bundle 
$\O(\sum\frac{a_i}{\delta}p_i)$ is a $\delta$-root of $f^*L$ which carries the \textit{correlation} information.
\end{obs}

As the correlator takes discrete values, the \textit{correlating map} $\kappa^\delta$ thus splits the moduli space $\M(Y|D^\pm)$ into disjoint (possibly still disconnected) components $\M^\theta(Y|D^\pm)$ indexed by the correlators $\theta\in T^\delta(L,\beta)$. In particular, the virtual class splits as a sum of  the so-called \textit{correlated classes} $\vir{\M^\theta(Y|D^\pm)}$. We define \textit{correlated GW invariants} by integrating cohomology classes on the correlated virtual classes. These invariants refine the usual Gromov-Witten theory on $\M(Y|D^\pm)$. We refer to \cite[Section 1.6]{blomme2024correlated} for details motivating the study of correlated GW invariants.

Since the image of $\kappa^\delta$ does not depend from the choice of the rational  section of $f^*L$ with given zeros and poles, the correlating map factors through the moduli space of rubber maps, which thus also splits naturally, as well as its virtual class. This induces a splitting of the DR cycle with target $X$:
$$\DR_{g,\bfw}(X) = \sum\epsilon_*\vir{R\M^\theta(Y|D^\pm)}.$$
The terms appearing in the right hand-side are  called \textit{correlated double ramification cycles}, denoted by $\DR_{g,\bfw}^\theta(X,\beta,L)$, or in a shorter way $\DR_{g,\bfw}^\theta(X)$ when the discrete data is clear.

\medskip

The main focus of this paper is to get a formula for the $\DR_{g,\bfw}^\theta(X).$ More precisely, we provide a Pixton type formula for $\DR_{g,\bfw}^0(X)$ in the case where $L=\O$. Deformation invariance, proven in \cite{blomme2024correlated}, ensures that the formula also applies for \emph{special correlators} ( See \cite[Theorem~5.3]{blomme2024correlated}) in the case where $L$ is non-trivial. 

When the mapping class group of $X$ acts transitively over the torsion elements of its Albanese variety, $\DR_{g,\bfw}^0(X)$ is enough to compute all $\DR_{g,\bfw}^\theta(X)$ through \textit{unrefinement relations} \cite[Lemma~3.12]{blomme2024correlated}.
Such an assumption is true when $X$ is a curve. This fact is already used in \cite{blomme2024correlated} to compute punctual invariants in the case where $X$ is an elliptic curve.

\subsection{Strategy}\label{sec:Strategy}

In order to get a Pixton formula for the correlated DR cycles $ \DR_{g,\bfw}^\theta(X):=\epsilon_*\vir{R\M^\theta(Y|D^\pm)},$ one could attempt 
carefully going along the proof of \cite{janda2020double}, namely: define a correlator map on the  moduli space  of orbifold/relative maps to the $\mathbb P^1$-bundle, and then perform the virtual localization analysis  on a component with fixed correlator. We did not try to go down this route. 

Instead, we prefer to use a different approach, relying on the application of the universal DR formula.



\medskip

The above definition of correlated DR cycles is obtained splitting the moduli space of rubber maps $R\M(Y|D^\pm)$ using the discrete values of $\kappa^\delta$. In order to obtain a Pixton type formula for the correlated components, the naive idea would be to use the map $\kappa^\delta$, whose definition makes sense on all $\M(X)$ (i.e. not only on the image of $\epsilon)$ to split the latter in open and closed components, and then apply the universal DR formula to each of them.
However, this does not quite work as the $\kappa^\delta$ defined above in terms of images of the marked point $q_i$ through $f$ has no reason to take discrete values: away from the DR locus there is no relation between $\kappa^{\delta}(f)$ and the $\delta$-roots of $f_*L$. 

\medskip

To remedy this problem, we consider the moduli of \textit{spin stable maps} $\M(X)(1/\delta)$, adding the data of a $\delta$-root $\T$ of $f^*L$ to the stable map to $X$. Constructing such space requires some care since the universal curve over $\M(X)$ does not only have smooth fibers. The problem has been solved using logarithmic geometry methods by \cite{holmes2023root}; we review their approach and how to apply it to our situation in Section~\ref{sec-various-moduli-of-curves-including-spin}.
 In the end, the space $\M(X)(1/\delta)$ comes endowed with a universal root $\T$; to be precise, this will be a root of $f^*L$ thought as a \emph{logarithmic line bundle} in the sense of \cite{molchowiselog}. 

\medskip

 For the moment, in order to give an idea of our approach, let's ignore the difficulties that arise at the boundary, i.e. when the curve becomes nodal.
\medskip

The key point is that it is possible to define a \textit{correlator} $\theta$ for a spin stable map as the image of the $\delta$-root $\T$ by the push-forward map
\[f_*\colon \Alb(C)\simeq\Pic^0(C)\longrightarrow\Alb(X).\]
The relation between this definition and the one from \cite{blomme2024correlated} presented above is explained in Section \ref{sec-link-correlators-line-bundle-stable-maps}, where, said imprecisely, we prove that the two definitions coincide over DR loci.

Using the new definition of the correlator map, we show that the moduli of spin stable maps splits into  a union of open and closed distinct (possibly still disconnected) components. Again, extending the definition over the locus of non smooth curves requires work and, again, the solution is found within logarithmic geometry.

\medskip

We \emph{now}  want to apply the universal DR-formula on each component to get a correlated-refined class for a suitable compactification of  the DR locus where $\T=\O(\sum\frac{a_i}{\delta}p_i)$ and finally push forward to to $\M(X)$.

\begin{rem}
The correlation refinement can also be seen as coming from the following observation. Let us keep working on the locus of smooth curves to avoid technicalities related to the issues mentioned above. The DR-locus consists of stable maps $f\colon C\to X$ such that $f^*L=\O(\sum a_i p_i)$. Now, two line bundles agree if and only if their torsors of $\delta$-roots also agree. However, both torsors carry natural trivializing sections:
\begin{itemize}
    \item $\O(\sum\frac{a_i}{\delta}p_i)$ in the case of $\O(\sum a_ip_i)$;
    \item $f^*L_0$ where $L_0$ is a $\delta$-root of $L$ on $X$ in the case of $f^*L$.
\end{itemize}
These sections may not agree, and their difference is some torsion line bundle. Any continuous discrete function on the group of torsion line bundle, e.g. the order in the group, can be used to split the DR-locus. In our case, we use the \textit{Weil pairing} with pull-back of torsion line bundle from $X$, which is yet another way to incorporate the correlator information.

In a nutshell,
 \textit{correlation} is about comparing the natural sections of these torsors.
\end{rem}
Though the plan sounds easy enough , we need along the way to overcome some technical problems arising, as already mentioned above, when dealing with singular curves. We summarize these issues below.
\begin{itemize}[leftmargin=0.4cm]
    \item First, the degree of the projection from spin stable maps with fixed correlator to stable maps is not constant but depends  on the cardinality of
    the kernel $K$ of the pull-back map $f^*:\Pic^0(X)[\delta]\to\Pic^0(C)[\delta].$
The key observation is contained in Lemma \ref{lem-pull-back-open-closed} which proves that
$\ker(f^*)\subseteq \Pic^0(X)[\delta]$ is locally constant.  

We denote by $\M_K(X)$ the open and closed component  where $\ker(f^*)=K$ (Section~\ref{sec-refined-strat-overview}).
  Later in Section \ref{sec-correlator-torsion line bundle}, we discuss how this splitting in components relate to the one induced by the correlator map allowing us to compute the desired pushforward. 
 
    \item The ideas presented above are valid when the domain $C$ is smooth, but need to be extended to nodal curves, for which we encounter the following issues:
        \begin{itemize}[leftmargin=0.4cm,label=$\triangleright$]
            \item The group of multidegree $0$ line bundles $\Pic^{[0]}(C)$ is not an abelian variety anymore, so we need to make sense of $\Alb(C)$ and $f_*;$
            \item The number of $\delta$-torsion line bundles depends on the genus of the dual graph.
        \end{itemize}
\end{itemize}
After solving these  technical issue, we can  write down  a formula for the correlated DR  cycle.

\subsection{Sections content  and results}

We explain more carefully the content of the  various sections. Sections \ref{sec-abelian-stuff}, \ref{sec-log-picard-group}, \ref{sec-refined-stratification-via-coverings} and \ref{sec-various-moduli-of-curves-including-spin} are independent of each other and may be of seperate interest, besides their application to Section \ref{sec-correlated-DR-where-evth-comes-to-place}, where we prove the correlated DR-formula.

\subsubsection{Abelian varieties and Weil pairing} In Section \ref{sec-abelian-stuff}, we recall the construction of the Weil pairing for smooth abelian varieties and the definition of log-abelian varieties \cite{kajiwara2008logabelian2}. We then propose an extension of the Weil pairing to the world of log-abelian varieties in the particular case of log-abelian varieties with constant degeneration. We expect this definition to extend to the case of a general log-abelian variety.

\subsubsection{Log-Picard group} Section \ref{sec-log-picard-group} is devoted to  the log-Picard group of a log smooth curve, as introduced in \cite{molchowiselog}. It is proved in \cite{molchowiselog}, that  the degree 0 part of the log Picard of a log curve with constant degeneration is a log-torus in the sense of \cite{kajiwara2008logabelian2}. To prove it is a log-abelian variety, we show that it possesses a \textit{polarization}, which furthermore provides an isomorphism with its \textit{dual} by results of \cite{delignepair}. We also provide a description of the logarithmic Abel-Jacobi map, and prove that the logarithmic Picard group with is initial for morphism from log smooth curves to abelian varieties;  we later use these results to relate the definition of correlator of the present paper to the one from \cite{blomme2024correlated}.

\subsubsection{Splitting of $\M(X)$}

In Section \ref{sec-refined-stratification-via-coverings}, we prove that $\M(X)$ splits into disjoint components, depending on how the  kernel of the pull-back map $f^*:\Pic^0(X)[\delta]\to\Pic^{[0]}(C)[\delta]$  varies.
We denote by $\M_K(X)$ the open and closed component  where $\ker(f^*)=K$.

The virtual classes induced by the natural obstruction  theory on $\M_K(X)$ are related, up to inclusion-exclusion combinations, to the classes  of moduli spaces of stable maps to certain coverings of $X$, which we call $H$-coverings, introduced in Section \ref{sec-H-covering}. This enables computations with these classes, proving they are \textit{tautological}.

The moduli space $\M(X)$ inherits a natural stratification from the forgetful map to the Artin stack of nodal curves, where strata are indexed by dual graphs. This stratification is naturally refined by keeping track of how the curve class splits; the strata are then indexed by the so-called $X$-valued stable graph, where the graphs a decorated by an homology class $\beta_v\in H_2(X,\ZZ)$ attached to each vertex.
The stratification is \textit{recursive} since each strata is related to the product of moduli spaces indexed by vertices of the dual graph. 

This recursive structure and natural stratification  interact with the above splitting in open and closed components indexed by the  restriction-kernel $K.$ The description of of the boundary  of $\M_K(X)$ occupies the remainder of the section. More precisely, restricting to $\M_K(X)$, the boundary strata are indexed by $(\Gamma,\widetilde{K},\varphi)$ where
    \begin{itemize}
        \item $\Gamma$ is a $X$-valued stable graph,
        \item $\widetilde{K}$ is a subgroup of $\Pic^0(X)[\delta]$ containing $K$,
        \item $\varphi\colon\widetilde{K}\otimes\rmH_1(\Gamma,\ZZ_\delta)\to\mu_\delta$ is a pairing with left kernel $K$: it induces $\widetilde{K}/K\hookrightarrow \rmH^1(\Gamma,\mu_\delta)$.
    \end{itemize}
Concretely, $\widetilde{K}$ is the subgroup of line bundles that pull-back trivially componentwise, and the map $\varphi$ determines the global pull-back on a curve in the corresponding stratum.  Considering $\M_K(X)$ with the logarithmic structure induced by the underlying pre-stable curve, we have that 
 the above discrete data also indexes the cones of the cone stack $\Sigma_K(X)$ which we call the tropicalization of $\M_K(X)$.

\subsubsection{Moduli of spin stable maps} In Section \ref{sec-various-moduli-of-curves-including-spin}, following \cite{holmes2023root}, we construct the moduli $\M(X)(\frac{1}{\delta})$ of \textit{spin} stable maps to $X$ by adjoining the data of a (log-)$\delta$-root of the trivial bundle $\O$. The necessity to consider log-roots arises for several reasons, one of them being to always have $\delta^{2g}$ $\delta$-roots to a given line bundle, even when going to the boundary.

From the projection to $\M(X)$, the moduli of spin stable maps inherits a stratification: restricting to the component $\M_K(X)$, the cones of the tropicalization (intended in the sense  of the previous section) $\Sigma_K(X)(\frac{1}{\delta})$ of $\M_K(X)(\frac{1}{\delta})$ are indexed by $(\Gamma,\widetilde{K},\varphi,\mathtt{T})$, where $(\Gamma,\widetilde{K},\varphi)$ is as above and $\mathtt{T}$ is a linear equivalence class of $\delta$-torsion tropical divisor on $\widetilde{\Gamma}$, the graph obtained subdividing any edge of $\Gamma$ into $\delta$ parts of equal length.
Furthermore, the construction from \cite{holmes2023root} exhibits a particular choice $D$ of tropical divisor in the equivalence class $\mathtt{T}$.

We will use that,  via the correlator map, we can see that this moduli space is a union of open and closed components indexed by the correlator. 
Considering each of these components with the induced logarithmic structure, we get a cone stack, the tropicalization of a component with fixed correlator, 
whose cones are indexed by a  refined data keeping track of the correlator information. 

\subsubsection{The correlated DR formula and its refinement} In Section \ref{sec-correlated-DR-where-evth-comes-to-place}, all previous ingredients come together to prove the correlated DR formula, stated in Theorem \ref{theo-correlated-DR-formula}. The section is organized as follows:
    \begin{itemize}[leftmargin=0.5cm]
        \item First, we recall the Pixton formulas for DR-cycle with target \cite{janda2020double} and universal DR \cite{bae2023pixton}.
        \item  We define correlators for torsion log-line bundles (extending the definition for smooth curves mentioned in Section~\ref{sec:Strategy}). We use the latter to split the moduli space of spin stable maps. We then investigate the interaction between the correlation refinement and its natural boundary stratification.
        \item We relate the definition of correlators for torsion log-line bundles and the definition of correlators for stable log map to $Y$ from \cite{blomme2024correlated} via the log Abel-Jacobi map.
        \item Finally, we apply the universal DR formula to components of the moduli space of spin stable maps with fixed correlator and then push-forward to $\M(X)$ to obtain the correlated DR-formula.
    \end{itemize}

More precisely, we have the following splitting:
$$\M(X)(\frac{1}{\delta}) = \bigsqcup_\theta \M^\theta(X)(\frac{1}{\delta}),$$
which we can also restrict to the component $\M_K(X)(\frac{1}{\delta})$. We then consider the $0$-correlator part. Due to the interaction between the correlation and the refined stratification induced by restriction-kernel, the tropicalization $\Sigma_K^0(X)(\frac{1}{\delta})$ of $\M_K^0(X)(\frac{1}{\delta})$ has cones indexed by $(\Gamma,\widetilde{K},\varphi,\mathtt{T})$ where $\mathtt{T}$ lies in the right kernel of $\varphi$. This means every strata (i.e. any choice of $\mathtt{T}$) does not show up on the boundary of $\M_K^0(X)(\frac{1}{\delta})$.

It  is  convenient to rephrase the  formula in terms of piece-wise polynomials. We refer the reader to \cite[Section~6]{holmes2025logDR} and references therein for results  on piece-wise polynomial on  Artin fans  and a vocabulary  translating these into decoarted strata classes. We have a map $\Sigma_K^0(X)(\frac{1}{\delta})\to\Sigma(\frac{1}{\delta})$, so that we can pull-back piecewise polynomial functions on $\Sigma(\frac{1}{\delta})$. We consider the following ones, which are a  slight variation  of those  already  considered  in \cite{holmes2023root,holmes2025logDR} to prove Pixton's  formula for Spin and Log DR respectively.:
\begin{enumerate}[leftmargin=0.5cm]
    \item $\mathfrak{P}$ defined on the cone indexed by $(\Gamma,\widetilde{K},\varphi,\mathtt{T})$ by the formula
    $$\mathfrak{P}|_{\widetilde{\sigma}_{\Gamma,\widetilde{K},\varphi,\mathtt{T}}} = \left.\sum_{w\in W_{\widetilde{\Gamma},r}(\bfw/\delta)}\frac{1}{r^{b_1(\Gamma)}}\prod_{e=(h,h')} e^{\frac{w(h)w(h')}{2}l_e}\right|_{r=0},$$
    where the sum is over weightings with divergence $\mathtt{T}$, the preferred divisor in the class $\mathtt{T}$, and slopes at infinity $\bfw/\delta$. The expression is a polynomial in $r$ for $r$ big enough, the $|_{r=0}$ means we take the constant coefficient;
    
    \item $\mathfrak{L}$ defined on the cone indexed by $(\Gamma,\widetilde{K},\varphi,\mathtt{T})$ by the formula
    $$\mathfrak{L}|_{\widetilde{\sigma}_{\Gamma,\widetilde{K},\varphi,\mathtt{T}}} = \frac{1}{\delta^2}\sum_{v\in\sfV(\widetilde{\Gamma})}\alpha_\mathtt{T}(v)\cdot\delta \mathtt{T}(v),$$
    where $\alpha_{\mathtt{T}}$ is piecewise linear function on $\widetilde{\Gamma}$ with divisor $\delta\mathtt{T}$, so that $\delta \mathtt{T}(v)$ is actually the degree of $\O(\alpha_{\mathtt{T}})$ at $v$.
\end{enumerate}
We then consider the piecewise polynomial function $e^{-\frac{1}{2}\mathfrak{L}}\cdot\mathfrak{P}$. Its expression may seem to depend on the choice of representatives $\mathrm{T}$ as both factor do, but \cite[Lemma 6.4]{holmes2023root} proves it is actually not the case, and it only depends on $\mathtt{T}$.

We have a map of fans $\Sigma_K^0(X)(\frac{1}{\delta})\to\Sigma_K(X)$ forgetting the tropical torsion information. Thus, we can push forward  piecewise polynomial functions once we know  the degree of the projection on each stratum. The computation of this degree carried in Lemma \ref{lem-degree-proj-corr0}; this is the main reason for introducing the new refined stratification on $\M(X)$. We define a piecewise polynomial function on the cone $\sigma_{\Gamma,\widetilde{K},\varphi}$ of $\Sigma_K(X)$ summing over the cones $\sigma_{\Gamma,\widetilde{K},\varphi,\mathtt{T}}$ of $\Sigma_K^0(X)(\frac{1}{\delta})$ for all $\mathtt{T}$. We define two piecewise polynomial functions: one yielding the $0$-correlator part of the correlated DR-cycle, and one with values in the group algebra $\QQ[\Alb(X)]$ whose $(\theta)$-coefficient is the $\theta$-correlator part of thecorrelated DR-cycle. To this end, for $\mathtt{T}\in \rmH_1(\Gamma,\ZZ_\delta)$ and $\varphi\colon\rmH_1(\Gamma,\ZZ_\delta)\to \operatorname{Hom}(\widetilde{K},\mu_\delta)\simeq\Alb(X)[\delta]/\widetilde{K}^\perp$, we set $T_{\varphi(\mathtt{T})}=\frac{1}{|\widetilde{K}^\perp|}\sum_{\theta\in\varphi(\mathtt{T})}(\theta)$ to be the group algebra average of the $\widetilde{K}^\perp$-torsor $\varphi(\mathtt{T})\subset\Alb(X)[\delta]$. We get
\begin{align*}
\mathfrak{DR}_K|_{\sigma_{\Gamma,\widetilde{K},\varphi}} = & \frac{\delta^{2g-b_1(\Gamma)}}{|\widetilde{K}^\perp|}\sum_{\mathtt{T}\in\TTT} e^{-\frac{1}{2}\mathfrak{L}|_{\widetilde{\sigma}(\Gamma,\widetilde{K},\varphi,\mathtt{T})}}\mathfrak{P}|_{\widetilde{\sigma}(\Gamma,\widetilde{K},\varphi,\mathtt{T})}, \\
\underline{\mathfrak{DR}}_K|_{\sigma_{\Gamma,\widetilde{K},\varphi}} = & \delta^{2g-b_1(\Gamma)}\sum_{\mathtt{T}\in\TTT} e^{-\frac{1}{2}\mathfrak{L}|_{\widetilde{\sigma}(\Gamma,\widetilde{K},\varphi,\mathtt{T})}}\mathfrak{P}|_{\widetilde{\sigma}(\Gamma,\widetilde{K},\varphi,\mathtt{T})} \cdot T_{\varphi(\mathtt{T})}.    
\end{align*}
The correlated DR-cycle is then obtained summing over the various components $\M_K(X)$.

\begin{theom}\textbf{\ref{theo-correlated-DR-formula}}
As a strict piecewise polynomial function, the correlator $0$ part of the DR-cycle has the following expression:
    $$\DR_{g,\bfw}^0(X,\beta) = \sum_{K} e^{-\frac{1}{2}\sum(\frac{a_i}{\delta})^2\psi_i}\mathfrak{DR}_K,$$
    where each summand yields a class in $\M_K(X)$, and we take the degree $g$ part. The total correlated DR-cycle, with values in the group algebra $\QQ[\Alb(X)]$ whose $(\theta)$-coefficient gives the value for correlator $\theta$ is given by
    $$\underline{\DR}_{g,\bfw}(X,\beta) = \sum_{K} e^{-\frac{1}{2}\sum(\frac{a_i}{\delta})^2\psi_i}\underline{\mathfrak{DR}}_K,$$
    where each summand yields a class in $\M_K(X)$, and we take the degree $g$ part.
\end{theom}
 In Section~\ref{sec:loglift} we also discuss a lift of the correlated DR cycle to a class in the logarithmic Chow of $\M(X)$ for some particular choices of targets $X.$ We defer to a future work the discussion of the correlated class to the logarithmic Chow in full generality. The case treated here is sufficient for the application in \cite{blomme2025multiple}.
\subsubsection{Applications to concrete computations} Using the correlated DR-formula we are able to recover the computation of \cite{blomme2024correlated} for correlated invariants with point constraints, and we are able to include a $\lambda$-class as well.

Let $X=E$ be an elliptic curve and $Y=E\times\PP^1$, $\beta=d[E]$, $g=1$, $m=1$ interior marked point and $n$ log-points with associated weights $a_i$. We then also set $a_0=0$. Let $\delta$ be common divisor of the weights $a_i$ to consider the level $\delta$ refinement. Insertions are chosen to be points for every point except for the first log point, whose weight is $a_1$. The  correlator $0$ counterpart for level $\delta$ is denoted by:
$$N^0_d(\bfw) = \gen{\pt_0,1_{a_1},\pt_{a_2},\dots,\pt_{a_n}}^{Y/D^\pm,0}_{1,d[E]}.$$

\begin{theom}\textbf{\ref{theo-computation-points}}
We have the following identity
$$N^0_d(\bfw) = \left(\frac{a_1}{\delta}\right)^2d^{n-1}\sum_{\omega|\delta}J_2(\omega)\sigma\left(\frac{d}{\omega}\right) = a_1^2\cdot d^{n-1}\cdot \sum_{l|d}\frac{d}{l}\cdot \frac{\gcd(l,\delta)^2}{\delta^2}.$$
\end{theom}

We now increase the genus from $1$ to $g$, and to match the virtual dimensions we insert an additional $\lambda$-class constraint $\lambda_{g-1}$, which is the $(g-1)$th-chern class of the Hodge bundle:
$$N^0_{g,d}(\bfw) = (-1)^{g-1}\gen{\lambda_{g-1};\pt_0,1_{a_1},\pt_{a_2},\dots,\pt_{a_n}}^{Y/D^\pm,0}_{1,d[E]}.$$

\begin{theom}\textbf{\ref{theo-computation-lambda}}
    We have the following identity:
    $$N^0_{g,d}(\bfw) = \frac{a_1^2}{a_1\cdots a_n}\left(\sum_{S\subset\{1,\cdots,n\}}(-1)^{|S|}a_S^{2g-2+n}\frac{(-1)^{n+g-1}}{(n+2g-2)!}\right)\cdot d^{n-1}\sum_{k|d}\left(\frac{d}{k}\right)^{2g-1}\frac{\gcd(k,\delta)^2}{\delta^2}.$$
\end{theom}

We denote by $\underline{N}^\delta_{d}(\bfw)$ and $\underline{N}^\delta_{g,d}(\bfw)$ the associated full correlated invariants, which are group algebra elements such that the $(\theta)$-coefficient provides the correlated invariant for the correlator $\theta$. From the $0$-coefficient and unrefinement relations, or alternatively the full correlated DR-cycle formula, we can provide an expression very close to the usual uncorrelated one. Let $\underline{\sigma}_m^\delta(d)=\sum_{k|d}\left(\frac{d}{k}\right)^mT_{\delta/\gcd(\delta,k)}$ be the group algebra refinement of the sum of divisors functions, where we denoted by $T_p$ the group algebra average of $p$-torsion elements. One gets the following group algebra expressions:
\begin{align*}
    \underline{N}^\delta_{0,d}(\bfw) = & a_1^2\cdot d^{n-1}\cdot \underline{\sigma}^\delta_1(d), \\
    \underline{N}^\delta_{g,d}(\bfw) = & \frac{a_1^2}{a_1\cdots a_n}\left(\sum_{S\subset\{1,\cdots,n\}}(-1)^{|S|}a_S^{2g-2+n}\frac{(-1)^{n+g-1}}{(n+2g-2)!}\right)\cdot d^{n-1} \cdot\underline{\sigma}^\delta_{2g-1}(d). \\
\end{align*}

\bigskip

\textit{Acknowledgements.} The authors would like to thank Dhruv Ranganathan for discussions in Belalp suggesting the universal DR approach to tackle the problem, and Patrick Kennedy-Hunt, Sam Molcho, Jonathan  Wise for discussions about the log Abel-Jacobi theorem and  in particular Sam Molcho, Jonathan  Wise and Martin Ulirsch for communicating their results on LogPic. They further thank David Holmes for answering several questions about the  spin DR formula. They finally thank Felix Janda for having spotted a mistake in the Introduction of an earlier version. T.B. thanks F.C. for the nice working conditions in Rome where part of the work was handled. T.B. is partly funded by the SNF grant 204125.

%% file: sec-abelian.tex
\section{About abelian varieties and Weil pairing}
\label{sec-abelian-stuff}

In this section we review the definition of abelian and log-abelian varieties, a generalization introduced in \cite{kajiwara2008logabelian2}. We then review several prospective on the Weil pairing for smooth abelian varieties and propose an extension in the setting of log-abelian varieties with constant degeneration.

\subsection{Abelian varieties}
     Following \cite[Chapter 2.6]{griffiths2014principles}, a \textit{complex torus} is the quotient of a complex vector space $V$ by a full-dimensional lattice $L\subset V$. The complex torus $V/L$ is an \textit{abelian variety} when it is a projective variety. This is equivalent to the existence of an ample line bundle on $V/L.$

    \subsubsection{Dual torus}
    \label{sec-dual-torus-abelian-var}
  Consider the complex torus $A=V/L$. 
  
    The dual torus parametrizes holomorphic degree $0$ complex line bundles over $V/L$. It may be obtained as follows.
 Following \cite{griffiths2014principles}, a complex line bundle on $A$ is determined by an action of $L$ on the trivial bundle $V\times\CC$. In the degree $0$ case, such an action is determined by the choice of a character $\chi:L\to\CC^*$;  the line bundle on $A$ is then the quotient of $V\times\mathbb C$ by the action
     $$l\cdot(v,\xi)=(v+l,\chi(l)\xi).$$
    Up to choosing a lift for a basis of $L$, every character $\chi:L\to\CC^*$ is of the form $\chi(l)=e^{2i\pi\varphi(l)}$ for some $\varphi:L\to\CC$. We thus consider $\Hom_\ZZ(L,\CC)$, naturally identified with  $\Hom_\RR(V,\CC)$ since $V$, as real vector space, is isomorphic to $L\otimes\RR$. 

We denote by $J$ the multiplication by imaginary unit $i$ of the underlying real vector space $V$; this induces an action on $\Hom_\RR(V,\CC)$ by pre-composing and a splitting of the latter in its $i$-eigenspace and $(-i)$-eigenspace. 

In the identification with $\Hom_\ZZ(L,\CC)$, they correspond to the restriction of morphisms $V\to\CC$ which are $\CC$-linear or $\CC$-anti-linear. The dual torus is a quotient $\Hom_\RR(V,\CC)$ by:
     \begin{itemize}
         \item The subspace $\Hom_\CC(V,\CC)\subset\Hom_\RR(V,\CC)$ of $\CC$-linear morphisms, since the line bundle associated to $\varphi:V\to\CC$ possesses a holomorphic section given by $v\mapsto e^{2i\pi\varphi(v)}$. This gives the complex $g$-dimensional vector space $\Hom_\RR(V,\CC)/\Hom_\CC(V,\CC)$.
         \item The sublattice $\Hom_\ZZ(L,\ZZ)\subset\Hom_\RR(V,\CC)$, since such elements do not change the value of the character $\chi$.
     \end{itemize}

     \begin{defi}
    The dual torus $A^\vee$ of $A=V/L$ is quotient of the $g$-dimensional vector space $\Hom_\RR(V,\CC)/\Hom_\CC(V,\CC)$ by the projection of the lattice $\Hom_\ZZ(L,\ZZ)$.
    \end{defi}
\
\subsubsection{Torsion line bundle and their  monodromy presentation}     The $n$-torsion elements of the dual torus come from elements $\varphi\in\Hom_\RR(V,\CC)$ such that $n\varphi$ projects to a lattice element, i.e. belongs to $\Hom_\CC(V,\CC)+\Hom_\ZZ(L,\ZZ)$. The $\RR$-subspaces $\Hom_\RR(V,\RR)$ and $\Hom_\CC(V,\CC)$ are direct summands. Therefore, for each such $\varphi$, we have that $\varphi+\Hom_\CC(V,\CC)$ intersects $\Hom_\RR(V,\RR)$ in a unique point. As $n\varphi$ projects to a lattice element, it maps $L$ to $\frac{1}{n}\ZZ$ and thus belongs to $\Hom_\ZZ(L,\frac{1}{n}\ZZ)$. In particular, there is a unique character $\chi=e^{2i\pi\varphi}$ in the equivalence class with values in $\mu_n$, the $n$th-roots of unity. 
     
Therefore, torsion line bundles over $V/L$ are in bijection with $\Hom(L,\mu_n)$, which are elements coming from morphisms $L\to\frac{1}{n}\ZZ$. We say that the line bundle associated to $\chi:L\to\mu_n$ has monodromy $\chi(l)\in\mu_n$ along the loop $l\in L\simeq\pi_1(V/L)$.

  \subsection{Log-abelian varieties}
    \label{sec-log-abelian}

    We refer to \cite{kajiwara2008logabelian2} and references therein for details on log-abelian varieties. We will mostly restrict to the case of constant degeneration.

    \subsubsection{Log-tori}
    We start with the definition of a log-torus over a fs log-scheme $(S,M_S)$ with constant degeneration, i.e. $\bar M_S:=\bar M$ is a constant sheaf of monoids.
    
    We recall the definition of $\Gmlog$ and $\Gmtrop$ the abelian groups on fs log schemes whose value on $(S,M_S)$ is $\rmH^0(S,M_S^\gp)$ and $\rmH^0(S,\bar{M}_S^\gp)=\bar{M}^\gp$ in the constant degeneration case.
    
    These fit into an exact sequence of abelian groups over fine and saturated log schemes:
\[1\to\Gm\to\Gmlog\to\Gmtrop\to 1.\]
    
        Let $G$ be a complex semi-torus over $S$, sitting in the following exact sequence
    $$0\to T\to G\to B\to 0,$$
    where $T=\Hom(X,\Gm)$ is an algebraic torus with character lattice $X$, and $B$ a complex torus. We then consider the push-out with $T_{\log}=\Hom(X,\Gmlog)$:
    \bcd
 0 \arrow[r] & T \arrow[r]\arrow[d]\arrow[dr, phantom, "\lrcorner"] & G \arrow[r]\arrow[d] & B \arrow[r]\arrow[d,equal] & 0 \\
 0 \arrow[r] & T_{\log} \arrow[r] & G_{\log} \arrow[r] & B \arrow[r] & 0.
    \ecd

    In particular, we have that $G_{\log}/G\simeq T_{\log}/T = \Hom(X,\Gmtrop)$. 
    
    Suppose we are given a $Y$ be a second lattice of the same rank as $X$,(sheaf of lattices over $S$ to be precise) and a morphism $u:Y\to G_{\log}$. This data is called s \textit{log-$1$-motif} $[Y\xrightarrow{u}G_{\log}]$. 
     We will denote by $\pi$ the composition 
    \[\pi\colon Y\to G_{\log}\to B.\]
    Composing $u$ with the quotient map to $G_{\log}/G\cong \Hom(X,\Gmtrop)$, we obtain a pairing
    $$\langle-,-\rangle\colon X\otimes Y\to\Gmtrop.$$

    We denote by $ \Hom(X,\Gmtrop)^{(Y)}$ the subgroup of $ \Hom(X,\Gmtrop)$
of elements which are \textit{locally bounded by $Y$}.i.e.:
    \begin{equation}\label{eq:bounded monodromy1}
      \Hom(X,\Gmtrop)^{(Y)} = \{\varphi\; |\;\forall x,\ \exists y_x,y'_x\in Y \text{ with }\langle x,y_x\rangle |\varphi(x)|\langle x,y'_x\rangle\}
    \end{equation}
    where $a|b$ in $\Mbargp$ means that $ba^{-1}\in\overline{M}$.
    
    We then denote by $G^{(Y)}_{\log}$ the pre-image of $\Hom(X,\Gmtrop)^{(Y)} $ in $G_{\log}$ . By construction, $u(Y)$ is contained in the latter. 
    
    We will refer to $G^{(Y)}_{\log}$ as the subgroup of element with \emph{bounded monodromy}. The log torus associated with the data above is defined to be the quotient:

    \[\A:=G^{(Y)}_{\log}\slash Y.\]

   We call \emph{tropicalization} of $\A$ and denote by $\operatorname{Trop}(\A)$ the quotient 
   \[\operatorname{Trop}(\A):=\Hom(X,\Gmtrop)^{(Y)}\slash Y.\]
   There is an exact sequence of pre-sheaves of groups over fs log-schemes over $(S,M_S)$ given by:
   \[0\to G\to\A\to \operatorname{Trop}(\A)\to 0\]
   If we consider a log point $(\Spec k,\mathbb R_{\geq 0})$ with a  morphism to $(S,M_S),$ i.e. a point of $S$ together with a monoid morphism $\bar{M_S}\to \mathbb R_{\geq 0}$, then $\operatorname{Trop}(\A)(\mathbb R)$ is a tropical torus as considered in \cite{mikhalkin2008tropical}.

   \medskip

    A log-abelian variety is a log-torus endowed with a polarization, which is a morphism to its dual log-torus which we now define.

\begin{rem}
    The quotient presentation above is well suited to study log abelian varieties in the so called \emph{constant degeneration case}. In \cite{kajiwara2008logabelian2} the authors also introduce a more general notion of a log torus  $\mathcal A$ over a fs log scheme  $S$. Up to a technical assumption\footnote{Namely that the diagonal $\A\to\A\times\A$ is representable by a finite morphism of log schemes}, $\mathcal A$ is defined to be a  sheaf  of abelian groups on fs log scheme such that, \'etale locally on $S$, there exists  a pair of lattices $X,Y$ and a \emph{non-degenerate} pairing  $\langle,\rangle\colon X\times Y\to \mathbb G_{\rm{trop},S}$
    such that $\mathcal A$ admits on this \'etale chart a quotient presentation as above
$\mathcal A$ is  then said a log abelian variety  over $S$ if, for  any $s\in S$ endowed with the pull-back log structure, $\mathcal  A_s$ is polarizable in the sense explained below.
\end{rem}

    \subsubsection{Dual log-torus} Suppose we are given a log torus as in the previous paragraph. Using the same notations as above, we define:
        \begin{itemize}[leftmargin=0.4cm]
            \item Let $B^\vee=\mathcal Ext^1(B,\mathbb G_m)$ be the dual of $B$ \cite[I.Proposition~9.3]{milneAV} and set
            $$G^\vee=\{(F,h) \text{ with }F\in B^\vee \text{ and }h:Y\to F \text{ s.t. }\mathrm{pr}\circ h=\pi\}$$
            where $\mathrm{pr}\colon F\to B$ is the projection to $B$.            
            In particular, $G^\vee_{\log}$ is the set of pairs $(F_{\log},h)$ where $F_{\log}$ is obtained from $F$ thought as a $\mathbb G_m$ torsor over $B,$ extending $\mathbb G_m$ to $(\mathbb G_m)_{\log}$, and $h\colon Y\to F_{\log} $ satisfies the same condition as before.
            \item Let $T^\vee=\Hom(Y,\mathbb G_m)$, endowed with the following map to $G^\vee$: take $F=B\times\mathbb G_m$ the trivial bundle and for $\chi\in T^\vee$ take $h_\chi(y)=(\pi(y),\chi(y))$.
            \item There is a forgetful map $G^\vee\to B^\vee$ forgetting $h$, so that we have the short exact sequence
            $$0\to T^\vee\to G^\vee\to B^\vee\to 0.$$
            \item There is a natural map $u^\vee:X\to G^\vee_{\log}$ defined as follows: for $x\in X$, define  $F(x)\in B^\vee$ as the push-out 
            
            \bcd
            T=\Hom(X,\mathbb G_m)\ar[r]\ar[d,"\ev_x"] \arrow[dr, phantom, "\lrcorner"]  & G\ar[d] \ar[r] & B \arrow[d,equal] \\
            \mathbb G_m\ar[r] & F(x)  \ar[r] & B 
            \ecd
             and take $h_x$ to be the composition $Y\to G_{\log}\to F(x)_{\log}$.
        \end{itemize}

    Similarly to the definition of $G_{\log}^{(Y)}$, we can consider the subgroup of elements bounded by $X$, denoted by $(G_{\log}^\vee)^{(X)}$, which naturally contains the image of $X$. The dual log-torus is then defined as the quotient:
    \[\A^\vee:=(G_{\log}^\vee)^{(X)}\slash  X.\]

    The tropicalization of the dual log torus is then given by
    \[\operatorname{Trop}(\A^\vee)=\Hom(Y,\Gmtrop)^{(X)}\slash X\]

\subsubsection{The  dual log torus as moduli space of  $\Gmlog$-torsors}\label{sec:logpicofabelian} 
We can view  $\A^\vee$ as parametrizing certain $\Gmlog$-torsors with \emph{bounded monodromy }over $\A$.

\medskip

First we observe that, given the quotient presentation of $\A$, a $\Gmlog$  torsors over $\A$ correspond to a $\Gmlog$  torsors $\mathcal F$ over $G^{(Y)}_{\log}$ together with an action $Y\xrightarrow{h}\mathcal F$ compatible with $Y\xrightarrow{u}G_{\log}.$ 
Consider those $\mathcal F$ obtained with the following construction: take $F$ a  degree zero line bundle on $B$, i.e. a point of $B^\vee$, consider its extension to a log line bundle over $B$, i.e. $F_{\log}=F\otimes_{\mathcal O_B^*}M^{\gp}$ and finally take its pull-back along the map $G_{\log}\to B.$
These are precisely the first part of data of an object in $G_{\log}^\vee$ considered in the definition of $\A^\vee$  recalled above.
Now to get a torsor on $\A$ we need an action of the lattice $Y$ on $\mathcal F$, and  that is precisely the second part of data of an object in   $G^\vee_{\log}.$ 
  
The requirement $Y\xrightarrow{h}F_{\log}\xrightarrow{pr} B$ coincide with $\pi$ imposed in the definition of $G^\vee_{\log}$ is simply imposing the action on $F_{\log}$ and $G_{\log}$ to be compatible; the condition is necessary for the torsor to descend to the log abelian variety.

The bounded monodromy condition is simply a condition on the  induced morphism $h\colon Y\to F_{\log}\slash F\cong\Gmtrop.$

\medskip

Finally we notice that the elements in the image of $X\xrightarrow{u^\vee} B^\vee$ may be considered as trivial torsors since they possess a section over $G_{\log}$: the map $G_{\log}\to F(x)_{\log}$. Hence, $\A^\vee$ is obtained quotienting by $X.$
 
 \medskip

We call a \emph{degree zero log line bundle} on $\A$ to be a pair $(F_{\log},h)$ as before such that $h$ has bounded monodromy.

 \medskip

The above interpretation implies that the dual tropical torus 
$\operatorname{Trop}(\A^\vee)$ is a parameter space of $\Gmtrop$-torsors on $\operatorname{Trop}(\A).$ Indeed, the quotient $F_{\log}\slash F$ is by construction a $\Gmtrop$ torsor over $T_{\log}\slash T$ and the homomorphism $\bar{h}\colon Y\to F_{\log}\slash F$ is the data necessary to descend it to $\operatorname{Trop}(\A).$

  \begin{rem}
        As stated in \cite[Definition 2.7.4]{kajiwara2008logabelian2}, one can show that the dual of the dual is canonically identified with the initial log-torus.
    \end{rem}

\begin{rem}
The  reason why one should  define a  \emph{degree zero log line bundle} on $\A$ as we did above  is, on one  hand, the construction of the dual log torus given in
\cite{kajiwara2008logabelian2}; on the  other hand that, generalizing what we know in the usual abelian variety setting,  the extension in the category of groups over log scheme $\mathcal Ext^1(\A,\Gmlog)$  are given by $G_{\log}^\vee\slash X$ \cite[Theorem~7.4]{kajiwara2008logabelian2}.

For  other applications it might be interesting to consider more general $\Gmlog$ on $\A$: e.g. coming  from line bundles of  degree non zero on $B$ or also defined as a quotient of $Y$ acting on $F_{\log} $
 not necessarily coming from a group homomorphism $h\colon Y\to F_{\log}. $  
 These more general actions  should be defined through a logarithmic analogue of Theta function and should yield log line bundle with non zero the tropical degree ( in the sense of \cite[Section~5]{mikhalkin2008tropical}.)

We  intend to return  to  this analysis in subsequent work.
\end{rem}

\subsubsection{Polarization} Following \cite[Section 2.3]{kajiwara2008logabelian2}, a map between two log-tori $u:Y\to G_{\log}$ and $u':X'\to G'_{\log}$ is a commutative diagram
\bcd
Y \ar[r]\ar[d,"h_{-1}"'] & G_{\log} \ar[d,"h_0"] \\
Y' \ar[r] & G'_{\log}.
\ecd

A \textit{polarization} on $\A=[u:Y\to G_{\log}]$ is a morphism $h:\A\to \A^\vee=[u^\vee:X\to G^\vee_{\log}]$ satisfying the following conditions:
\begin{enumerate}[label=(\alph*)]
    \item the morphism $B\to B^\vee$ induced by $h_0$ is a polarization on $B$;
    \item the morphism $h_{-1}\colon Y\to X$ has finite cokernel;
    \item the pairing $\langle-,h_{-1}(-)\rangle$ on $Y\times Y$ is positive definite in the sense that the value at $(y,y)$ belongs to $\overline{M}\subset\Mbargp$;
    \item the morphism $T_{\log}\to T^\vee_{\log}$ induced by $h_0$ is induced by $Y\to X$.
\end{enumerate}

We prove in the next section that the log-Picard group possesses a natural \textit{principal} polarization, so that it is actually isomorphic to its dual.

    \subsection{Weil pairing for smooth abelian varieties} 
    We review below several ways of thinking about the Weil pairing, including the special case of curves. 

    Let $A[n], A^\vee[n]$ denote the subgroup of $n$-torsion elements of the complex tori $A$ and $A^\vee$ respectively.
 The Weil pairing is a bilinear map
    $$W_n:A[n]\otimes A^\vee[n]\to\mu_n.$$

Depending on the context, we will find convenient to consider the pairing to have values in $\mu_n$, $\ZZ_n$ or the $n$-torsion subgroup of $\RR/\ZZ$, which are all canonically isomorphic.

    \subsubsection{Weil pairing in additive notation} Assume that the complex torus $A$ is written as the quotient of a complex vector space $V$ of complex dimension $g$ by a lattice $L\subset V$ of rank $2g$. In this setting, we saw that the dual torus $A^\vee$ can be seen as $\Hom_\RR(V,\RR)/\Hom_\ZZ(L,\ZZ)$ where $L^\vee=\Hom_\ZZ(L,\ZZ)\hookrightarrow\Hom_\RR(V,\RR)$ is defined by associating to $\lambda:L\to\ZZ$ its unique $\RR$-linear extension to $V$. The respective torsion subgroups of the $A$ and $A^\vee$ have the following identifications:
    $$A[n]=\left(\frac{1}{n}L\right)/L\simeq L\otimes\ZZ_n \text{ and } A^\vee[n]=\left(\frac{1}{n}L^\vee\right)/L^\vee\simeq L^\vee\otimes\ZZ_n.$$
    The pairing \textit{Weil pairing} is  the  one induced by the natural pairing between $L$ and $L^\vee$: if $l\in L$ and $l^\vee\in L^\vee$,  set
    $$W_n\left(\frac{l}{n},\frac{l^\vee}{n}\right) =\frac{\langle l,l^\vee\rangle}{n}\in\RR/\ZZ[n],$$
    whose value does indeed not depend on the representative of the elements mod $L$ and $L^\vee$. 
    Up to taking the exponential, it may also be seen as taking values in $\mu_n\subset\CC^*$.

    \begin{lem}
        The Weil pairing is non-degenerate.
    \end{lem}

    \begin{proof}
        This follows from the non-degeneracy of the pairing between $L$ and $L^\vee$.
    \end{proof}

\subsubsection{Weil pairing and subgroups}    If $K\subset A[n]$ (resp. $A^\vee[n]$) is a subgroup, we denote by $K^\perp\subset A^\vee[n]$ (resp. $A[n]$) its orthogonal via the Weil pairing. Using non-degeneracy, we thus have a natural isomorphism
    $$A^\vee[n]/K^\perp \simeq \widehat{K},$$
    where $\widehat{K}=\Hom(K,\mu_n)$ is the Pontrjagin dual of $K$, isomorphic to $K$ as a group but non canonically. Indeed, consider the restriction to $K$ of the adjoint of the Weil pairing: $\L\mapsto W(-,\L)|_K$; it is easy to see that this map is surjective, and its kernel is by definition $K^\perp$.

    \begin{lem}\label{lem-change-order-torsion}
        If $(x,y)\in A\times A^\vee$ are $n$-torsion elements and $k\in\NN$, they are also $kn$-torsion elements. The relation between the associated pairings is
        $$W_{kn}(x,y)=kW_n(x,y)\in\RR/\ZZ.$$
    \end{lem}

    \begin{proof}
        We may write $x=\frac{l}{n}=\frac{k\cdot l}{kn}$ and $y=\frac{l^\vee}{n}=\frac{k\cdot l^\vee}{kn}$. Therefore,
        $$W_{kn}(x,y)=\frac{\langle k\cdot l,k\cdot l^\vee\rangle}{kn} = k\frac{\langle l,l^\vee\rangle}{n}=k\cdot W_n(x,y)\in\RR/\ZZ.$$
    \end{proof}

    \begin{rem}
    In the particular case where $A$ is the Albanese variety of a complex manifold $X$, which turns out to be also the general case since an abelian variety is its own Albanese variety, the lattices involved are the free part of $\rmH_1(X,\ZZ)$ and $\rmH^1(X,\ZZ)$, endowed with their natural pairing.
    \end{rem}

    It also follows from the definition that the Weil pairing is functorial in sense of the following lemma.
    
    \begin{lem}
    If $f_*\colon A\to B$ is a morphism between two complex tori and $f^*\colon B^\vee\to A^\vee$ its dual map, then we have for $x\in A[n]$ and $y\in B^\vee[n]$,
    $$W_A(x,f^*(y)) = W_B(f_*(x),y).$$
    \end{lem}

    \subsubsection{Weil pairing through line bundles} Using the identification $A^\vee=\Pic^0(A)$  we may provide a second definition of the Weil pairing. Let $\P$ be the unique line bundle on $A\times\Pic^0(A)$ in the sense of \cite[Section 4]{bosch2012neron} such that $\P|_{A\times\{\L\}}\cong\L$ and $\P|_{\{0_A\}\times A^\vee}=\O_{A^\vee}$, where the second equality means that the line bundle has been trivialized over the $0_A$-section. This line bundle is called the Poincar\'e line bundle.

    Let $\P_{x,\L}$ denote the fiber of $\P$ over $(x,\L)\in A\times A^\vee$. We have natural isomorphisms inherited from the compatibility with group operations: if $x,x'\in A$ and $\L,\L'\in A^\vee$,
    $$ \P_{x,\L\otimes\L'}\simeq\P_{x,\L}\otimes\P_{x,\L'} \text{ and } \P_{x+x',\L}\simeq\P_{x,\L}\otimes\P_{x',\L}.$$
    Furthermore, the two isomorphisms
    $$\P_{x+x',\L\otimes\L'}\simeq\P_{x,\L}\otimes\P_{x',\L}\otimes\P_{x,\L'}\otimes\P_{x',\L'},$$
    obtained from the previous ones coincide. 
   In particular, if $x$ and $\L$ are torsion elements, we have natural isomorphisms:
        \begin{itemize}
            \item using the second coordinate $\P_{x,\L}^{\otimes n}\simeq\P_{x,\L^{\otimes n}}=\P_{x,\O}\simeq\CC$,
            \item using the first coordinate $\P_{x,\L}^{\otimes n}\simeq\P_{nx,\L}=\P_{0_A,\L}\simeq\CC$.
        \end{itemize}
    As these two isomorphisms may not agree, their composition defines an automorphism of $\CC$, i.e. multiplication by an element $\lambda$ of $\CC^*$. Using that $\P_{x,\L}^{\otimes n^2}=\P_{nx,\L^{\otimes n}}=\P_{0_A,\O}=\CC$, we have the following commutative diagram
    \bcd
    \P_{0_A,\L}^{\otimes n} \ar[rr,"\lambda^n"]\ar[rd] & & \P_{x,\O}^{\otimes n}. \ar[dl] \\
     & \P_{0_A,\O} & 
    \ecd
    We conclude that this automorphism is actually multiplication by a root of unity $\lambda\in\mu_n$. The Weil pairing $W(x,\L)$ evaluated at the $n$-torsion elements $(x,\L)\in A[n]\times A^\vee[n]$ is defined to be this root of unity $\lambda.$

    \subsubsection{Weil pairing for Jacobian of curves} We end with the special case of the Jacobian of curves, which was historically the case where Weil introduced his pairing. If $C$ is a smooth complex curve, its Jacobian $\operatorname{Jac}(C)=\Pic^0(C)$ is endowed with a natural principal polarization defining an isomorphism $\Alb(C)\cong\Pic^0(C)$. In this  case the Weil pairing can be presented as a pairing between the torsion-elements of the \textit{same} group. The pairing is then skew-symmetric.

    The Weil pairing revolves around the \textit{Weil's reciprocity law}, which is stated as follows. If $f,g:C\to\CC$ are meromorphic functions. We can consider the local symbol at $p\in C$:
    $$(f,g)_p = (-1)^{\nu_p(f)\nu_p(g)}\left.\frac{f^{\nu_p(g)}}{g^{\nu_p(f)}}\right|_p,$$
    where $\nu_p(f)\in\ZZ$ is the vanishing order of $f$ at $p\in C$. If $p$ is neither a pole or a zero of $f$ or $g$, then $(f,g)_p=1$. In particular, only a finite number of points have a non-trivial symbol. The local symbols are of course multiplicative in each variable.
    
    \begin{theo}\cite[Chapter 2.3]{griffiths2014principles}
        If $f,g$ are meromorphic functions on $C$, we have
        $$\prod_p (f,g)_p=1.$$
    \end{theo}

    Now, we can use the law to define the Weil pairing. Assume $D_1$ and $D_2$ are two $n$-torsion divisors. Let $f_1,f_2:C\to\CC$ be meromorphic functions with respective divisors $nD_1$ and $nD_2$. In particular, the order of $f_1$ and $f_2$ at every point is divisible by $n$. These functions are unique up to multiplication by a constant. Weil's reciprocity law ensures that the following quantity, which does not depend on the actual choice of $f_1$ and $f_2$, is actually a root of unity:
    $$W(D_1,D_2)=\prod_p (-1)^{\nu_p(f_1)\nu_p(f_2)/n}\left.\frac{f_1^{\nu_p(f_2)/n}}{f_2^{\nu_p(f_1)/n}}\right|_p\in\mu_n\subset\CC^*.$$
    Weil's reciprocity law also ensures that the value does not change when adding a principal divisor to $D_1$ or $D_2$. Therefore, we get a well-defined pairing on the torsion elements of the Jacobian.

    \begin{rem}
        An extension of the Weil reciprocity law for sections of $\mathbb{G}_{\log}$ has been proved in \cite[Proposition~4.6]{BCK}. This allow a more explicit definition of log Weil Pairing for the logarithmic Picard group of log smooth curves.
    \end{rem}

    \subsubsection{Weil pairing and monodromy of torsion line bundles} \label{sec-weil-pairing-line-bundles} 
    
Given $X$ a projective variety over $\mathbb C$, we have a 
     short exact sequence of sheaves of groups on $X$
    $$0\to\mu_n\to\O^\times\xrightarrow{\cdot^n}\O^\times\to 0,$$
    inducing the following long exact sequence in cohomology:
    $$0\to \rmH^1(X,\mu_n)\to \rmH^1(X,\O^\times)\to \rmH^1(X,\O^\times).$$
   
   From the latter we see that isomorphisms classes of $n$-torsion line bundles are naturally identified with $\mu_n$-principal bundles on $X$. In turn, a principle $\mu_n$-bundle is characterized by its monodromy, namely it is determined by a morphism $\pi_1(X)\to\mu_n$. Such morphism factors through $\rmH_1(X,\ZZ)$ since $\mu_n$ is abelian, and even $\rmH_1(X,\ZZ_n)$ given that $\mu_n$ is of $n$-torsion.

   When $X$ is assumed to be smooth $\Pic^0(X)$ is an abelian variety with dual
   $Alb(X)$ and every degree $0$ line bundle on $X$ is the pull-back of a degree $0$ line bundle on $\Alb(X)$ via $X\to\Alb(X).$    
    In this case we can describe the monodromy representative of a torsion line bundle using the Weil pairing:
    \begin{prop}
        Let $\L$ be a degree $0$ $n$-torsion line bundle on $X$ coming from a line bundle on $\Alb(X)$, also denoted by $\L$. If $\gamma\in\pi_1(X)$ is a loop and $[\gamma]\in \rmH_1(X,\ZZ_n)$ the class it realizes in homology, the monodromy of the the $\mu_n$-torsor associated to $\L$ along the loop $\gamma$ is $W_n([\gamma],\L)\in\mu_n$.
    \end{prop}

    \begin{rem}
        The above formula is compatible with Lemma \ref{lem-change-order-torsion} since choosing a bigger power annihilating $\L$ also changes the order of $[\gamma]$ in $\rmH_1(X,\ZZ_n)$.
    \end{rem}

    \begin{proof}
        We unravel the definition of the dual torus. Let $l^\vee\in\Hom_\ZZ(L,\ZZ)$ be an element in the dual lattice and consider $\frac{l^\vee}{n}\in\Hom_\RR(V,\RR)\subset\Hom_\RR(V,\CC)$, which induces a complex line bundle on $V/L$. By definition of the dual torus in Section \ref{sec-dual-torus-abelian-var}, the identification between the fibers over $v$ and $v+l$ for $l\in L$ is via $\chi(l)=e^{2i\pi\left(\frac{l^\vee}{n}\right)(l)}$, which is precisely the Weil pairing between the torsion element $\frac{l^\vee}{n}$ and the image of the class of $l$ in $\rmH_1(V/L,\ZZ_n)$.
    \end{proof}

    \subsection{Weil pairing for log-abelian varieties}

    We now propose an extension of the definition of Weil pairing to the case of log-abelian varieties with constant degeneration. We expect that there is a natural way to define the Weil pairing  for general log-abelian schemes.
    

    \subsubsection{Definition through Poincar\'e bundle}\label{sec:logPoincare'} We aim to define the Weil pairing for log-tori using the line bundle approach. To do that, we define a natural $\Gmlog$-torsors on $\A\times\A^\vee,$ which we call log-Poincar\'e bundle.

    With the notations from Section \ref{sec-log-abelian}, the product $\A\times\A^\vee$ is constructed from the semi-abelian variety
    $$0\to T\times T^\vee\to G\times G^\vee\to B\times B^\vee \to 0,$$
    using the map $(u,u^\vee):Y\times X\to G_{\log}\times G^\vee_{\log}$.

    Given the point of view on $\Gmlog$-torsors on log tori discussed in Section~\ref{sec:logpicofabelian} to construct the 
    the log Poincar\'e bundle, we start with the classical Poincar\'e  line bundle $\P$ on $B\times B^\vee$ and then describe the monodromy  action $h\colon Y\times X\to \P_{\log}$ for its pull-back to $G_{\log}\times G^{\vee}_{\log}$. 
    By definition, $\P$ restricts to trivial bundles over the zero sections of $B$ and $B^\vee$. We shall now find a map $h:Y\times X\to\P_{\log}$ such that the composition with the projection is $(\pi,\pi^\vee)$. This means that $h(y,x)$ should belong to \[(\P_{\pi(y),\pi^\vee(x)})_{\log} = (F(x)_{\pi(y)})_{\log}=(F^\vee(y)_{\pi(x^\vee)})_{\log}\]
    where we are denoting by $F(x)$ (respectively $F^\vee(y)$) the line bundle on $B$ (on $B^\vee$) corresponding  to $\pi^\vee(x)\in B^\vee$ (to $\pi_(y)\in B$) and $F(x)_{\log}$ is the $M^\gp$-torsor 
 on $G_{\log}$ ( on $G^\vee_{\log}$ ) obtained  first extending $\Gm\to\Gmlog$
and then  pulling-back.
     
     It follows from the definition of $G^\vee$ that  $F(x)_{\log}$ comes equipped with a morphism $h_x:Y\to F(x)_{\log}$ defining the monodromy action and thus inducing a $M^\gp$-torsor over $G_{\log}^{(Y)}/Y$, but also with a section $G_{\log}\to F(x)_{\log}$. We thus consider the following function:
    $$h(y,x)=h_x(y)\in F(x)_{\log}\rvert_{\pi(y)}.$$
    The latter indeed projects to $\pi(y)\in B$ thanks to the assumption on $h_x$, and the fact that $F(x)=\P|_{B\times\{\pi^\vee(x)\}}$. 
    The same reasoning apply exchanging the role of  of $B$ and  $B^\vee$ and it is not difficult to convince oneself that $h(y,x)=h_y^\vee(x)$ where $h_y^\vee\colon X\to F^\vee(y)$ is the morphis naturally obtained identifying $G$ with $(G^\vee)^\vee.$

\medskip

    We can now proceed as in the classical setting: torsion elements of the log-torus provide torsion log-line bundles over the dual. Using the Poincar\'e bundle, we get two trivializations of the fiber over the pair of torsion elements. These two isomorphisms differ by a $n$-torsion element, which is necessarily a root of unity, finishing the definition of the log-Weil pairing.

    In the case of LogPic, the principal polarization identifies $\A$ with its dual, so that we may consider the log-Weil pairing as a (skew-symmetric) pairing on $n$-torsion elements of $\A$.

    \subsubsection{A few properties} Let us consider a log-torus with the previous notations. First, applying $n$-torsion to the short exact sequence defining the semi-torus $G$ yields the following, with of course $G[n]=G_{\log}[n]$ since the monoid elements do not add any more torsion:
    $$0\to T[n]=\Hom(X,\mu_n)\to G[n] \to B[n] \to 0.$$
    We have a second short exact sequence:
    $$0\to G \to \A \to \mathrm{Trop}\ \A\to 0,$$
    where $\mathrm{Trop}\ \A=\Hom(X,\Mbargp)^{(Y)}/Y$ is the tropicalization of $\A$. After a base change of monoid, the $n$-torsion of the latter is isomorphic to $Y\otimes\ZZ_n$. Thus, we get
    $$0\to G[n] \to \A[n] \to Y\otimes\ZZ_n \to 0.$$
    In particular, $G[n]$ is a natural subgroup of $\A[n]$, and $T[n]\subset G[n]$ as well.

    \begin{prop}\label{prop:logweil}
        We have the following:
        \begin{enumerate}[label=(\arabic*)]
            \item The restriction of the log-Weil pairing between $G[n]$ (resp. $G^\vee[n]$) and $T^\vee[n]$ (resp. $T[n]$) is trivial. 
            \item The induced quotient pairing between $B[n]$ and $B^\vee[n]$ is the Weil pairing of $B$.
            \item The induced pairing between $T[n]=\Hom(X,\mu_n)$ and $\mathrm{Trop}\ \A^\vee[n]=X\otimes\ZZ_n$ is the natural pairing, similarly for the duals.
            \item The log-Weil pairing is non-degenerate.
        \end{enumerate}
    \end{prop}

    \begin{proof}
    \begin{enumerate}[label=(\arabic*)]
        \item We need to prove that the pairing between elements of $G[n]$ and $T^\vee[n]$ is $0$. For such elements, since $G$ is a subgroup of $\A$ and $T^\vee$ a subgroup of $\A^\vee$, the quotients by $Y$ and $X$ plays no role. the fiber of $\P_{\log}$ is simply the pull-back of the fiber of $\P$. As $T^\vee$ projects to $0\in B^\vee$, the pairing is $0$.
        \item In particular, the pairing between elements in $G[n]$ and $G^\vee[n]$ only depends on their projection in $B[n]$ and $B^\vee[n]$, inducing a quotient pairing. The same argument as before shows that the pairing between elements of $G[n]$ and $G^\vee[n]$ is the pairing between their projections in $B[n]$ and $B^\vee[n]$.
        \item In particular, we have a well-defined pairing between the elements of $T[n]=\Hom(X,\mu_n)$ (resp. $T^\vee[n]=\Hom(Y,\mu_n)$) and the tropical torsion of the dual $\mathrm{Trop}\ \A^\vee[n]=X\otimes\ZZ_n$ (resp. $Y\otimes\ZZ_n$). We claim this pairing is the natural pairing. To see that, let $(\O,h)\in T^\vee[n]$ be a torsion element, represented by $h:Y\to\mu_n$, and let $g\in G_{\log}$ with $g^n=u(y)$ for some $y\in Y$. We then gather the two isomorphisms:
            \begin{itemize}
                \item $\P_{g,\O}^{\otimes n}\simeq \P_{g,\O}$ since $\O^{\otimes n}=\O$, identified with $\CC^*$ given it's the trivial bundle over $G$;
                \item  $\P_{g,\O}^{\otimes n} \simeq \P_{g^n,\O}=\P_{u(y),\O}$, which is identified with $\CC$ via multiplication by $h(y)$, which is the expected result.
            \end{itemize}
        Therefore, the pairing is as expected.
        \item The pairing is non-degenerate: assume an element of $\A^\vee[n]$ is paired trivially with every other element. Then,
            \begin{itemize}
                \item It is paired trivially with every element of $T[n]$ so that its class in $\mathrm{Trop}\ \A^\vee[n]$ is $0$ and the element belongs to $G^\vee[n]$.
                \item It is paired trivially with every element of $G[n]$ so that its class in $B^\vee[n]$ is $0$ and it belongs to $T^\vee[n]$. The latter is paired trivially with every element of $\mathrm{Trop}\ \A[n]$ so that it is in fact $0$.
            \end{itemize}
    \end{enumerate}
    \end{proof}

    \begin{rem}
        This definition of log-Weil pairing suggests there is no tropical analog of the Weil pairing, since we rather have a pairing between the tropicalization of $\A$, which is a tropical abelian variety, and the fiber of tropicalization map for the dual. There is a priori no natural pairing between $\mathrm{Trop}\ \A[n]$ and $\mathrm{Trop}\ \A^\vee[n]$.
    \end{rem}

%% file: sec-log-pic.tex
\section{Log-Picard group  of curves}
\label{sec-log-picard-group}

\subsection{Jacobians for smooth curves}
One of the most studied and best understood example of abelian scheme is the relative Jacobian of a family $C\to S$ of smooth, projective curve.
We recall the some known results about $\Pic^0(C\slash S)$ and refer the reader for example to \cite[Chapter~II]{milneAV}.

Let us assume that $C\to S$ is a family of smooth, projective genus $g$ curve and $\sigma_0\colon S\to C$ a section.

\begin{itemize}
    \item  $\Pic^0(C\slash S)\to S$ is a proper, smooth family of rank $g$-dimensional abelian varieties. Furthermore $\Pic^0(C\slash S)$ is a fine moduli space parametrizing line bundles on $C$ of degree zero along the fibers.
    \item The Abel-Jacobi map is the morphism  $AJ_{\sigma_0}:C\to\Pic^0(C/S)$ given by:
\[\begin{array}{rccl}
    AJ_{\mathcal \sigma_0}\colon & C &\longrightarrow & \Pic^0(C/S)\\
    & (T\to C) & \longmapsto & \mathcal O_{C_T}(p_T -\sigma_0\rvert_T).
\end{array}\]
where $p_T:T\to C\times_S T$ is the induced section of $ C_T=\C\times_S T.$
\item The pull-back along the Abel-Jacobi map 
\[    AJ_{\mathcal \sigma_0}^*\colon \Pic^0(C/S)^\vee\to\Pic^0(C/S)\]
identify $\Pic^0(C/S)$ with its dual, in the sense of Abelian schemes; the reader can find a proof also in \cite[9.5.26]{fantechi2006fundamental}. In other words, the Jacobian of a smooth curve is principally polarized. A proof of the self duality isomorphism by exhibiting an explicit principal polarization can be found for example in \cite[Chapter~III]{milneAV}

\item The self duality isomorphism above in particular shows that $(AJ_{\mathcal \sigma_0},\Pic^0(C/S))$ is initial for morphism  $C\xrightarrow{\varphi} A$  from $C$ to an abelian scheme over $S$ such that $\varphi(\sigma_0)=0_A,$ i.e. $\Pic^0(C/S)$ is the Albanese variety $\operatorname{Alb}(C/S)$ of $C/S.$
\end{itemize}
The last two properties often go under the name \emph{Abel-Jacobi Theorem} for smooth projective curves.
In particular, for any morphism $F\colon C/S\to X/S$ one gets a natural morphism
\[F_*\colon\Pic^0(C/S)\to\operatorname {Alb}(X/S)\]
coming, via the self duality isomorphisms, from the universal property of the Albanese variety which satisfies $F_*\circ AJ_{\sigma_0}=a_{F(\sigma_0)}\circ F$ where $a_{F(\sigma_0)}\colon X/S\to \operatorname {Alb}(X/S)$ is the Albanese map for $X.$
\begin{rem}
We notice here that it is true more generally (e.g. also for  any $X\to S$ smooth, projective with a section $\sigma$) that $\Pic^0(X/S)\to S$ is an abelian scheme over $S$ representing the moduli functor of line bundle on $X\to S$ of degree zero along the fibers and that $\Pic^0(X/S)^\vee\cong \operatorname{Alb}(X/S).$
This last statement is almost tautological once we think of an abelian variety as fine moduli space of degree zero line bundle on its dual. A detailed proof can be found in \cite[Appendix]{mochizuki2012topics}.
The identification $\Pic^0(X/S)^\vee\cong \operatorname{Alb}(X/S)$ implies that the homorphism $F_*$ described above is nothing but the dual homomorphism in the category of abelian varieties of the pull-back $F^*.$
\end{rem}

To carry out our study of correlated GW invariants we need to extend the definition of the morphism $F_*$  to  families $C\to S$ of pre-stable curves. 
We still have a semi-abelian scheme (i.e. smooth, but with fibers which are only quasi projective) $\Pic^{[0]}(C/S)\to S$ whose fiber over $s\in S$ parametrizes multidegree zero line bundles on $C_s.$ The fibers of $\Pic^{[0]}(C_s)$ are extension of an abelian variety, namely the Jacobian of the normalization of the curve, and an algebraic torus of rank the genus of the dual graph of $C_s.$ Notice that as the curve vary in the family, the rank of the abelian part and the torus part of the extension presentation jump.

\begin{rem}
   In various situations, existence of an Abel-Jacobi map and self-duality can be saved. Indeed, working with \emph{fine compactified Jacobian} \cite{melo2015compactifications,MeloRapagnettaViviani,kass2019stability,fava2025new,fava2024complete,fava2026complete} one recovers a self duality statement for the Jacobian, and as long as we look at a family of integral curves there is a natural extension of the Abel-Jacobi map with target the compactified Jacobian.

   However, these partial extensions are not sufficient to deal with the universal curve over moduli of maps.
\end{rem}

The solution to these problems comes from working with logarithmic abelian varieties. 
We start by recalling some results about the logarithmic Picard group of a log smooth vertical curve $(C,M_C)\to (S, M_S)$ constructed in \cite{molchowiselog}. The \emph{logarithmic Jacobian} will turn out to be the right replacement for the Jacobian of a smooth curve in the case of nodal curves.

\subsection{Logarithmic, tropical curves  and their sheaves of functions}
We assume the reader to be already familiar with logarithmic curves \cite{KatoF}, and more in general with the notion of tropicalization of a fine and saturated logarithmic scheme as well as the theory of Artin fans. We refer for example to \cite{abramovich2014comparison,cavalieri2020moduli} and references therein for an introduction.

\medskip

To fix notation: we denote by $(C,M_C)\to (S,M_S)$ be a proper, \emph{vertical} (no marking carrying log structure), logarithmic smooth curve  and  $\Gamma\to (S,M_S)$ be the associated tropical curve over $S;$ see for example \cite[Section 2.4]{molchowiselog} for more details on tropical curves over log schemes.
The latter is the data of: for each  geometric point $s\in S$ the metrization of $\Gamma_s$ (the dual complex of  $C_s$) with values in $\bar{M}_{S,s}$ induced by the logarithmic structure; for each geometric specialization $t\rightsquigarrow  s$ a edge contraction $\Gamma_s\to\Gamma_t$ where the contracted edges correspond to the  nodes of $C_s$ smoothed in $C_t$ and  are precisely the edges whose length $\ell(e)\in \bar{M}_{S,s}$ goes to zero under the morphism $\bar{M}_{S,s}\to \bar{M}_{S,t}.$

\medskip

When the underlying scheme of $(S,M_S)$ is a point (or more generally for a family of log curves with constant degeneration) there is an identification of $\rmH^0(C,\bar{M}_C^\gp)$ and \textit{strict piece-wise} linear functions (sometimes also called \textit{cone-wise} linear functions). These are functions $\alpha\colon\Gamma\to \bar{M}_S^\gp$ which are linear with integral slope along the edges, i.e. $\frac{\alpha(v)-\alpha(w)}{\ell(e)}\in\mathbb Z$ for $e$ the edge between $v$ and $w.$

\medskip

We now endow $\Gamma$ with the \emph{tropical topology} (as defined in \cite[Section~3]{molchowiselog}) whose fundamental opens are the stars of the vertices and two such opens intersect  in a disjoint union of open intervals. We can consider sheaves in this topology; for example the sheaf $\mathcal{PL}$ of strict piecewise linear functions is the sheaf whose sections over an open set $U$ are the data of: a function $\alpha$ on the vertices of $U$ with values in $\bar{M}_S^{\gp}$  a collection $(w_e)$ of slopes for edges in $U$ such that $\frac{\alpha(v)-\alpha(w)}{\ell(e)}=w_e$ if $v,w,e\in U$. We also consider the following two sub-sheaves:
    \begin{itemize}[label=$\circ$,leftmargin=0.6cm]
        \item the constant sheaf $\bar{M}_{S,s}^\gp$ on $\Gamma_s$,
        \item the subsheaf $\mathcal L\hookrightarrow\mathcal P\mathcal L$ of 
\emph{balanced} strict piece-wise linear function, i.e the sum of the outgoing slopes at each vertex is zero.
    \end{itemize}

The sheaf $\L$ should be thought as the tropical analogue of invertible regular function on $\Gamma,$ while $\mathcal{PL}$ should be thought as the analogue of the rational function.
These two sheaves of functions were already considered by \cite{mikhalkin2008tropical} in their construction of the tropical Jacobian.

These sheaves fit together (see \cite[Section~3.4]{molchowiselog} ) in the following commutative diagram:

\begin{equation}\label{eq:tropicalsheaves}
    \begin{tikzcd}
& & 0\ar[d] & 0\ar[d]\\
0\ar[r] & \bar{M}_{S,s}^{\gp}\ar[d,"="]\ar[r]&\mathcal L\ar[d]\ar[r] &\mathcal H\ar[d]\ar[r] & 0\\
0\ar[r] & \bar{M}_{S,s}^{\gp}\ar[r]&\mathcal P\mathcal L\ar[d]\ar[r] &\mathcal E\ar[d]\ar[r] & 0\\
& & \mathcal V\ar[r,"="] \ar[d]& \mathcal V\ar[d]\\
& & 0 & 0
\end{tikzcd}
\end{equation}
where $\mathcal E$ and $\mathcal V$ are the constant groups sheaves freely generated by the edges and the vertices respectively and the map between them is the homology boundary map.
 These sequences induce log exact sequence in cohomology, meaning Čech cohomology with respect to an acyclic cover in the tropical topology (take for example the cover given by finitely many stars at the vertices).
All the following results about the cohomologies of these sheaves can be found in \cite[Section~3.4]{molchowiselog}. 
\begin{itemize}[leftmargin=0.6cm]
\item For  $\mathcal{PL}$ we have, simply unraveling definitions:
\[\rmH^0(\Gamma,\mathcal P\mathcal L)=\rmH^0(C,\bar{M}_C^\gp),\qquad \rmH^1(\Gamma,\mathcal P\mathcal L)=\rmH^1(C,\bar{M}_C^\gp).\]
\item  $\rmH^0(\Gamma,\mathcal V)$ is the set of divisors supported on the vertices of $\Gamma$. We also refer to these as \emph{tropical divisors} \footnote{For classical tropical curves, tropical divisors can have support on vertices. We consider these as being supported on a subdivision of the curve.} supported on $\Gamma.$ 
\item $\rmH^i(\Gamma,\mathcal V)$ and $\rmH^i(\Gamma,\mathcal E)$ are zero for $i>0;$ 
\item  $\rmH^0(\Gamma,\mathcal H)$ is identified with the first homology group $\rmH_1(\Gamma,\mathbb Z)$ of the graph underlying $\Gamma$ and  $\rmH^0(\Gamma,\mathcal H)$ with $\rmH_0(\Gamma,\mathbb Z);$
\item there is a natural isomorphsims 
\[\rmH^1(\Gamma,\bar{M}_{S,s}^{\gp})\cong\operatorname{Hom}(\rmH_1(\Gamma,\mathbb Z),\bar{M}_{S,s}^{\gp}).\]
\end{itemize}
If $\Gamma\xrightarrow{\pi} S$ is a tropical curve over a general fs logarithmic scheme, we can still define these sheaves and the maps between them by  considering the definitions we just gave on geometric points and imposing compatibility along the maps induced by specialization $t\rightsquigarrow s.$

\subsection{Tropical Jacobian and tropical Picard group}

\begin{defi}\cite{molchowiselog}
An element $f\in \operatorname{Hom}(\rmH_1(\Gamma,\mathbb Z),\bar{M}_{S,s}^\gp)$ has \emph{bounded monodromy} if for each cycle $\gamma\in\Gamma,$  $f(\gamma)$ is bounded by the length of $\gamma$, meaning, there exists two integer $m,n$ such that $m\ell(\gamma)\leq f(\gamma)\leq n\ell(\gamma)$ in the partial order on $\bar{M}_{S,s}^\gp$ induced by $\bar{M}_{S,s}.$
\end{defi}

A $\mathcal P\mathcal L$-torsor $Q\in \rmH^1(\Gamma,\mathcal P\mathcal L)$ on $\Gamma$ has bounded monodromy if a lift $\tilde{Q}\in \operatorname{Hom}(\rmH_1(\Gamma,\mathbb Z),\bar{M}_{S,s}^\gp)$ of $Q$ has bounded monodromy (and consequently all lifts have bounded monodromy). 

\begin{defi}
If we fix an orientation of the edges of $\Gamma,$ then there is a natural pairing, called the \emph{intersection pairing}:
\[\rmH_1(\Gamma,\mathbb Z)\times \rmH_1(\Gamma,\mathbb Z)\xrightarrow{\cdot} \bar{M}_{S,s}^\gp.\]
The latter is defined by restricting to $\rmH_1(\Gamma,\ZZ)$ the following pairing on the free abelian group generated by edges:
\[e\cdot f=\begin{cases} \ell(e)\;\; \text{if }\;\; e=f\\
-\ell(e)\;\; \text{if }\;\; e=-f\\
0\;\;\text{otherwise.}
\end{cases}\]
\end{defi}

The definition of bounded monodromy given above then coincides with the one of \eqref{eq:bounded monodromy1} from \cite{kajiwara2008logabelian2} taking $X=Y=\rmH_1(\Gamma,\mathbb Z)$ and $\langle-,-\rangle$ to be the intersection pairing.

\begin{defi}
Let $\Gamma\to S$ be a tropical curve.
\begin{itemize}[label=$\circ$,leftmargin=0.6cm]
    \item The {\bf{tropical Jacobian }}is defined as:
\[\TroJac(\Gamma)=\operatorname{Hom}(\rmH_1(\Gamma,\mathbb Z),\bar{M}_{S,s}^\gp)^\dagger\slash \rm{H}_1(\Gamma,\mathbb Z)\cong\operatorname{Ker}(\rmH^1(\Gamma,\mathcal L)^\dagger\to\rm{H}_0(\Gamma,\mathbb Z)),\]
where $\dagger$ denotes the subgroup of elements with bounded monodromy.
    \item The {\bf{tropical Picard group}} is defined as:
\[\TroPic(\Gamma)=\rmH^1(\Gamma,\mathcal L)^\dagger.\]
\end{itemize}

\end{defi}
Notice that the tropical Jacobian of a tropical curve $\Gamma\to(S,M_S)$ over a fs log scheme is a sheaf of groups on fs log schemes. Its points over $(\Spec k,\mathbb R_{\geqslant 0})$ recover the tropical Jacobian in the sense of \cite{mikhalkin2008tropical}.

\begin{rem}
 Molcho-Wise give two useful insights about the necessity of imposing the bounded monodromy condition.
First, by  \cite[Proposition 3.5.1]{molchowiselog}, a $\mathcal L$-torsor $Q\in \rm{H}^1(\Gamma,\mathcal L)$ on $\Gamma$ has bounded monodromy if and only if, up to enlarging the base monoid $\bar{M}_{S,s}\to\widetilde{\bar{M}}_{S,s}$ (e.g. adding some roots) and passing to a subdivision $\widetilde\Gamma\xrightarrow{b}\Gamma$, $b^*Q$ can be trivialized.
They show (it can be seen chasing the long exact sequnces in cohomology induced by Diagram~\eqref{eq:tropicalsheaves} )that this in particular implies that there exists a divisor $D$ supported on the vertices of  $\widetilde{\Gamma}$ representing  $b^*Q$.

Secondly, as explained in \cite[Section~3.6, Proposition~3.7.5,Example~3.6.4]{molchowiselog} the bounded monodromy condition is necessary to ensure that $\TroJac(\Gamma)$ defines a functor, and thus a \emph{tropical moduli problem} on the category of monoids $P$ receiving a map from $\bar{M}_{S,s}.$ 
\end{rem}

Choosing a divisor representative of $Q\in \rmH^1(\Gamma,\mathcal L),$ possibly on a subdivision of the tropical curve, we can think of the map $\rmH^1(\Gamma,\mathcal L)\to \rmH_0(\Gamma,\mathbb Z)$, coming from the connecting homorphism in cohomology for the first row of Diagram~\eqref{eq:tropicalsheaves}, as the \emph{total degree map}.

This is well-defined since two divisors representing $Q$ differ by a strict piece-wise linear function and the total degree of the divisor associated to such a function is zero.

In conclusion, we can identify the tropical Jacobian as the subgroup of $\TroPic(\Gamma)$ of total degree $0$ $\mathcal L$-torsors on the tropical curve. We will sometimes refer to $\mathcal L$-torsors as \emph{tropical line bundles}.
This is the name given to these objects in \cite{mikhalkin2008tropical}. Their descriptions of tropical Picard and tropical Jacobian coincide with the one given in this section if one specializes to the case $\bar{M}_{S,s}=\mathbb R_{\geqslant 0}.$
In particular in  \cite{mikhalkin2008tropical} the bounded monodromy condition is not mentioned as it is automatically satisfied for classical tropical curves.

\subsection{Logarithmic Picard group}

Let $(C,M_C)\to(S,M_S)$ be a proper, vertical, log smooth curve over a log point (i.e. the underlying scheme of $S$ is a point).
Taking the long exact sequence associated to 
\[
0\to\mathcal O_C^{\times}\to M_C^\gp\to\overline{M}_C^\gp\to 0,\]
we obtain
\begin{align}\label{eq:exact}
    \cdots\to\rm{H}^0(C,\overline{M}_C^\gp)\to
    \rmH^1(C,\mathcal O_C^{\times})\to
    \rmH^1(C,M_C^\gp)\to\rm{H}^1(C,\overline{M}_C^\gp).
\end{align}
The $\rmH^1(C,M_C^\gp)$ parametrizes isomorphism classes of $\Gmlog$-torsors on $C.$
Given such a torsor $P$, we denote by $\bar{P}$ the associated $\Gmtrop$-torsor. These induce a $\mathcal{PL}$-torsor on the tropicalization $\Gamma$ of the curve $C$. We say that $P$ has bounded monodromy if $\bar P$ does, where the latter is identified with an isomorphism class of $\mathcal P\mathcal L$-torsor on the tropical curve: $\rmH^1(C,\overline{M}_C^\gp)=\rmH^1(\Gamma,\mathcal P\mathcal L).$

\begin{defi}\cite{molchowiselog}
    Let $(C,M_C)\xrightarrow{\pi}(S,M_S)$ be a proper, vertical, log smooth curve. A \emph{logarithmic line bundle}
on $C$ is a $\Gmlog$-torsor on $C$ in the strict \'etale topology whose fibers over $S$ have bounded
monodromy. Then the logarithmic Picard group is the sheaf of  abelian groups over log schemes parametrizing isomorphism classes of logarithmic line bundles. i.e:
\[\LogPic(C\slash S)= (R^1\pi_* M_C^{\gp})^\dagger,\] 
the subsheaf of $R^1\pi_* M_C^{\gp}$  where bounded monoodromy is satisfied.
\end{defi}

\begin{rem}
    As explained in \cite[Section~4.3]{molchowiselog}, for $(C,M_C)\to(S,M_S)$ a proper, vertical, log smoooth curve, a  $M_C^\gp$ -torsor $P$ has bounded monodromy along the fibers if and only if there is a logarithmic modification (log blowups and root construction) $\widetilde{S}\to S$ possibly  followed by a subdivision
    $\widetilde{C}\to C\times_{S}\widetilde{S} $ such that the pull-back $\widetilde{P}$  on $\widetilde{C}$ is represented by a line bundle, i.e. $\widetilde{P}$ is in the image  of

\[\Pic(\widetilde{C}\slash\widetilde{S})=\rmH^0(\widetilde{S},R^1\pi_*\mathcal O_{\widetilde{C}}^{\times})\to \LogPic(\widetilde{C}\slash\widetilde{S})\subseteq \rmH^0(\widetilde{S},R^1\pi_*M_C^\gp). \]
\end{rem}

Let $(C,M_C)\to(S,M_S)$ be a curve, with $S$ a point. Given $P$ a log line bundle, by the remark above there is a line bundle $L\in\Pic(\widetilde{C})$ for $\widetilde{C}$ over some enlargement $M_{\widetilde{S}}$ (we will refer to these as  \emph{models} of $C$)  representing $P$. Then the quantity $\mathrm{totdeg}(L)$does not depend on the above choice. Indeed, we see from the usual exact sequence \eqref{eq:exact} that given $L_1$ and $L_2$ on some model of $C$ representing $P$, we have $L_1=L_2\otimes\mathcal O_{\widetilde{C}}(\alpha)$ for some  strict piece-wise linear function $\alpha$. Since the total  degree of  $\mathcal O_{\widetilde{C}}(\alpha)$ is zero, we can conclude.

\begin{defi}\label{defi:degreeloglb}
Given $(C,M_C)\to(S,M_S)$ with  $S$ a point, there is a well-defined degree map 
\[\deg\colon\LogPic(C)\to\ZZ.\]
\end{defi}

Furthermore, for $(C,M_C)\to (S,M_S)$ over a  more general logarithmic  scheme the degree is locally constant as proved in \cite[Proposition~4.5.2]{molchowiselog}.

\subsection{Log abelian variety structure}

\begin{prop}\cite[Corollary~4.6.3]{molchowiselog}
   Let $\LogPic^0(C\slash S)$ be the connected component  of degree $0$ logarithmic line bundles. Then this is a logarithmic  torus over $S$ (in fact a logarithmic abelian variety over $S$) in the sense of \cite{kajiwara2008logabelian2}.
\end{prop}
This means  that  over a point $S=\Spec k$ (and in fact more in general for $(C,M_C)\to(S,M_S)$ a proper vertical, log smooth curve with constant degeneration), we can present $\LogPic^0(C\slash S)$ as the quotient of the logarithmic extension of a semi-abelian variety (a semi-abelian log scheme over $S$)  by a lattice.

\medskip

While we refer the reader to Molcho-Wise for the proof, let us give an idea of the  construction.

Let $\Gamma$ be the dual  graph of the curve $C\to S$; working with the constant degeneration case implies that $\Gamma_s$ is constant.
Moreover, in this case $C\to S$ admits a simultaneous normalization
$\prod_{v\in V(\Gamma)} C_v\xrightarrow{\nu} C$ and the
multi-degree $0$ component of Picard group is the semi-abelian scheme over $S$ given by the following extension:

\[0\to \rmH^1(\Gamma,\mathcal O_S^*) \to \Pic^{[0]}(C\slash S) \to \prod_v \Pic^0(C_v\slash S) \to 0.\]
   The algebraic torus $\rmH^1(\Gamma,\mathcal O_S^*)$ parametrizes line bundles which on each fiber are trivial on the components $(C_v)_s$.

The commutative group over log schemes obtained extending $\Gm^{b_1(\Gamma)}$ to $\Gmlog^{b_1(\Gamma)}$ and then taking the push-out of the exact sequence above is nothing but \[\rmH^1(C,\pi^*M_S^\gp)^{[0]}.\]

Then looking  at the long exact sequence in cohomology induced by
\[0\to \pi^*M_S^\gp\to M_C^\gp\to\bar{M}_{C\slash S}^\gp\to 0\]
we get a diagram:

\bcd
0\ar[r] &\rmH^0(C,\bar{M}_{C\slash S}^\gp)\ar[r]\ar[d,"="] &  \rmH^1(C,\pi^*M_S^\gp)\ar[r]\ar[d] & \rmH^1(C,M_C^\gp)\ar[r] \ar[d]&0\\
& \mathbb Z^{E(\Gamma)}\ar[r] &\mathbb Z^{V(\Gamma)}\ar[r] &\mathbb Z\ar[r] & 0.
\ecd

By the Snake Lemma, the latter induces a short exact sequnce
\[0\to \rmH_1(\Gamma,\mathbb Z)\to (\rmH^1(C,\pi^*M_S^\gp)^{[0]})^\dagger\to \LogPic(C)\to 0\]
where $(\rmH^1(C,\pi^*M_S^\gp)^{[0]})^\dagger$ are the bounded monodromy elements in $\rmH^1(C,\pi^*M_S^\gp)^{[0]})$ in the sense of \cite{kajiwara2008logabelian2} with $X=Y= \rmH_1(\Gamma,\mathbb Z) $ and pairing the intersection pairing.

\subsection{Logarithmic Abel-Jacobi theory}\label{sec:logAJ}

In a soon to appear paper \cite{delignepair}, Molcho, Ulirsch  and Wise prove (in  particular) that the logarithmic Jacobian $\LogPic^0(C/S)$ of a family of log smooth curves satisfies the natural generalization in the setting of logarithmic geometry of the Abel-Jacobi theorem. 

More  precisely, let consider $\mathcal G\to\operatorname{LogSch}\slash S$ with $\mathcal G$  being $\Gmlog$, $\mathbf{B}\Gmlog$ or $\A$ an abelian (or log abelian) scheme.
Then they prove the following:
\begin{theom}\cite{delignepair}\footnote{They actually consider more general $\mathcal G$,  but since the presented level of generality is sufficient for our applications,  we state their theorem in this form}
Let $\mathcal G$ as above and let $ C\to S$ be a vertical,  log smooth curve. Then
\begin{itemize}[leftmargin=0.6cm]
\item Every  $S$-morphism $ C\to\mathcal G$  factors through  the  Abel -Jacobi map $C\to \LogPic^0(C/S);$
\item One obtains a pairing
\[\mathcal G( C\slash S)\times \LogPic^0(C/S)\to \mathcal G \] 
which is perfect, namely it induces an equivalence
\[\mathcal G( C\slash S)\to\Hom(\LogPic^0(C/S),\mathcal  G).\]
\end{itemize}
\end{theom}
In particular, for $\mathcal G=\mathbf{B}\Gmlog$ the Theorem implies that $\LogPic^0(C/S)$ is self dual in the category of log abelian varieties. Moreover, for $\mathcal G$ a (log abelian) variety, the Theorem shows that $\LogPic^0(C/S)$  satisfy the log Albanese property.

\medskip

For the applications in this paper we however still  require an explicit description of the log Abel-Jacobi morphism as well as an explicit description of the morphism $f_*\colon  \LogPic^0(C/S)\to A$ through  which $f\colon C\to A$  factors  in the  case  where  $A$ is a honest  abelian variety.
In this section  we thus:
\begin{itemize}[leftmargin=0.6cm]
    \item Explicitly define the  logarithmic Abel-Jacobi map. In particular  this induces a polarization on $\LogPic^0(C/S)$ in the sense of  \cite{kajiwara2008logabelian2}.
    \item Prove, through an explicit description of the homorphism,   that $\LogPic^0(C/S)$  with the log Abel-Jacobi section is initial for morphism to  abelian varities. This is done in the constant degeneration case.
\end{itemize}

\subsubsection{Logarithmic Abel-Jacobi map}\label{sec:logabeljacobi}

We want to define a \emph{logarithmic Abel Jacobi map}
\[\rm{logAJ}_{\sigma_0}\colon C\to\LogPic^0(C\slash S).\]

\begin{defi}
    
A \emph{logarithmic section}  of $(C,M_C)\xrightarrow{\pi} (S, M_S)$,  is a logarithmic map \[(S, M_S)\xrightarrow{p}  (C,M_C),\] 
such that $\pi\circ p\colon (S, M_S)\to (S, M_S)$ is the identity over fine and saturated logarithmic schemes.
In particular, given a logarithmic section we get a morphism
\begin{equation}\label{eq:logsection}
 M_{C,p}\to M_S,
\end{equation}
compatible under the structure map with the schematic map $\mathcal O_{C,p}\to \mathcal O_S$
and such that the composition $M_S\xrightarrow{\bar{\pi}\circ\bar{p}} M_S$ is the identity.
\end{defi}

Such data induces via tropicalization a maps of generalized cone complexes (to be precise cone stacks in the sense of \cite{cavalieri2020moduli}):

\bcd
\Sigma\ar[r, "\bar{p}
"]\ar[dr,"\cong"] &\Gamma\ar[d,"\bar{p}"] \\
& \Sigma,
\ecd

where we denote by $\Sigma$ the tropicalization of $(S,M_S).$ We can think of $\Sigma$ as a diagram of cones with face maps (including self maps and multiple face maps) or as an algebraic stack, an Artin fan. The results of \cite{cavalieri2020moduli} allow us to use the two points of view interchangeably. 

\medskip

When $p$ lies in the smooth locus of  $C$ then $\bar{M}_{C,p}=\pi^*\bar M_S\rvert_p$ and the morphism \eqref{eq:logsection} is determined by the schematic data, the logarithmic Abel-Jacobi map is defined essentially as before by 
\[p\mapsto[\mathcal O_C(p-\sigma_0)],
\]
where $[\mathcal O_C(p-\sigma_0)]$ denotes the  image  under the the morphism
\[\Pic^{\rm totdeg=0}(C\slash S)\to\LogPic^0(C\slash S),\]
well-defined using Definition~\ref{defi:degreeloglb},

In the case where $p$ is intersecting  the singular locus, the definition we gave is too restrictive and $p$ cannot always be lifted to a strict logarithmic  section.
Indeed, let's restrict to a fiber $C_s$ such that $p(s)$ is a node. Requiring to have a map of monoids $\bar{p}\colon\bar{M}_{C,p}\to\bar{M}_S$, where
\[\bar{M}_{C,p}=\left\{(a,b)\in \bar{M}_S\times\bar{M}_S\;|\;a-b=\ell(e)\mathbb Z\right\}\subseteq  \bar{M}_S\times\bar{M}_S,\]
such that $\bar{M}_S\xrightarrow{\bar \pi}\bar{M}_{C,p}\xrightarrow{\bar{p}}\bar{M}_S$ is the identity  is equivalent to require the existence of a section passing through the node for the local toric model:

\[\Spec( k[\bar{M}_S][x,y]\slash (xy-t^{\ell(e)})\to \Spec( k[\bar{M}_S]),\]
which is not always possible for $t$ a non invertible parameter in $k[\bar{M}_S]$. In other words, the image of $\bar{p}$ has to tropicalize to a lattice point of the edge $e$, of which there may not be any. Such points may appear after a base change.

\begin{example}
    Let $(C, M_C)\to (\Spec k,\NN)$ be a log smooth curve over the standard log point and let $p\in C$ a node. Then, étale locally around $p$ we have $\mathcal O_{C,p}=k[x,y]/(xy).$ A quasi-log section with underlying section $\Spec k\to p$ is the data of a monoid morphism
    \[\bar{M}_{C,p}\cong\left\{(a,b)\in\mathbb N^2_{f_1,f_2}\;|\; a-b=\ell(e)\mathbb Z\right\}\xrightarrow{(m\;n)}\mathbb N,\]
    with $m,n\neq 0.$ On the other hand, the structure morphism $\bar{M}_S=\mathbb N_f\xrightarrow{\pi} \bar{M}_{C,p}$ identifies the image of the generator $f$  with $\frac{f_1+f_2}{\delta(e)}$. So the composition $\bar{p}\circ\bar{\pi}\colon \mathbb N\to\mathbb N$ sending $f\to \frac{n+m}{\ell(e)}f$ is the identity only when it corresponds to multiplication by one, which is impossible if $\delta(e)=1.$ On the other hand, the induced map of rays $\rho\cong\mathbb R_{\geq 0}$ is always a isomorphism. Up to refining the latticed $\mathbb Z$ to $\frac{1}{n+m}\mathbb Z$ the isomorphism respect the integral structure.
\end{example}

We  thus give the following definition:

\begin{defi}
A \emph{quasi-logarithmic section}  of   $(C,M_C)\xrightarrow{\pi}(S,M_S)$  is a logarithmic map \[(S, M_S)\xrightarrow{p}  (C,M_C)\] 
 such that the induced maps of cone stacks $\bar{p}\colon\Sigma\to\Gamma$ is a section in the sense that $\bar{\pi}\circ\bar{p}\colon\Sigma\to\Sigma$ is a rational cone stack isomorphism, i.e. an isomorphism cone by cone compatible with face morphisms. Local quasi-log sections, i.e. section on a \'etale or Zariski $U\to S$ are defined the same way.
 \end{defi}
A logarithmic quasi-section induces a commutative diagram which by \cite{KatoF},  \'etale locally on $S,$ looks as follows

\bcd
\bar{M}_{C,p}=\bar{M}_S\oplus_{\mathbb N}\oplus\mathbb N^2\oplus \mathcal O^*_{C,p}\ar[r, "p"]\ar[d] &\bar{M}_S\oplus \mathcal O_{S}^*.\ar[d]\\
\mathcal O_{C,p}=\mathcal O_S[x,y]\slash (xy-t)\ar[r] & \mathcal O_S
\ecd

If we restrict to the locus $t=0,$ then 
since the generator $e_1$ and $e_2$ of $\mathbb N^2$ go to $x$ and $y$ under the left vertical map and so to $0$  in $\mathcal O_S\slash (t),$ it follows by the commutativity that neither of the two can map to $0$ in $\bar{M}_S$ under $\widetilde{p}.$

Namely, working relatively to the base, over a fixed point $q\in\Sigma$ we have that $\widetilde{p}(q)$ 
lies in the interior of the edge $e\in\Gamma_q$  corresponding  to the node $p.$

\begin{lem}
Let $p$ be a logarithmic point in the sense above.
  There exists a logarithmic modification $(S',M_S')\to (S,M_S)$  and a subdivision $(\widetilde{C},M_{\widetilde{C}})\to (C',M_C')$ such that the pull-back $p'$ of $p$ to $S'$ is a a strict log point which admits a lift  $\widetilde{p}$ to $(\widetilde{C},M_{\widetilde{C}}).$
\end{lem}
\begin{proof}
Everything is pulled back from the universal geometry, i.e. the local toric picture, where both $(S',M_S')$ and $(\widetilde{C},M_{\widetilde{C}})$ are determined by the maps of cones $\sigma\xrightarrow{\mathfrak{p}}\Gamma.$ 

In the universal geometry case $(S',M_S')\to (S,M_S)$ is the Kummer extension determined by the change of lattices necessary to make $\mathfrak{p}$ integral and  $(\widetilde{C},M_{\widetilde{C}})$ by  the subdivison of $\Gamma$ necessary to make  $\mathfrak{p}$ combinatorially flat and resolve the singularities of the total space of base change $C'$ . Finally $\widetilde{p}$ is simply the proper transform of $p'.$

\end{proof}
Let $p_0$  be a smooth point of $C,$ i.e. a strict section.
Then given a logarithmic point $p\colon (S,M_S)\to (C,M_C)$ we define

\[\rm{logAJ}_{p_0}(p)=[\mathcal O_{\widetilde{C}}(\widetilde{p}-p_0)]\in\LogPic(\widetilde{C}\slash S')=\LogPic(C\slash S).\]

It is now immediate to define $\rm{logAJ}_{p_0}$ for a curve over a log point $(C, M_C)\to(S=\Spec k, M_S).$
Then $\LogPic^0(C\slash S)$ represents the moduli functor over fine and saturated log schemes parametrizing isomorphism classes of logarithmic line bundles on  $(C,M_C)\to(S,M_S)$ and 
\[\rm{logAJ}_{\sigma_0}\colon C/S\to \LogPic^0(C/S)\]
is the  data of a log line bundle  $L$ on $(C,M_C)\times_{(S,M_S) }(C,M_C)\xrightarrow{\pi_1}(C,M_C)$ such that the restriction to $(p,M_S)\to (C,M_C)$ a log point of the base is the morphism described above. 
\begin{rem}
    Since $(C,M_C)\to(S,M_S)$ is combinatorially flat, the fiber product in the category of  fs Log scheme coincides with the  fiber product in coherent log schemes and in particular the underlying scheme is $C\times_S C$
\end{rem}

Let $(C, M_C)\xrightarrow{\Delta} (C\times C, M_C\times_{M_S} M_C)$ be the diagonal. 
Applying the previous Lemma to the diagonal section we obtain a
 logarithmic modification $(\widetilde{\mathcal C},M_{\widetilde{\mathcal C}})\to (C\times C, M_C\times_S M_C)$
such that the  strict transform of $\widetilde{\Delta}$ of the  diagonal is a section lying in the smooth locus.

Then the required log line bundle is the line bundle represented by the honest line bundle 

\[\mathcal O_{\widetilde{C}}(\widetilde{\Delta})\otimes\pi_2^*\mathcal O_C(-p_0)\]

\begin{lem}
Let $(C,M_C)\to (S,M_S)$ be a log smooth curve over a log point. Then the \emph{tropicalization} the logarithmic Abel-Jacobi map as defined above is the tropical Abel- Jacobi map of Mikhalkin-Zharkov\cite{mikhalkin2008tropical}.
    \end{lem}
    
\begin{proof}
It is proved in \cite[Section~4.9]{molchowiselog}  that the $\mathcal P\mathcal L$-torsor (automatically with bounded monodromy) associated to a log line budle $\mathcal F$ admits a canonical lift to a torsor under the sheaf $\mathcal L$ of linear functions. So we have am explicit description of the \emph{tropicalization map}:
\[\rm{trop}\colon\LogPic^0(C\slash S)\to\TroJac(\Gamma\slash S),\]
where $\Gamma$ is the tropical curve metrized in $\bar{M}_S$ associated to $(C,M_C)$
Furthermore, it follows from the bounded monodromy condition that every $\mathcal L$-torsor with bounded monodromy is representable by a tropical divisor of degree zero on some model of $\Gamma$, i.e. a tropical curve metrized in extensions of $\bar{M}_S$, obtained from $\Gamma$ subdividing the edges. In other words, in analogy with the $\LogPic$, the tropical Jacobian admits a cover (in the log étale topology of log fs schemes) by Jacobians of models of $\Gamma$\footnote{If we consider the classical setting of tropical curve metrized in $\mathbb R$ there is no need to extend the base monoid, i.e. all the models are just subdivision of the given curve, and the tropical Jacobian of $\Gamma$ simply parametrizes degree zero tropical divisors, up to tropical rational equivalence. We refer to \cite{mikhalkin2008tropical} for a thorough study of this case.}.

Using this point of view, the tropical Abel-Jacobi map is simply defined in \cite{mikhalkin2008tropical} as
\[\begin{array}{rccl}
    \operatorname{tropAJ}_{v_0}\colon &\Gamma & \longrightarrow & \operatorname{TropJac}(\Gamma) ,\\
    & p & \longmapsto & [p-v_0]
\end{array} \]
where $[p-v_0]$ denotes the class of the total degree zero divisor $p-p_0$.
Notice that the choice of $p\in\Gamma$ determines a model of $\Gamma$. To be precise, a point $p\in\Gamma$ is a section over the cone $\sigma$ dual to the base monoid $\bar{M}_S$

Now, given $p\colon (S,M_S)\to (C,M_C)$ a quasi-log section, this determines via tropicalization $C\to\Gamma$ a point $p\in\Gamma$ in the sense above, and thus a model $\widetilde{\Gamma}$ of $\Gamma$ metrized in $\bar{M}_{S'}$ and a corresponding model $(\widetilde{C},M_{\widetilde{C}})\to (S',\bar{M}_{S'})$ for the log curve such that the proper transform $\widetilde{p}$ is a section through the smooth (alias the strict) locus. Now the image of $p$ under the log Abel map is the class in $\LogPic^0(C/S)$ of $\mathcal O_{\widetilde{C}}(\widetilde{p}-\sigma_0).$ On the other hand this tropicalizes to $\mathrm{tropAJ}_{v_0}(p)=[p-v_0]$ showing that the map induced by tropicalization is the obvious one.
\end{proof}

\subsubsection{Polarization from the Abel-Jacobi Section}
Recall the construction of $\LogPic^0(C\slash S)^\vee$ as log dual torus. 
We have only discussed how to do so in the case of constant degeneration, but we notice that the semi-abelian scheme $(R^1\pi_*\mathcal O_{C}^*)^{[0]}$ and thus its dual in the sense of \cite{kajiwara2008logabelian2} is globally defined over $S$, as well as the sheaves $X=Y=R^1p_*\underline{\mathbb Z}$ where $p$ is the projection from the universal tropical curve $\Gamma$ to the base $S$; these restrict to sheaves of lattices over each logarithmic stratum of $S$.

\begin{lem}
    Via the log Abel Jacobi map one can define a polarization, namely a morphism
\[\LogPic^0(C\slash S)^\vee\xrightarrow{\rm{logAJ}^*_{\sigma_0}} \LogPic^0(C\slash S).\]
\end{lem}

\begin{rem}
    Technically, this is rather a polarization on the dual. In this specific case, it is expected to be an isomorphism so that we also get a polarization on $\LogPic^0(C/S)$. Otherwise, it can be shown that a polarization on a log-torus induces a polarization on its dual.
\end{rem}

\begin{proof}
    An object $\mathcal F$ in $\LogPic^0(C\slash S)^\vee$ is a $M_S^{\gp}$-torsor over $\LogPic^0(C\slash S)$ obtained as described in \ref{sec-dual-torus-abelian-var}. We can consider the pull-back $\rm{logAJ}_{\sigma_0}^{-1}\mathcal F$ which is a $\pi^*M_S^{\gp}$-torsor to which is naturally associated the $M_C^{\gp}$- torsor 
$a^*\mathcal F:=\rm{logAJ}_{\sigma_0}^{-1}\mathcal F\otimes_{\pi^*M_S^{\gp}}M_C^{\gp}.$
In order to prove that $a^*\mathcal F\in\LogPic^0(C/S)$ we need to argue that it has bounded monodromy and degree $0.$ Both are conditions that can be checked on the associated $\mathcal{PL}$-torsor \[\overline{a^*\mathcal F}=\overline{\rm{logAJ}_{\sigma_0}^{-1}\mathcal F\otimes_{\pi^*M_S^{\gp}}M_C^{\gp}}=\rm{tropAJ}_{v_0}^{-1}\bar{\mathcal F}\otimes_{\pi^*\bar{M}_S^{\gp}}\bar{M}_C^{\gp}.\]

Moreover these are fiber-wise conditions, so we can reduce to working with over $(S,M_S)=(\Spec k, \bar{M})$ for some monoid $ \bar{M}.$

So we reduce to study the following: $\Gamma$ is tropical curve metrized in $\bar{M}$, and
\[a\colon\Gamma\to \Hom(\rmH_1(\Gamma,\mathbb Z),\bar{M}^{\gp})^\dagger\slash \Hom(\rmH_1(\Gamma,\mathbb Z)),\]
the tropicalization of the log Abel-Jacobi, 
$\overline{a^*\mathcal F}\in \rm H^1(\Gamma,\mathcal P\mathcal L)$ has bounded monodromy, namely some lift to a representative in $\Hom(\rm H_1(\Gamma,\mathbb Z),\bar{M}^{\gp}) $ has bounded monodromy (see diagram \ref{eq:tropicalsheaves}), but such lift is tautologically $(\bar{h},\bar{\mathcal F})$ which has bounded monodromy by assumption.

\end{proof}

\begin{rem}
    We remark that since the tropicalization of the logarithmic Abel-Jacobi map is the tropical Abel-Jacobi map, which in the classical setting of tropical curves metrized in $\mathbb R$ is the one considered in \cite{mikhalkin2008tropical}, by the results of \emph{ibid.} the polarization above tropicalizes to the self-duality isomorphism for the tropical Jacobian of tropical curves. 
With some care, it should be possible to conclude that the polarization is in fact an isomorphism of tropical abelian variety over fs log scheme in the sense of \cite[Section~3]{molchowiselog}.
\end{rem}

\subsubsection{Universal property for morphisms to abelian varieties}
\begin{prop}\label{prop:logalb}
Let $(C,M_C)\to(S,M_S)$ be a proper, vertical, log smooth curve with constant degeneration and $\sigma_0$ a strict section. Then $(\LogPic^0(C\slash S),\rm{logAJ}_{\sigma_0})$ is initial for morphism $C\xrightarrow{\varphi} \mathcal A$ to abelian varieties such that $\varphi(\sigma_0)=0_{\mathcal{A}}$
\end{prop}

\begin{proof}
    We need to show that $(\LogPic^0(C\slash S),\rm{logAJ}_{\sigma_0})$ satisfies the universal property with respect to maps $\varphi\colon   C\to\mathcal A$ mapping $\sigma_0$ to the identity. This means defining $\phi\colon(\LogPic^0(C\slash S)\to\mathcal A$ such that $\phi\circ\rm{logAJ}_{\sigma_0}=\varphi.$
    Let $p\in C$  a logarithmic point (over some fs log scheme $
    (T,M_T)\to(S,M_S)$) then $\rm{logAJ}_{\sigma_0}(p):=[\mathcal O_{\widetilde{C}}(p-\sigma_0)]$  where $\widetilde{C}$ is a model of $C$ possibly over a log modification $(T',M_{T'})\to(T,M_{T})$, such that $p$ pulls-back to a strict logarithmic point of $\widetilde{C}$. We denote by $\widetilde{\Gamma}$ the tropicalization of this model. 

     Furthermore, by Lemma~\ref{lem:degreeonedges}  below, we can assume that there exists $\alpha\in \rmH^0(T',\pi_*\bar{M}_{\widetilde{C}}^\gp)$ such that $\mathcal O_{\widetilde{C}}(p-\sigma_0)(\alpha)\in \Pic^0(\widetilde{C}\slash T')$ has non-zero degree only possibly on $\widetilde{C}_v$ for $v\in \sfV(\widetilde{\Gamma})\setminus\sfV(\Gamma)$. Since these components are rational $\Pic^{d_v}(\widetilde{C}\slash T')\cong T'$. In other words, the degree is $0$ at vertices of $\Gamma$.

    We denote by $\Pic^{[0]_{\sfV(\Gamma)}, [d_e]}(\widetilde{C}\slash T')$ the connected component of $\Pic^0(\widetilde{C}\slash T')$ containing the section $\O_{\widetilde{C}}(p-\sigma_0)(\alpha).$
We wish to define a morphism 
\[\phi\colon \Pic^{[0]_{V(\Gamma)}, [d_e]}(\widetilde{C}\slash T')\to \mathcal A,\]
such that $\phi (  \O_{\widetilde{C}}(p-\sigma_0)(\alpha))=\varphi(p)$
and such that it does not depend from the choice of strict piece-wise linear  function $\alpha$ with the  property above.

Notice that $\Pic^{[0]_{\sfV(\Gamma)}, [d_e]}(\widetilde{C}\slash T')$  is a $\Gm^{b_1(\Gamma)}$- bundle over
\[\prod_{v\in \sfV(\widetilde{\Gamma})} \Pic^{d_v}(C_v\slash T')\cong \prod_{v\in \sfV(\Gamma)}  \Pic^{0}(C_v\slash T').\]
Since morphisms from an algebraic torus to an abelian variety are constant \cite[Chapter~1]{milneAV}, $\phi$ factors through 
\[\prod_{v\in \sfV(\widetilde{\Gamma})} \Pic^{d_v}(C_v),\]

Let us remark that the image of  $C_v\subseteq \widetilde{C}$ via $\widetilde{C}\xrightarrow{\rho} C$ is the node $T'\xrightarrow{q_e} C$ for $e$ the edge of $\Gamma$ containing $v.$ In what follows we abuse notation and write $q_e$ for $q_e(T').$

We define the desired morphism $\phi$ as follows:
\[\phi (-):= \sum_{v\in \sfV(\Gamma)} {\phi_{v,*}(-)}\;\; + \tau(\alpha)\]

where $\phi_{v,*}$ is the natural map obtained dualizing $\varphi\rvert_{C_v}^*\colon \mathcal A^\vee\to  \Pic^{0}(C_v\slash T')$
via the self-duality of the Jacobian scheme for smooth projective curves and $\tau(\alpha)$ is the translation defined by:

\[\tau(\alpha)=\sum_{e\in \sfE(\Gamma)}\left(\sum_{v\in e^{\circ}} d_v\right)\varphi(q_e).\]
    
Here we denoted by $v\in e^{\circ}$ the vertices of $\widetilde{\Gamma}$ supported in the interior of  $e\in E(\Gamma).$
The Abel-Jacobi Theorem for smooth curves ensures that for any $(n_i)_{i\in I}$ such that $\sum n_i=0,$
\[\phi_{v,*}(\O_{C_v}(\sum n_iq_i))=\sum n_i\varphi\rvert_{C_v}(q_i)\]
and on each vertex of $\Gamma,$ $\O_{\widetilde{C}}(p-\sigma_0)(\alpha)$ is of this form, where except for the points $p$ and $p_0$, the $q_i$ appearing are nodes of $\widetilde{C}$ and the coefficients $n_i$ are determined by the slopes of $\alpha.$ It's then immediate to verify that $\phi$ satisfies the required commutativity.

To conclude the proof we need to show that $\phi$ is indeed well-defined, namely that it does not depend on the choice of $\alpha$. Namely, we need to show that for $\lambda\in \rmH^0(T',\pi_*\bar{M}_{\widetilde{C}}^\gp)$  such that $\sfV(\Gamma)$ does not intersect the bending locus of $\lambda$
\[\phi(\mathcal O_{\widetilde{C}}(\lambda))=0\in\mathcal A.\]

Let us compute:
\begin{align*}
      \phi(\O_{\widetilde{C}}(\lambda))= & \sum_{v\in \sfV(\Gamma)} {\phi_{v,*}(\mathcal O_{C_v}(\sum_{e\vdash v}s^v_e(\lambda) q_e)) +\sum_{e\in \sfE(\Gamma)}(\sum_{v\in e^{\circ}} d_v)\varphi(q_e)}\\
      =& \sum_{v\in \sfV(\Gamma)} {\sum_{e\vdash v}s^v_e(\lambda) \varphi(q_e)+\sum_{e\in \sfE(\Gamma)}(\sum_{v\in e^{\circ}} d_v)\varphi(q_e)}
          \end{align*}
Where we denoted by  $s^v_e(\lambda)$ the outgoing slope of $\lambda$ along $e$ at $v.$
Then for each edge $e\in E(\Gamma)$, let $v, w$ its end points, we have that 
\[s^v_e(\lambda) + s^w_e(\lambda)+ \sum_{v\in e^{\circ}} d_v=0\]
which concludes the proof.
\end{proof}

Concretely, we use that nodes have a constant image to glue the maps induced by the Abel-Jacobi maps of the components using translation paparemeters to make them agree over nodes. We expect this universal property to hold for the morphisms to log-abelian varieties as well. However, in this case we do not expect the nodes to have a constant image: nodes correspond to edges of the tropicalization, and their image should be related to the tropicalization of the map.

\begin{lem}\label{lem:degreeonedges}
    Let  $(C,M_C)\to (S,M_S)$ be a proper vertical, log smooth curve with constant degeneration, and let $\mathcal L$ a line bundle of total degree $0.$ Then there exists a model $(\widetilde{C},M_{\widetilde{C}})\to (S',M_{S'})$ and a strict piece-wise linear function $\alpha\in \rmH^0(S',\pi_*\bar{M}_{\widetilde{C}}^\gp)$ such that $\mathcal L(\alpha)$ on $\widetilde{C}$ has degree $0$  on the irreducible components $C_v$ with $v\in \sfV(\Gamma)\subseteq \sfV(\widetilde{\Gamma}).$
\end{lem}

\begin{proof}
Let $\Gamma\to\sigma$ be the tropicalization of $(C,M_C)\to (S,M_S)$. Here we denote by $\sigma$ the cone dual to $\bar{M}_S,$ which is a constant sheaf on $S$ because of the constant degeneration hypothesis.
 
The multi-degree $\underline{\deg}(\mathcal L)$ defines  a tropical divisor $D_L$ on $\Gamma$ supported at the vertices. The proof is essentially tropical: we need to prove that every tropical divisor is linearly equivalent to a divisor not supported at the vertices. For $D$ a tropical divisor, we have the tropical linear system (shortly TLS):
\[|D_L|=\left\{\alpha\colon\Gamma\to \bar{M}_S^\gp\;\;\text{PL on the edges}\;\;| \;\; D_L+\operatorname{div}(\alpha)\geq 0\right\}\slash \bar{M}_S^\gp. \]
We start defing a piecewise linear function $\alpha$ on a subdivision of $\Gamma$ as follows:
\begin{itemize}[label=$\circ$,leftmargin=0.6cm]
    \item If $v\in\sfV(\Gamma)$, we set $\alpha(v)=0\in\bar{M}_S^\gp$.
    \item For each flag $v\in e$, choose an outgoing slope $s_e^v\in\ZZ$ such that the sum of outgoing slopes at each vertex of $\Gamma$ is precisely the divisor $D_L$.
    \item For an edge $e=(v,v')$, in case the two outgoing slopes do not agree to be $0$, this does not define a strict piecewise linear function. We thus need to show that up to subdividing $e$, it is possible to extend the function to a piecewise linear one.
\end{itemize}
This last step can be performed subdividing (each edge) $e$ in three parts of equal lengths. This yields a subdivision $\widetilde{\Gamma}$ of $\Gamma$. Let $(v,w_1,w_2,v')$ be the four vertices along the subdivided edge. The value of $\alpha$ at $v$ and $v'$ is $0$. For the slope to be $s_e^v$ and $s_e^{v'}$ near $v$ and $v'$, we need to have $\alpha(w_1)=s_e^v\frac{\ell(e)}{3}$ and $\alpha(w_2)=s_e^{v'}\frac{\ell(e)}{3}$. Therefore, the slope over the edge $(w_1,w_2)$ is
$$\frac{s_e^{v'}\ell(e)/3 - s_e^v\ell(e)/3}{\ell(e)/3} = s_e^{v'}-s_e^v\in\ZZ.$$
By construction, the divisor $D_L+\operatorname{div}(\alpha)$ is linearly equivalent to $D_L$ and supported on $\sfV(\widetilde{\Gamma})\setminus\sfV(\Gamma)$.
\end{proof}

%% file: sec-splitting-moduli.tex
\section{A refined stratification of the moduli space of stable maps}
\label{sec-refined-stratification-via-coverings}


In this section we shortly denote by $\M(X)$ the  moduli space of stable maps $\M_{g,n+m}(X,\beta)$ for given discrete data $m,n,g$ and  $\beta\in \rmH_2(X,\ZZ)$.

\medskip

The goal of this section is first to show that the moduli space $\M(X)$ is a union of open and closed components described in terms on how torsion line-bundles on $X$ pull-back to curves. We then study the relation between this refinement and the natural boundary stratification of $\M(X)$. Finally we relate the restriction of virtual classes to these components to moduli of stable maps to  coverings of $X$.

\subsection{Coverings of complex manifolds via line bundles}
\label{sec-H-covering}

Let $X$ be a smooth complex projective variety with positive irregularity $q(X)$. We describe the construction of \textit{$H$-coverings}.

\subsubsection{Description via subgroups}

Let $H\subset \Pic^0(X)[\delta]$
be a finite subgroup of the $\delta$-torsion of the Picard group of $X$. Via the Weil pairing, we have its orthogonal $H^\perp\subset \Alb(X)[\delta]$. As $\Alb(X)[\delta]$ is canonically isomorphic to $\rmH_1(X,\ZZ_\delta)$ (where we deleted the torsion part of $\rmH_1(X,\ZZ)$), we can consider its preimage in $\pi_1(X)$ via the (surjective) Hurewicz morphism $\H:\pi_1(X)\to \rmH_1(X,\ZZ_\delta)$. The latter is a normal subgroup of $\pi_1(X)$ with quotient
$$\pi_1(X)/\H^{-1}(H)\simeq \Alb(X)[\delta]/H^\perp\simeq \operatorname{Hom}(H,\mu_\delta)=\widehat{H},$$
where the last equality is the definition of $\widehat{H}$, the Pontrjagin dual of $H$, isomorphic to $H$ though non-canonically.

\begin{defi}
We call $H$-covering of $X$ (denoted by $X_H$) the covering corresponding to the subgroup $\H^{-1}(H^\perp)\subset\pi_1(X)$. It is a Galois covering of $X$ with deck-automorphism group $\widehat{H}=\operatorname{Hom}(H,\mu_n)$.
\end{defi}

\begin{rem}
The covering $X_H$ may also be seen as the fiber product between $X\to\Alb(X)$ and the covering of the complex torus $\Alb(X)$ determined by $H^\perp$, or rather its preimage in $\rmH_1(X,\ZZ)$ modulo torsion.
\end{rem}

\begin{expl}
    If $X$ is an elliptic curve and $n=2$, we have $\rmH_1(X,\ZZ_2)\simeq \ZZ_2^2$, which has five subgroups: $\{0\}$, $\ZZ_2^2$ and three subgroups isomorphic to $\ZZ_2$, yielding five different coverings.
\end{expl}

\subsubsection{Description trivializing line bundles} As we intend to work with line bundles, we provide an alternative description of the $H$-covering of $X$ using a line bundle construction. Given $H\subset \Pic^0(X)[\delta]$, consider the vector bundle $P_H:=\operatorname{Tot}(\L_1\oplus\dots \L_h)\xrightarrow{p} X$ where $\L_1,\dots \L_h$ are the line bundle corresponding to generators of $H.$ Let furthermore denote by $\xi_i$ the tautological section of $p^*\L_i.$ Let $\widetilde{X}_H\subseteq P_H$ the sub-scheme cut out by $\left\{\xi_i^n-1=0\right\}_{i=1,\dots h}$ and let denote by $\mathrm{pr}:\widetilde{X}_H\to X$ the  natural map induced by $p$.
This shows that $\widetilde{X}_H$ is a $\Hom(H,\mu_n)$-principal bundle. All its connected components are therefore isomorphic (differing by the action of $\widehat{H}$) and provide coverings of $X$.

\begin{prop}\label{prop-lifting-K-covering}
The covering $\widetilde{X}_H$ satisfies the following properties:
\begin{enumerate}
    \item The pull-back of every $\L\in H$ to $\widetilde{X}_H$ is trivial.
    \item A map $f:Y\to X$ factors through $\widetilde{X}_H\to X$ if and only if $\forall\L\in H$, $f^*\L\simeq\O$.
    \item Each connected component of $\widetilde{X}_H$ is isomorphic to the $H$-covering of $X$ as defined in Section \ref{sec-H-covering}. In particular, $X_H$ also satisfies (2).
\end{enumerate}
\end{prop}

\begin{proof}
\begin{enumerate}
    \item For $\L\in H$, the pull-back of $\mathrm{pr}^*\L$ has a nowhere vanishing section by definition of $\widetilde{X}_H$, and thus the line bundle is trivial.
    \item By construction, for any line bundle $\L\in H$, $\mathrm{pr}^*\L$ possesses a nowhere vanishing section. In particular, the pull-back by any map $\widetilde{f}:Y\to\widetilde{X}_H$ is also trivial. Conversely, if $f:Y\to X$ satisfies $f^*\L\simeq\O$ for any $\L\in H$, choosing a section $\xi_\L$ such that $\xi_\L^{\otimes n}=1$ yields a map to $\widetilde{X}_H$.
    
    \item Since $\widetilde{X}_H$  is a $\widehat{H}$-torsor, all of its connected components are isomorphic. We now prove that the connected cover $X_H$ is isomorphic to any component of $\widetilde{X}_H$. To do so, we use the Hurewicz map $\H:\pi_1(X)\to \rmH_1(X,\ZZ_\delta)\simeq \Alb(X)[\delta]$ along with the Weil pairing.
    
    We have seen in Section \ref{sec-weil-pairing-line-bundles} that the Weil pairing $W([\gamma],\L)\in\mu_\delta$ gives the monodromy of the torsion line bundle $\L$ along the loop $\gamma\in\pi_1(X)$, where $[\gamma]$ is the class of $\gamma$ in $\rmH_1(X,\ZZ_\delta)$. Restricting to the subgroup $\H^{-1}(H^\perp)\subset\pi_1(X)$ thus yields the monodromy of the pull-back to $X_H$. In particular, if $\L\in H$ and $\widetilde{\gamma}$ a loop in $X_H$, we have
    $$W_X(\H(\pi_{H\ast}\widetilde{\gamma}),\L)=0.$$
    Therefore, every $\pi_H^*\L$ is trivial on $X_H$ and we get a morphism of covering $X_H\to\widetilde{X}_H$. As $X_H$ is connected, it only maps to a connected component of $\widetilde{X}_H$.
    
    Conversely, consider a loop $t\mapsto (\gamma(t),(\xi_\L(t)))$ in $\widetilde{X}_H$ and push it to $X$. For any $\L\in H$, $\gamma^*\L$ has a section given by $\xi_\L(t)$. Therefore, $W_X(\H_n(\gamma),\L)=0$ and we deduce that $\gamma\in\H^{-1}(H^\perp)$. Since the latter is the subgroup associated to the covering $X_H$, we get a morphism of covering $\widetilde{X}_H\to X_H$, finishing the proof.
\end{enumerate}
\end{proof}

\begin{rem}
    As an exercise, the reader may check that the covering may also be obtained as follows. Consider the covering $\widetilde{X}_{\Pic^0(X)[\delta]}$ trivializing all $\delta$-torsion line bundles on $X$. The group $H^\perp\subset \Alb(X)[\delta]$ acts on $\widetilde{X}_{\Pic^0(X)[\delta]}$ as follows. If $\chi\in H^\perp$, it induces a morphism $\Pic^0(X)[\delta]\to\mu_n\subset\CC^*$. We then set
    $$\chi\cdot(x,(\xi_\L)) = (x,(\chi(\L)\xi_\L)).$$
    The $H$-covering of $X$ is any connected component of the quotient of the above covering by $H^\perp$.
\end{rem}

\begin{expl}
    In case $X=A$ is an abelian variety, it is isomorphic to its Albanese variety, and we have $\Alb(A)[\delta]\simeq A[n]$. The $K$-covering of $A$ for $H\subset \Pic^0(A)[\delta]=A^\vee[n]$ is also an abelian variety. 
\end{expl}

By construction, the collection of coverings $\pi_H:X_H\to X$ is in bijection with the collection of subgroups of $\Pic^0(X)[\delta]$, which is a partially ordered set. In particular, we have the following.

\begin{lem}
    If $H_1\subset H_2\subset \Pic^0(X)[\delta]$, we have a morphism of coverings $X_{H_2}\to X_{H_1}$.
\end{lem}

\subsection{Components of the moduli space of maps}
\label{sec-defi-refined-stratification}

\subsubsection{Definition of the components and splitting of the smooth locus}
The decomposition of the smooth locus into open and closed substacks relies on the following crucial lemma.

\begin{lem}\label{lem-pull-back-open-closed}
    If $\L$ is a torsion line bundle on $X$, the locus of stable maps $f:C\to X$ where $f^*\L=\O$ is both open and closed.
\end{lem}

\begin{proof}
Let  $\Pic^{[0]}(C\slash\M(X))[\delta]$ be the finite subgroup of $\delta$-torsion line bundle of the universal Jacobian; it is a finite, étale, (non-proper, see e.g. \cite[Section~1.2]{chiodo2008stable}) group scheme over $\M(X)$ and thus the locus where two sections agree is open. 
Since it is always also closed, we conclude.

Alternatively, a torsion line bundle is fully determined by its monodromy along loops, which is locally constant since we can always deform a loop in a fiber to a nearby fiber.

\end{proof}



We immediately deduce the following.

\begin{coro}
    The morphism from $\M(X)$ to the subgroups of $\Pic^0(X)[\delta]$ mapping $f\colon C\to X$ to $\ker f^*$ is locally constant.
\end{coro}

In particular, given a subgroup $K\subset \Pic^0(X)[\delta]$, we can consider the following open and closed component of $\M(X)$:
$$\M_K(X)=\{ f:C\to X | \ker(f^*\colon\Pic^0(X)[\delta]\to \Pic^{[0]}(C)[\delta])=K \}.$$
It is an open substack of $\M(X)$ to which we can restrict the usual obstruction theory to define a virtual class $\vir{\M_K(X)}$. 

\medskip

The components $\M_K(X)$ are mutually disjoint. 
However, it may not be obvious why these classes $\vir{\M_K(X)}$ are ``tautological'', nor how to compute with them. This is the content of the next subsection.

\subsubsection{Relation with maps to $H$-coverings}

Let $H\subset \Pic^0(X)[\delta]$ be a subgroup. 
We consider the  moduli space of maps to the $H$-covering $X_H$:
$$\M_{g,n+m}(X_H,\beta)=\bigcup_{\widetilde{\beta}}\M_{g,n,m}(X_H,\widetilde{\beta}),$$
where the union is over homology classes $\widetilde{\beta}\in\rmH_2(X_H,\ZZ)$ such that $\pi_{H\ast}\widetilde{\beta}=\beta\in \rmH_2(X,\ZZ)$. If $\beta$ does not belong to $\mathrm{Im}\pi_{H\ast}$, this space is empty. We will denote this moduli space simply  by $\M(X_H)$, suppressing the additional combinatorial data in the notation.

\medskip

Composing with the covering map, we have a natural morphism of moduli spaces $\M(X_H)\to\M(X)$. 
The image consists of those maps that can be lifted to the $H$-covering and therefore it is not surjective. 

By Proposition~\ref{prop-lifting-K-covering}, for a map $f\colon C\to X$  to admit a lift to $X_H$, the kernel $K=\operatorname{Ker}(f^*)$ should contain $H$, though it may be strictly bigger. 
The image is thus contained in
$$\M^H(X)=\bigsqcup_{K\supset H}\M_K(X),$$
for $\M_K(X)$ the open and closed substacks of $\M(X)$ defined above.

The morphism $\M(X_H)\to\M^H(X)$ is in fact a $\widehat{H}$ -torsor as  deck automorphisms act on a chosen lift.
 We can give a natural modular interpretation to $\M^H(X).$
  Let $F:\C\to X\times S$ be a family of curve in $\M^H(X)$. By definition of these open and closed sub-stacks we have that $F^*\mathcal L$ is trivial along the fibers of $\C\to S$ for  each $L\in H$.  We surely have fiberwise lifts of $F_s$ to $F_H$, but in order to lift the family we need to trivialize $F^*L$ along a section. To do so, one can take the fiber product with a chosen section $S\xrightarrow{\sigma}\C\xrightarrow{F}X$:
\bcd
S_H\arrow[r]\arrow[d] & X_H\arrow[d] \\
S \arrow[r] & X.
\ecd

Another way to describe the substack $\M^H(X)$ is to define the moduli space of $H$-\textit{rubber} maps to $X_H$. The covering $X_H\to X$ being a $\widehat{H}$-principal bundle, it induces a map from $X$ to the classifying space of $\widehat{H}$. The moduli space $\M^{\sim H}(X)$ of $H$-rubber maps consists in families making the first diagram commute. Concretely, it means that there exists an $\widehat{H}$-torsor $S_H\to S$ such that the pull-back family $\C_H$ endowed with the corresponding $\widehat{H}$-action has an equivariant map to $X_H$, as depicted in the second diagram:

\begin{center}
\begin{tikzcd}
\C\arrow[r]\arrow[d] & X\arrow[d] \\
S \arrow[r] & B\widehat{H}
\end{tikzcd},
\adjustbox{scale=0.75}{
\begin{tikzcd}
 & \C_H\arrow[rr]\arrow[dd]\arrow[dl] & & X_H \arrow[dl] \\
 \C\arrow[rr,crossing over]\arrow[dd] & & X & \\
 & S_H \arrow[dl] & & \\
S &  &  &  \\ 
\end{tikzcd}.
}
\end{center}

Furthermore it follows from the given definition of rubber maps that two maps are isomorphic if they differ by the deck-transformation action (see \cite{carocci2024rubber} for a complete explanation of this fact.) 

\begin{lem}
    The moduli space of $H$-rubber stable maps coincides with the substack $\M^H(X)$. We thus forget about the tilde in the notation.
\end{lem}

\begin{proof}
    Given a map $S\to\M^{\sim H}(X)$, i.e. a commutative diagram as depicted before the proof, we can just forget about the maps to $B\widehat{H}$ to get a family of stable maps to $X$. Furthermore, if $\L\in H$, its pull-back to $X_H$ is trivial, and thus so is its pull-back to $\C_H$, which is fiberwise trivial. Therefore, so is its pull-back to $\C$.

Conversely, assume given a family of stable maps to $F:\C\to X\times S$ and assume that the pull-back of line bundles of $H$ are fiberwise  trivial. Using a section of $\C\to S$, we can consider the pull-back $S_H$  of $S\to X$ via the covering map $X_H\to X$ and trivialize line bundles from $H$ along this section as well, so that they are globally trivial on the family. This pull-back family provides a commutative diagram as displayed before the lemma.
\end{proof}

\begin{rem}
    Given a family of stable maps $F:\C\to X\times S$ fiberwise liftable to the $H$-covering, the fiber product $S_H$ actually encodes all the possible lifts to $X_H$, of which there are $|H|$, differing by the action of deck-automorphisms, leading to this equivariant map $\C_H\to S_H$.
\end{rem}


The $\M^H(X)$ form a \emph{lattice} (in the sense of \emph{posets}) in bijection with the lattice of subgroups of $\Pic^0(X)[\delta]$: for each inclusion $H_1\subset H_2$ we have a map $\M^{H_2}(X)\hookrightarrow\M^{H_1}(X)$.

\subsubsection{Relation with virtual class}
Let $\vir{\M^H(X)}$ be the virtual class of $\M^H(X)$. In particular we have that $\vir{\M^H(X)}=\sum_{K\supset H}\vir{\M_K(X)}$. As the expression only involves groups $K\supset H$, the system is triangular, and the data of $\M_K(X)$ is equivalent to the data of $\M^H(X)$. Concretely,  using the M\"obius function for finite abelian groups, one has:
$$\vir{\M^H(X)}=\sum_{K\supset H}\vir{\M_K(X)} \Longleftrightarrow \vir{\M_K(X)}=\sum_{H\supset K}\mu(H/K)\vir{\M^H(X)}.$$
Finally, the virtual class $\vir{\M^H(X)}$ can be computed using the following relation.

\begin{lem}
    The virtual dimension of $\M(X_H)$ and $\M(X)$ are the same. Furthermore, if we denote by $q:\M(X_H)\to\M^H(X)$ the composition map, we have
    $$q_*\vir{\M(X_H)} = |H|\cdot\vir{\M^H(X)}.$$
\end{lem}

\begin{proof}
    For the dimension statement, the only part that may change is the term $c_1(X)\cdot\beta$. However, if $\beta=\pi_{H\ast}\widetilde{\beta}$, we have
    $$c_1(X)\cdot\beta = c_1(X)\cdot\pi_{H\ast}\widetilde{\beta} = \pi_H^*c_1(X)\cdot\widetilde{\beta} = c_1(X_H)\cdot\widetilde{\beta},$$
    where we used that $\pi_H^* c_1(X)=c_1(X_H)$, since $X_H$ is a covering of $X$ and the tangent bundle pulls back to the tangent bundle. The equality then comes from the fact that the quotient map is \'etale.
\end{proof}

The virtual classes $\vir{\M(X_H)}$ being the virtual classes of some moduli space of stable maps, they are \textit{tautological}. Up to multiplication by $|H|$ and some linear combination, we just proved that the $\vir{\M_K(X)}$ are as well.

\subsubsection{Relation with evaluation maps} \label{sec-evaluation-map-correlated} The moduli space $\M(X_H)$ comes with an evaluation map to $X_H^{n+m}$. It sits in the following commutative diagram:
\bcd
\M(X_H) \arrow[r]\arrow[d] & \M(X) \arrow[d]\\
X_H^{n+m} \arrow[r] & X^{n+m}.
\ecd
Factoring through the action of $\widehat{H}$, acting diagonally on $X_H^{n+m}$, we get
\bcd
\M(X_H) \arrow[r,"|H|:1"',"q"]\arrow[d,"\ev_H"] & \M^H(X) \arrow[r]\arrow[d] & \M(X) \arrow[d,"\ev"]\\
X_H^{n+m} \arrow[r,"|H|:1"] & X_H^{n+m}/\widehat{H} \arrow[r] & X^{n+m}.
\ecd

In other words, quotienting by the $\widehat{H}$-action a family of fiberwise liftable stable maps has an evaluation map to $X_H^{n+m}/\widehat{H}$, since the class does not depend on the chosen lift anymore. We have the following lemma reducing computations on $\M^H(X)$ to computations on $\M(X_H)$. We denote by $\pi$ the forgetful map to $\overline{\M}_{g,n+m}$.

\begin{lem}
    For $\alpha\in \rmH^\bullet(\overline{\M}_{g,n+m},\QQ)$ and $\gamma\in \rmH^\bullet(X^{n+m},\QQ)$, we have
    $$\int_{\vir{\M^H(X)}} \pi^*\alpha\cup\ev^*\gamma = \frac{1}{|H|}\int_{\vir{\M(X_H)}} \pi^*\alpha\cup \ev_H^*(\pi_H^*\gamma) .$$
\end{lem}

\begin{proof}
    We use the commutative diagram, pulling-back the cohomology class $\gamma$ and the fact that the quotient maps are \'etale of degree $|H|$, and the push-pull formula.
\end{proof}

\begin{expl}
    Assume that $X=E$ is an elliptic curve, and take $n=2$. In particular, $\Pic^0(E)[2]\simeq\ZZ_2^2$ and it has five subgroups: $0$, $\Pic^0(E)[2]$, and the three subgroups generated by a unique non-trivial element, $H_1,H_2,H_3$.
    We thus have five components $\M_H$, and five coverings of $E$ of respective degree $1,2,2,2,4$. If $\beta=a[E]$, there is always a unique choice of $\widetilde{\beta}$ equal to $a[E_H]$ divided by the degree of the $H$-covering. The corresponding component is empty if $a$ is not divisible by the degree. We have the following splittings:
    \begin{itemize}
        \item $\M^{\{0\}}(E) = \M_{\{0\}}(E)\sqcup\M_{H_1}(E)\sqcup\M_{H_2}(E)\sqcup\M_{H_3}(E)\sqcup\M_{\ZZ_2^2}(E)=\M(E)$;
        \item $\M^{H_j}(E) = \M_{H_j}(E)\sqcup\M_{\ZZ_2^2}(E)$, for $j=1,2,3$;
        \item $\M^{\ZZ_2^2}(E) = \M_{\ZZ_2^2}(E)$.
    \end{itemize}
    The first and last identities are actually true for any groups, since $0$ is contained in any subgroup, and the only subgroup containing the whole group is the group itself.
\end{expl}

    \subsection{The stratification and an overview of its refinement}
    \label{sec-refined-strat-overview}

    \subsubsection{Graphs and Stratification of $\M(X)$}
The moduli space $\M(X)$ has a natural stratification; strata are indexed by so-called stable $X$-graphs \cite{janda2020double}. 

\begin{defi}\label{defi-graph}
(graph)

A graph $\Gamma$ of genus $g$ with $n+m$ ends consists in the following data
\[\Gamma = (\sfV , \sfH , \sfL , g : \sfV\to\mathbb Z_{\geqslant 0} , \rm{vert}:\sfH\to\sfV , \iota:\sfH\to\sfH)\]
where:
\begin{enumerate}[label=(\roman*)]
\item $\sfV$ is a vertex set with a genus function $g\colon\sfV \to\ZZ_{\geqslant 0}$,
\item $\sfH$ is the set of half-edges; it comes equipped with a vertex assignment $\rm{vert}:\sfH\to\sfV$ and an involution $\iota\colon\sfH\to\sfH$,
\item $\sfE$, the edge set, is defined by the 2-cycles of $\iota$ in $\sfH$ (self-edges at vertices are permitted),
\item $\sfL$, the set of legs, is defined by the fixed points of $\iota$ and is placed in bijective correspondence with $\{1,\dots,n+m\}$,
\item the pair $(\sfV,\sfE)$ defines a connected graph satisfying the genus condition
$$\sum_v g(v)+b_1(\Gamma)=g.$$
\end{enumerate}
An isomorphism between two graphs $\Gamma,\Gamma'$ consists of bijections $\sfV(\Gamma)\to\sfV(\Gamma')$ and $\sfH(\Gamma)\to\sfH(\Gamma')$ respecting the structures $\sfL, g,\rm{vert}$ and $\iota$. Let $\operatorname{Aut}(\Gamma)$ denote the automorphism group of $\Gamma$. The set of these graphs is denoted by $\GGG_{g,n+m}$.
\end{defi}

The set $\GGG_{g,n+m}$ is endowed with a natural operation: \emph{edge contraction}. If we contract and edge $e$ joining two vertices $v_1,v_2$ the resulting graph is obtaining deleting $e$ and replacing $v_1$, $v_2$ with a unique vertex $v$ of genus $g(v_1)+g(v_2)$. In case the edge $e$ is  a loop, we delete the edge and increase the vertex genus by $1$. 

\subsubsection*{ The cone stack associated to $\GGG_{g,n+m}$}
 To each graph $\Gamma$ one can associate a rational polyhedral cone  $\sigma_{\Gamma}:=\RR_{\geqslant 0}^{\sfE(\Gamma)};$ if $\Aut(\Gamma)\neq \left\{e\right\}$ then it acts on $\sigma_{\Gamma}$ permuting coordinates. Edge contractions correspond to face morphism of cones, i.e. if $\Gamma'$ is obtained from $\Gamma$ contracting some edges, then 
 $\sigma_{\Gamma'}\prec\sigma_{\Gamma}.$
 Taking the colimits of the diagram of cones $\left\{\sigma_{\Gamma}\right\}_{\Gamma\in\GGG_{g,n+m}}$ with maps induced by face morphism and automorphism we obtain a stacky cone complex $\Sigma_{g,n+m}$. This is a moduli stack of \emph{tropical pre-stable curves}  in the sense of \cite{cavalieri2020moduli}.

\begin{expl}
    Picking $g=2$ and $n=m=0$, the fan $\Sigma_{2,0}$ has two maximal cones of dimension $3$ corresponding to the theta graph and the dumbbell graph.
\end{expl}

\begin{defi}\label{defi-X-valued-graph}
(stable $X$-graph)

\begin{itemize}
    \item A stable $X$-graph is the data of a graph $\Gamma$, together with a class decoration function $\beta\colon \sfV(\Gamma)\to\rmH_2(X,\ZZ)$, satisfying the following stability condition: for each vertex $v$ such that $\beta_v=0$ we have $2g(v)-2+n(v)>0$, where $n(v)$ is the valency of $\Gamma$ at $v$ including both edges and legs. Automorphisms shall preserve the class function.
\item The total class of a stable $X$-graph is $\beta=\sum_{v\in\sfV} \beta_v$.
\item The set of stable $X$-graphs of total class $\beta$ is denoted by $\GGG_{g,m+n}(X,\beta)$, or shortly $\GGG(X)$ when the setting is clear.
\end{itemize}
\end{defi}

\subsubsection*{ The cone stack associated to $\GGG_{g,n+m}(X,\beta)$}

The set $\GGG_{g,n+m}(X,\beta)$ is also endowed with the edge contraction operation. When merging distinct vertices, we add their vertex classes.



The strata of $\M(X)$ are indexed by stable $X$-graphs. The boundary component associated to a stable $X$-graph $\Gamma\in\GGG(X)$ is obtained as follows. Let $\M_v(X,\beta_v):=\M_{g(v),n(v)}(X,\beta_v)$ be the moduli space of stable maps associated to a vertex, and let $\ev_v:\M_v(X,\beta_v)\to X^{\mathrm{vert}^{-1}(v)}$ be the associated evaluation map for adjacent half-edges which are not legs. The virtual stratum $\M_\Gamma(X)$ is the fiber product between $\prod_v\M_v(X,\beta_v)$ and the diagonal. 


By \cite{behrend1997gromov}), we have:
\begin{equation}\label{eq-splitting}
    \vir{\M_\Gamma(X)}=\sum_{(\beta_v)}\Delta^!\left(\prod\vir{\M_v(X,\beta_v)}\right)=\vir{\M(X)}\cap p^*\mathrm{PD}([j_*\mathfrak{M}_{\Gamma}]).
\end{equation}
where $\Delta$ is the diagonal  and $\mathfrak{M}_{\Gamma}$ is the closed substack in the (smooth) moduli stack of pre-stable curves. We refer the reader for example to the preliminary sections of \cite{bae2023pixton} and references therein for details about cohomology and Poincar\'e duality for smooth Artin stacks.

\subsubsection{Toward a refinement} In the previous section, we saw that $\M(X)$ splits into (possibly still disconnected) components $\M_K(X)$; we refer to this splitting as a \emph{refinement} of $\M(X).$  Our next goal is to explain how the boundary stratification described in the previous section interacts with the refinement.

\medskip

Using Lemma \ref{lem-pull-back-open-closed} and the above description of $\M_\Gamma(X)$, we can define several locally constant maps from $\M_\Gamma(X)$ to the poset of subgroups of $\Pic^0(X)[\delta]$:
\begin{itemize}
    \item the kernel $K$: $f\mapsto\operatorname{Ker}(f^*)$;
    \item for each component of the normalization, indexed by a given vertex $v$, the kernel $K_v$ of restriction to $C_v$: $f\mapsto\operatorname{Ker}(f\rvert_{C_v}^*)$;
    \item the subgroup $\widetilde{K}=\bigcap K_v\supset K$ of line bundles which pull-back to the trivial one component-wise.
\end{itemize}
The last two sheaves of subgroups are inherited from $\prod_v\M_v(X)$. By construction, the pull-back along $f\colon C\to X$ of line bundles in $\widetilde{K}$ lie in the subgroup $\rmH^1(\Gamma,\CC^*)\hookrightarrow\Pic^{[0]}(C)$.

\medskip

Unsurprisingly, $\M_\Gamma(X)$ is not connected. The components of the normalization $\M_v(X)$ are themselves disconnected, as they are disjoint  union of $\M_{v,K_v}(X)$ for various choices of $K_v$.

If we want to distinguish connected components of the boundary we should then definitely decorate the $X$-stable graphs with a choice of $(K_v)$ at the vertices. This leads us to introduce the notion of \textit{group decorated} stable $X$-graph (see Definition \ref{defi-group-decorated-graph}). Unfortunately, the decoration at each vertex is not sufficient to recover $K$; the fiber product of $\prod \M_{v,K_v}(X)$ over the diagonal is itself disconnected, and each connected component determines a subgroup $K\subset \bigcap K_v.$

\begin{rem}
This phenomena of disconnected fiber product was already encountered by the authors in \cite[Lemma~4.5,Remark~4.6]{blomme2024correlated} in the study of the correlated degeneration formula.
\end{rem}
\medskip

To recover the missing piece of information we introduce in the next subsection \textit{twisted diagonals}.

\begin{rem}
The reader may be scared that to compute with these classes, one needs to translate from classes $\vir{\M_{v,K_v}(X)}$ to classes $\vir{\M_v^{H_v}(X)}$ and deal with some crazy linear combination between them. Fortunately, and similarly to \cite{blomme2024correlated}, such a refinement is not necessary in the sense that we overdecomposed the moduli space, and it is only necessary to consider $\prod\M_v^H(X)$ for a common $H$ for all vertices.
\end{rem}

\subsection{Twisted diagonals}

This section is devoted to the introduction of twisted diagonals, used in the construction of refined boundary strata. In a nutshell, \textit{twisted diagonals} are preimages of the diagonals in $X$ by the covering maps obtained using the $H$-covering.

\subsubsection{Construction} Let $\Gamma\in\GGG_{g,n+m}(X,\beta)$ be a stable $X$-valued graph. We consider the following commutative diagram, where every square is a fiber product:

\begin{center}
\adjustbox{scale=0.75}{
\begin{tikzcd}
\overline{\operatorname{Diag}}_H \arrow[rr]\arrow[rd]\arrow[rdd]&  & \operatorname{Diag}_H\arrow[rd]\arrow[ddl] & \\
 & X_H^{\mathsf{H}(\Gamma)}\arrow[rr,crossing over,"q_H^\Gamma"]\arrow[rdd,crossing over,"\widehat{\pi}_H^\Gamma"] &  & X_H^{\mathsf{H}(\Gamma)}/\widehat{H}^{\sfV(\Gamma)} \arrow[ldd,"\pi_H^\Gamma"] \\
  & \operatorname{Diag}_0 \arrow[rd,hook] & & \\
 & & X^{\mathsf{H}(\Gamma)} & \\
\end{tikzcd}.}
\end{center}

\begin{itemize}
    \item The maps on the front face are coverings of the product $X^{\mathsf{H}(\Gamma)}$ induced by $H$-covering $X_H\to X$; an element $\hat{h}_v\in\widehat{H}^{\sfV(\Gamma)}$ corresponding to a vertex $v\in \sfV(\Gamma)$ (we mean that on all the other factors we have the identity) acts by deck-automorphism on all copies $X^e_H$ with $e$ an half-edge adjacent to $v$. These coverings have respective degrees
    $$\deg\widehat{\pi}_H^\Gamma = |H|^{|\mathsf{H}(\Gamma)|},\ \deg\pi_H^\Gamma = |H|^{|\mathsf{H}(\Gamma)|-|\sfV(\Gamma)|},\text{ and }\deg q_H^\Gamma=|H|^{|\sfV(\Gamma)|}.$$
    \item $\operatorname{Diag}_0=X^{\mathsf{E}(\Gamma)}$ is the diagonal in $X^{\mathsf{H}(\Gamma)}$, where an edge maps to corresponding half-edges.
    \item $\overline{\operatorname{Diag}}_H$ is the pull-back of the diagonal by the first covering,
    \item $\operatorname{Diag}_H$ is the pull-back of the diagonal by the second covering, or alternatively the quotient of $\overline{\operatorname{Diag}}_H$ by the action of $\widehat{H}^{\sfV(\Gamma)}$.
\end{itemize}

Using the automorphism group of the $H$-covering, we can see that pull-back diagonals are disconnected.

\begin{lem}
    We have the following
    \begin{enumerate}
        \item $\overline{\operatorname{Diag}}_H$ has $|H|^{|\mathsf{E}(\Gamma)|}$ connected components indexed by $C^1(\Gamma,\widehat{H})$, denoted by $\overline{\operatorname{Diag}}_{H,\psi}$. For any $\psi\in C^1(\Gamma,\widehat{H})$, the covering $\overline{\operatorname{Diag}}_{H,\psi}\to\operatorname{Diag}_0$ has degree $|H|^{|\mathsf{E}(\Gamma)|}$.
        \item Let $\partial:C^0(\Gamma,\widehat{H})\to C^1(\Gamma,\widehat{H})$ be the coboundary map. If $\chi\in C^0(\Gamma,\widehat{H})=\widehat{H}^{\sfV(\Gamma)}$, we have $\chi\cdot\overline{\operatorname{Diag}}_{H,\psi}=\overline{\operatorname{Diag}}_{H,\psi+\partial\chi}$.
        \item $\operatorname{Diag}_H$ has $|H|^{b_1(\Gamma)}$ connected components indexed by $\rmH^1(\Gamma,\widehat{H})$. For any $\varphi$, the covering $\operatorname{Diag}_{H,\varphi}\to\operatorname{Diag}_0$ has degree $|H|^{|\mathsf{E}(\Gamma)|-1}$ and each map $\overline{\operatorname{Diag}}_{H,\psi}\to\operatorname{Diag}_{H,[\psi]}$ has degree $|H|$.
    \end{enumerate}
\end{lem}

\begin{proof}
    \begin{enumerate}
        \item The covering $X_{H}\to X$ is Galois. Therefore,  an element $(y_h)\in X_H^{\mathsf{H}(\Gamma)}$ is mapped by $\widehat{\pi}_H^\Gamma$  into the image of diagonal of $X$ if and only if for each oriented edge $e=(h,h')$ there exists a deck-automorphism $\varphi_{h'}^{h}$ such that $y_{h'}=\varphi_{h'}^{h}(y_h)$. If one reverses the orientation, we naturally have $(\varphi_{h'}^{h})^{-1}=\varphi_{h}^{h'}$. Notice that, since the action of the deck automorphism is free, $\varphi_{h'}^{h}$ is unique.
 The pull-back of the diagonal is thus the product of graphs of deck automorphisms, which are disjoint since deck automorphisms are fixed-point free. Therefore, the choice of a diagonal component is precisely indexed by this choice of $(\varphi_e)$, which may be seen as a cochain $C^1(\Gamma,\widehat{H})$.
        \item A straightforward computation shows that if $\chi\in C^0(\Gamma,\widehat{H})=\widehat{H}^{\sfV(\Gamma)}$ and $y\in\overline{\operatorname{Diag}}_{H,\psi}$ for some $\psi$, then $\chi\cdot y\in\overline{\operatorname{Diag}}_{H,\psi+\delta\chi}$. Therefore, the stabilizer of the component is $\rmH^0(\Gamma,\widehat{H})$ which is isomorphic to $\widehat{H}$. Furthermore, there are $|\widehat{H}^{|\sfV(\Gamma)|-1}$ components identified by the action.
        \item It follows from the previous point that connected components are indexed by $\rmH^1(\Gamma,\widehat{H})$. The degree has also been computed as the cardinality of the component stabilizer. Alternatively, the total degree $|\widehat{H}|^{|\sfV(\Gamma)|}$ is preserved and there are $|\widehat{H}|^{|\sfV(\Gamma)|-1}$ components, so that the degree for each of them is $|H|$.
    \end{enumerate}
\end{proof}

Connected components of $\overline{\operatorname{Diag}}_H$ and $\operatorname{Diag}_H$ are called \textit{twisted diagonals}. Having introduced twisted diagonals, we can refine the above commutative diagram. We thus have the following diagrams for $\varphi\in \rmH^1(\Gamma,\widehat{H})$ and $\psi\in C^1(\Gamma,\widehat{H})$:
\begin{center}
\adjustbox{scale=0.7}{
\begin{tikzcd}
\overline{\operatorname{Diag}}_{H,\psi} \arrow[rr]\arrow[rd]\arrow[rdd]&  & \operatorname{Diag}_{H,[\psi]}\arrow[rd]\arrow[ddl] & \\
 & X_H^{\mathsf{H}(\Gamma)}\arrow[rr,crossing over]\arrow[rdd,crossing over] &  & X_H^{\mathsf{H}(\Gamma)}/\widehat{H}^{\sfV(\Gamma)} \arrow[ldd] \\
  & \operatorname{Diag}_0 \arrow[rd,hook] & & \\
 & & X^{\mathsf{H}(\Gamma)} & \\
\end{tikzcd},
\begin{tikzcd}
\bigsqcup_{[\psi]=\varphi}\overline{\operatorname{Diag}}_{H,\psi} \arrow[rr]\arrow[rd]\arrow[rdd]&  & \operatorname{Diag}_{H,\varphi}\arrow[rd]\arrow[ddl] & \\
 & X_H^{\mathsf{H}(\Gamma)}\arrow[rr,crossing over]\arrow[rdd,crossing over] &  & X_H^{\mathsf{H}(\Gamma)}/\widehat{H}^{\sfV(\Gamma)} \arrow[ldd] \\
  & \operatorname{Diag}_0 \arrow[rd,hook] & & \\
 & & X^{\mathsf{H}(\Gamma)} & \\
\end{tikzcd}.
}
\end{center}

Let us stress that the squares are commutative but clearly no longer fiber product, except the top square in the right diagram.

\begin{rem}
Elements of $\rmH^1(\Gamma,\widehat{H})$ may  be interpreted as pairings $H\otimes\rmH_1(\Gamma,\ZZ)\to\mu_\delta\subset\CC^*$. This will be the  point of view we'll take in the next section. Such pairings are called \emph{monodromy pairings} for reasons explained below. 
\end{rem}

\subsubsection{Compatibility by change of subgroup} We finish this section proving a compatibility statement about refinements indexed by two different groups $H_1\subset H_2$ yielding a covering $X_{H_2}\to X_{H_1}$ of degree $|H_2/H_1|$. We denote by $\iota$ the inclusion and by $\iota^*:\widehat{H}_2\to\widehat{H}_1$ the restriction between Pontrjagin duals, as well as the morphisms it induces between $C^1(\Gamma,\widehat{H}_j)$ and $\rmH^1(\Gamma,\widehat{H}_j)$. Furthermore, we have $\ker\iota^*=\widehat{H_2/H_1}$, which is the automorphism group of the covering $X_{H_2}\to X_{H_1}$.

\begin{lem}
    Assume $\iota:H_1\hookrightarrow H_2$ are subgroups.
    \begin{enumerate}
        \item We have maps $\overline{\operatorname{Diag}}_{H_2,\psi_2}\to\overline{\operatorname{Diag}}_{H_1,\iota^*\psi_2}$ of degree $|H_2/H_1|^{|\mathsf{E}(\Gamma)|}$. For given $\psi_1$, there are $|H_2/H_1|^{|\mathsf{E}(\Gamma)|}$ elements $\psi_2$ with $\iota^*\psi_2=\psi_1$.
        \item We have maps $\operatorname{Diag}_{H_2,\varphi_2}\to\operatorname{Diag}_{H_1,\iota^*\varphi_2}$ of degree $|H_2/H_1|^{|\mathsf{E}(\Gamma)|-1}$. For given $\varphi_1$, there are $|H_2/H_1|^{b_1(\Gamma)}$ elements $\varphi_2$ with $\iota^*\varphi_2=\varphi_1$.
    \end{enumerate}
\end{lem}

\begin{proof}
    Finding the maps is straightforward, and computing the degrees follows from the degrees of maps to $\operatorname{Diag}_0$.
\end{proof}

\subsection{Boundary components and  coverings}
\label{sec-monodromy-graphs-stratif-MK}

We can now describe the interaction between the stratification and the refinement, and doing so we describe the induced stratification on the $\M_K(X)$. 
To index the boundary strata of $\M_K(X)$, we need new kind of decorated graphs which we now introduce.

\subsubsection{Group decorated graphs}

\begin{defi}\label{defi-group-decorated-graph}
(group decorated graph)
\begin{itemize}
    \item A \textit{group decorated} stable $X$-graph (or shortly a \textit{group decorated graph}) is the data of a stable $X$-graph $\Gamma$ together with the data of a subgroup $K_v\subset\Pic^0(X)[\delta]$ at each vertex.
    \item Given a group decorated graph $\Gamma$, its \textit{kernel} is the subgroup $\widetilde{K}=\bigcap K_v$.
    \item A \textit{global monodromy} on a group decorated graph is the data of a pairing
    $$\varphi\colon \widetilde{K}\otimes\rmH_1(\Gamma,\ZZ)\longrightarrow \mu_\delta\subset\CC^*.$$
    We call \textit{core}  the subgroup $K=\{\L\in\widetilde{K} \text{ s.t. }\varphi(\L,-)=1\}$.
    \item A group decorated graph together with a choice of kernel, core and monodromy is called a \textit{monodromy graph}. The set of monodromy graphs is denoted by $\GGG^\gp_{g,n+m}(X,\beta)$.
\end{itemize}
\end{defi}

As advertised, this graph decoration is meant to encompass the information about the pull-back to curves with dual graph $\Gamma$ of $\delta$-torsion line bundles: the $K_v$ is the subgroup of $\delta$-torsion line bundles that pull-back trivially on the component $C_v\subset C$ of such a curve $C$. In particular, the kernel $\widetilde{K}$ corresponds to line bundles pulling back to the trivial one on the normalization of $C$, i.e. line bundles whose pull-back  lie in the subgroup $\rmH^1(\Gamma,\CC^*)\subseteq\Pic^0(C)$. The  data of the monodromy pairing $\varphi$ is equivalent  to the data of the pull-back  $f^*\L$ for $\L$ restricting trivially component-wise. The core $K\subset\widetilde{K}$ corresponds to line bundles that, on the different components of the boundary $\M_{\Gamma}(X)$, pull-back to the trivial one not only componentwise but globally.

\begin{rem}
\label{rem-refined-monodromy-graphs}
    We also have a notion of \textit{local monodromy}. Let $\Gamma$ be a group decorated graph. If $A\subset\sfV(\Gamma)$ is a subset and $\Gamma_A$ the subgraph of $\Gamma$ with vertices of $A$ and edges between them, line bundles from $K_A=\bigcap_{v\in A}K_v$ restrict trivially on the components of $v\in A$. The local monodromy is the data a pairing
    $$\varphi_A\colon \left(\bigcap_{v\in A} K_v\right)\otimes \rmH_1(\Gamma_A,\ZZ)\longrightarrow\mu_\delta\subset\CC^*.$$
    These should satisfy thus the following compatibility condition: if $\gamma$ is a loop belonging to two subgraphs $\Gamma_A$ and $\Gamma_B$, then for $\L\in K_A\cap K_B=\bigcap_{A\cup B}K_v$ we have
    $$\varphi_A(\L,\gamma) = \varphi_B(\L,\gamma).$$
    Fixing such a family of $(\varphi_A)_A$ is called a monodromy on the group decorated graph. The global monodromy is merely the pairing for $A=\sfV(\Gamma)$.
    
    As we only care about the core, we do not need this refined notion of monodromy graph fixing a group decoration and a family of pairings $(\varphi_A)$ except when defining edge contraction for defining a cone stack, which is our next step.
\end{rem}

\subsubsection*{ The cone stack associated to $\GGG_{g,n+m}^{\gp,K}(X,\beta)$}

The set of monodromy graphs $\GGG_{g,n+m}^\gp(X,\beta)$, shortly denoted by $\GGG^\gp(X)$ is also endowed with an \emph{edge contraction operation,} defined as follows: let $\Gamma$ be a monodromy graph with kernel $\widetilde{K}$, global monodromy $\varphi$ and let $e$ be an edge. The kernels and pairing are modified as follows.
\begin{itemize}
    \item If $e$ has distinct ends, we intersect the adjacent $K_v$, so that the kernel $\widetilde{K}$ does not change nor the pairing $\varphi$ since we have an isomorphism $\rmH_1(\Gamma,\ZZ)\to\rmH_1(\Gamma/e,\ZZ)$.
    \item If $e$ is a loop, $K_v$ is replaced by the orthogonal of $\gamma$ in $K_v$ in the setting of Remark \ref{rem-refined-monodromy-graphs}, so that the kernel is replaced by
    $$\widetilde{K}_e = \{\L\in\widetilde{K} \text{ s.t. }\varphi(\L,e)=1\}.$$
    Restricting the first argument to $\widetilde{K}_e$, the pairing factors through $\rmH_1(\Gamma,\ZZ)\to\rmH_1(\Gamma/e,\ZZ)$.
\end{itemize}

\emph{Crucially,} the core is preserved by the edge contraction operation. 
Therefore, if makes sense to consider edge contraction restricted to the subset $\GGG^{\gp,K}_{g,n+m}(X,\beta)$ of monodromy graphs with a given core $K$.
Then, as already done for graph and class $X$-decorated strata graph, we get a cone complex $\Sigma^{\gp,K}_{g,n+m}(X,\beta)$ (or shortly $\Sigma_K(X)$) taking the colimit (along maps induced by edge contraction and authomorphisms) of the collection of cones indexed by monodromy graphs with core $K$. In particular, it means the global monodromy induces an injective morphism $\varphi\colon\widetilde{K}/K\hookrightarrow\rmH^1(\Gamma,\CC^*)$.


    \subsubsection{Refinement of the stratification}
Let $\Gamma$ be a monodromy graph with group decoration $(K_v)$ and global monodromy $\varphi\colon\widetilde{K}\otimes\rmH_1(\Gamma,\ZZ)\to\CC^*$. We now explain how to construct an associated stratum.

 As $\widetilde{K}\subset K_v$, the maps of $\M_{v,K_v}(X)$ can be lifted to the $\widetilde{K}$-covering (up to automorphism, see Section~\ref{sec-H-covering}), and choosing $H=\widetilde{K}$ we get, as explained in Section~\ref{sec-evaluation-map-correlated}, a refined evaluation map
$$\ev:\prod_v\M_{v,K_v}(X)\to X_{\widetilde{K}}^{\mathsf{H}(\Gamma)}/\widehat{\widetilde{K}}^{\sfV(\Gamma)}.$$
We then make the fiber product with the twisted diagonal determined by the global monodromy $\varphi$:
\bcd
\M_{\Gamma,(K_v),\varphi}(X) \arrow[r]\arrow[d] & \prod\M_{v,K_v}(X) \arrow[d] \\
\operatorname{Diag}_{\widetilde{K},\varphi} \arrow[r,hook,"\Delta_{\widetilde{K},\varphi}"] & X_{\widetilde{K}}^{\mathsf{H}(\Gamma)}/\widehat{\widetilde{K}}^{\sfV(\Gamma)}.
\ecd
Using Gysin pull-back, each one is endowed with a virtual class $\Delta_{\widetilde{K},\varphi}^!(\prod\vir{\M_{v,K_v}(X)})$.

\begin{lem}\label{lem-kernel-restriction-bdry-strata}
    We have that $\operatorname{Ker}(f^*)$ coincides with $\left\{\L\;| \; 1=\varphi(\L,-)\colon \rm{H}_1(\Gamma,\mathbb Z)\to\mu_{\delta}\right\}$, i.e. the core of the monodromy graph. In particular, $\M_{\Gamma,(K_v),\widetilde{K},\varphi}$ lies in the boundary of $\M_K$.
\end{lem}

\begin{proof}
    If a line bundle restricts trivially, in particular it restricts trivially componentwise, and thus belongs to $\widetilde{K}$. Unraveling the definition of $\varphi$, if $\L\in\widetilde{K}$, $\varphi(\L)\in \rmH^1(\Gamma,\mu_n)$ actually provides the monodromy of the pull-back of $\L$ along loops of $\Gamma$. Indeed, considering a partial normalisation so that the dual graph is a tree, maps can be lifted to $X_{\widetilde{K}}$, and fibers associated to points resulting from normalising nodes differ by the action of a deck automorphism, yielding the result.
\end{proof}

We are  actually not interested in keeping track in the boundary of the vertex refinement $K_v$;
we thus define the stratum $\M_{\Gamma,\widetilde{K},\varphi}(X)$ associated to a monodromy graph without a choice of group decoration:
$$\M_{\Gamma,\widetilde{K},\varphi}(X)=\bigsqcup_{(K_v):\bigcap K_v=\widetilde{K}} \M_{\Gamma,(K_v),\varphi}.$$

In the end, we get the stratification of $\M_K(X)$ with strata of indexed by \emph{monodromy graphs,} i.e. a stable $X$-graph with a choice of kernel $\widetilde{K}$ and global monodromy $\varphi\colon\widetilde{K}\otimes\rmH_1(\Gamma,\ZZ)\to\CC^*$ with core $K$. 

The dual intersection complex with respect  to such a stratification, keeping track of the dual graph automorphisms, is the stacky cone  complex  $\Sigma_K(X)$ described above.

\subsubsection{Relation to moduli spaces of stable maps to coverings}

By definition, the virtual class of the stratum $\M_{\Gamma,\widetilde{K},\varphi}(X)$ is a sum over the possible group decorations $(K_v)$. Therefore, it may seem this requires excruciating computations. In this section, we prove that they can be drastically reduced to computations for moduli spaces of stable maps to a $H$-covering of $X$, where $H$ is common to all vertices of $\Gamma$.

\medskip

Consider the fiber product between the product of substacks $\M_v^H(X)$ with a common choice of $H$ for all vertices, and the twisted diagonal indexed by $\varphi:H\to \rmH^1(\Gamma,\mu_n)$:
\bcd
\M_{\Gamma,\varphi}^H \arrow[r]\arrow[d] & \prod\M_v^H(X) \arrow[d] \\
\operatorname{Diag}_{H,\varphi} \arrow[r,hook,"\Delta_{H,\varphi}"] & X_H^{\mathsf{H}(\Gamma)}/\widehat{H}^{\sfV(\Gamma)}.
\ecd
It is endowed with a virtual fundamental class coming from Gysin pull-back.

We now relate classes $\Delta_{H,\varphi}^!(\prod\vir{\M_v^H(X)})$ to classes $\sum_{(K_v):\bigcap K_v=\widetilde{K}}\Delta_{\widetilde{K},\psi}^!(\prod\vir{\M_{v,K_v}(X)})$. All computation take place in $\rmH_\bullet(\M_\Gamma(X),\QQ)$.

\begin{prop}
    We have the following relation expressing classes of stable maps liftable to the $H$-covering in terms of classes of refined strata:
    $$\Delta_{H,\varphi}^!(\prod\vir{\M_v^H(X)}) =
    \sum_{\widetilde{K}\supset H}\sum_{\psi:\psi|_H=\varphi}\Delta_{\widetilde{K},\psi}^!\left(\sum_{(K_v):\bigcap K_v=\widetilde{K}}\prod\vir{\M_{v,K_v}(X)}\right).$$
    Furthermore, it is possible to conversely express $\vir{\M_{\Gamma,\widetilde{K},\varphi}}$ as linear combinations of the classes $\Delta_{H,\varphi}^!(\prod\vir{\M_v^H(X)})$.
\end{prop}

\begin{proof}
    We first split $\vir{\M_v^H(X)}=\sum_{K_v\supset H}\vir{\M_{v,K_v}(X)}$ at each vertex. We then sort out according to the value of $\bigcap K_v=\widetilde{K}$:
    $$\Delta_{H,\varphi}^!(\prod\vir{\M_v^H(X)}) =
    \sum_{\widetilde{K}\supset H}\Delta_{H,\varphi}^!\left(\sum_{(K_v):\bigcap K_v=\widetilde{K}}\prod\vir{\M_{v,K_v}(X)}\right).$$
    To finish, we need to switch the $\Delta_{H,\varphi}^!$ on the right hand-side into a $\Delta_{\widetilde{K},\psi}$. As the preimage of the twisted diagonal indexed by $\varphi:H\to \rmH^1(\Gamma,\mu_n)$ is given by all extensions $\psi:\widetilde{K}\to \rmH^1(\Gamma,\mu_n)$ with $\psi|_H=\varphi$, we get the expected result. Since the sum on the right hand-side only involves $\widetilde{K}\supset H$, the system of equations is triangular, and therefore invertible since the coefficient is $1$ when $\widetilde{K}=H$.
\end{proof}
\begin{rem}
A very similar computation appeared in \cite[Section~4]{blomme2024correlated} where the authors proved a refinement of the degeneration formula keeping track of the correlators.
\end{rem}

We have the following commutative diagram where every square is a fiber product:

\begin{center}
\adjustbox{scale=0.75}{
\begin{tikzcd}
\M_\Gamma(X_H) \arrow[rr]\arrow[rd]\arrow[dd]&  & \M_\Gamma^H(X) \arrow[rd]\arrow[dd] & \\
 & \prod\M_v(X_H) \arrow[rr,crossing over] &  & \prod\M_v^H(X) \arrow[dd] \\
\overline{\operatorname{Diag}}_H \arrow[rr]\arrow[rd]\arrow[rdd]&  & \operatorname{Diag}_H\arrow[rd]\arrow[ddl] & \\
 & X_H^{\mathsf{H}(\Gamma)}\arrow[from=uu,crossing over]\arrow[rr,crossing over]\arrow[rdd,crossing over] &  & X_H^{\mathsf{H}(\Gamma)}/\widehat{H}^{\sfV(\Gamma)} \arrow[ldd] \\
  & \operatorname{Diag}_0 \arrow[rd,hook] & & \\
 & & X^{\mathsf{H}(\Gamma)} & \\
\end{tikzcd}.}
\end{center}
\begin{itemize}
    \item On the left faces we have the moduli space of (twisted) stable maps to the $H$-covering, constructed as fiber product between the evaluation map and the twisted diagonals.
    \item On the right faces we get the corresponding components for the moduli space of stable maps to $X$, i.e. after the quotient by the deck-automorphisms.
\end{itemize}

\begin{rem}
    The right faces describe what actually happens in the moduli space of stable maps to $X$. The left faces relate it to stable maps to the $H$-covering. All maps going from left side to right side are coverings of degree $|H|^{|\sfV(\Gamma)|}$. Those on the back have degree $|H|$ when restricting to a specific component.
\end{rem}

%% file: sec-moduli-roots.tex
\section{Moduli of logarithmic roots, logarithmic modifications and (virtual) stratifications }
\label{sec-various-moduli-of-curves-including-spin}

In this section we briefly review: the theory of stability conditions for  line bundles on a logarithmic curve $ C\to S$, following \cite[Section~4]{holmes2025logDR}; the construction of moduli of logarithmic roots of a line bundle on a logarithmic curve following \cite{holmes2023root}. This will include a discussion of logarithmic (virtual) stratifications for algebraic stack with log structure and for certain logarithmic modifications.

\subsection{Logarithmic stratifications and logarithmic modifications}
We recall that for $S$ an fs log scheme, choosing charts in the sense of \cite[II.2]{ogus} we get, \'etale locally, morphisms to Artin cones, i.e. algebraic stack of the form $[U_{\sigma}/T_{\sigma}]$ with $U_{\sigma}$ an affine toric variety and $T_{\sigma}$ its dense torus.  These glue together to a  strict logarithmic morphism to an algebraic stack with logarithmic structure
\[t\colon S\to\mathcal A_{\Sigma},\] 
which we call \emph{tropicalization} of $S$. 

The stack $\mathcal A_{\Sigma}$ is an \emph{Artin fan} in the sense of \cite{abramovich2014comparison, cavalieri2020moduli,olsson2003logarithmic}; namely an algebraic stack with logarithmic structure which admits an \'etale cover given by Artin cones. 
By \cite{cavalieri2020moduli}, the category of Artin fans is equivalent to the category of \emph{cone stacks}, these are a generalization  of cone complexes where multiple face maps between two cones are allowed and self maps other than the identity are allowed.

Because of this result, in what follows we  \emph{do not distinguish}
the algebraic stack $\mathcal A_{\Sigma}$ and the combinatorial object
$\Sigma$ and \emph{refer to both as the tropicalization } of $S$ and often write 
\[S\to\Sigma.\]

\begin{rem}
    We remark that, following \cite{maulik2023logarithmic}, (see also \cite[Remark~5]{holmes2025logDR} our definition of tropicalization is slightly more general than the original one of Abramovich-Wise \cite{abramovich2014comparison} as we simply required the map $S\to\mathcal A_{\Sigma}$ to be strict, but we do not insist that it is also initial among the factorizations $S\to\mathcal Log$ where $\mathcal Log$ in Olsson's stack of log structures \cite{olsson2003logarithmic}. 
\end{rem}

\subsubsection{} Let $S$ a fs log scheme and $t\colon S\to \Sigma$ a strict map to an Artin fan. This induces a logarithmic stratification of $S,$ whose strata correspond to (stacky) cones $\sigma\subseteq \Sigma$ of the cone stack. Given such a cone $\sigma,$ we will denote by $S_{\sigma}\subseteq S$ the locally closed stratum corresponding to $BT_{\sigma}\to\mathcal A_{\Sigma}$ and by $\bar{S}_{\sigma}$ its closure. 
The latter is usually not a closed strata in $S$ but rather its normalization. 

\begin{example}

Let $\mathfrak{M}_{g,n+m}$ the Artin stack of pre-stable genus $g$ curves endowed with the boundary log structure.
Let $\Sigma(\mathfrak{M})$ be a cone stack whose stacky cones $\sigma_{\Gamma}$  are indexed by the set $\GGG_{g,n+m}$ of (unstable) dual graphs $\Gamma$. There are two natural choices for $\sigma_{\Gamma}$, differing by the automorphisms group of the cone, namely: $\sigma_{\Gamma}=[\mathbb R^{|E(\Gamma)|}_{\geq 0}/\mathrm{Aut}(\Gamma)]$ or $\sigma_{\Gamma}=[\mathbb R^{|E(\Gamma)|}_{\geq 0}/G_{\Gamma}]$ for $G_{\Gamma}$ the monodromy group of the stratum as in \cite[Definition~7]{holmes2025logDR}. The two differ by $\Aut^{\text{loop}}(\Gamma):$ the subgroup of automorphisms acting as the identity of the edges, i.e. only reversing the orientation of loops.

The logarithmic strata, forgetting the action on the cones are 
\[S_{\sigma_{\Gamma}}=\prod_v\mathfrak{M}^{\rm{sm}}_{g(v),n(v)}\text{ and } \bar{S}_{\sigma_{\Gamma}}=\prod_v\mathfrak{M}_{g(v),n(v)};\]
recalling the group acting on the cones we have the quotient of the latter by $\mathrm{Aut}(\Gamma)$ or by $G_{\Gamma}$ respectively.


\medskip

On the cone $\sigma_{\Gamma}$ we have a family of tropical curves $\mathfrak{u}\colon\mathtt{C}_{\sigma_{\Gamma}}\rightarrow\sigma_{\Gamma}.$ 
The data of the morphism of cone complexes $\mathtt{C}_{\sigma_{\Gamma}}\to\sigma_{\Gamma}$ can be encoded in the dual graph underlying $\mathtt{C}_s=\mathfrak{u}^{-1}(s)$ for $s$ a point in the relative interior of $\sigma_{\Gamma}$ together with a \emph{length function} $\ell\colon \sfE(\mathtt{C}_s)\to \bar{M}_{\mathfrak{M}_{\eta}}$ for $\eta$ the generic point of the stratum.
\end{example}

\subsubsection{} Given a cone stack $\Sigma$ we say that $\Sigma^\diamond\to\Sigma$ is a subdivision if for each cone $\sigma\to\Sigma$ we have that $\Sigma^\diamond\times_{\Sigma}\sigma\to\sigma$ is a subdivision in the classical sense of toric geometry.

We say that a morphism of logarithmic schemes $S^\diamond\xrightarrow{p} S$ is a \emph{subdivision} or \emph{logarithmic modification} if 
$S^\diamond=S\times_{\Sigma }\Sigma^\diamond.$ We also consider logarithmic modifications  $\widetilde S\xrightarrow{p} S$  obtained  from $S$ by a root construction, see for example \cite[Section~2]{holmes2023root}. Locally at the level of cones, this correspond only  to a refinement of the lattice structure, while for an actual subdivision there are new cones in $\Sigma^\diamond.$




\subsection{Stability conditions}

We follow \cite[Section~4]{holmes2025logDR}. Let $ C\xrightarrow{\pi}S$ be a family of  vertical log smooth curves and $\mathcal L$ a line bundle on $\mathcal C$ of total degree $d$ along the fibers: $(C, S, L)$ is the data of a morphism
\[S\to\mathfrak{P}ic_{g,n,d}\]
where $\mathfrak{P}ic_{g,n,d}$ is the universal Picard stack over the moduli space of pre-stable curves. Obviously, we can also think of $\L$ as a log line bundle, i.e. a section $S\to\LogPic^d_{C/S}.$

\medskip

As in \cite[Definition~28]{holmes2025logDR}: a stability condition $\vartheta$ of degree $d$ on $\pi$ is a function

\[\vartheta\colon \sfV(\Gamma_s)\to\mathbb Q,\]

such that for each $s$ $\sum_{v\in V(\Gamma_s)}\vartheta(v)=d$ and the function $\vartheta$ is compatible with edge contractions $\Gamma_s\to\Gamma_t.$

\medskip

Given a quasi-stable model $\widehat{C}\to C\times_ST\to T$ of $C$ over a log scheme $T\to S$ (i.e. $\widehat{C}_t\to C_t$ is a partial destabilization obtained adding at most one semi-stable $\mathbb P^1$ at every node, which tropically means subdividing each edge at most once), $\vartheta$ lifts to a stability condition of $\widehat{C}$ setting its value to be zero at the new vertices.

\medskip

For a fixed stability condition $\vartheta$ we say that a line bundle $\L$ on a quasi-stable model $\widehat{C}\to C\times_ST\to T$ is $\vartheta$-semi-stable if for each fiber $t$:
\begin{itemize}
    \item the degree of $\mathcal L\rvert_{\hat{C}_t}$ on the exceptional $\mathbb P^1$'s of the semi-stabilization  $\widehat{C}_t\to C_t$ is $1$ (\emph{admissibility})
    \item for each $G\subseteq V(\widehat{\Gamma}_t)$ the following inequality is satisfied
    \[\sum_{v\in G}\vartheta(v)-\frac{\mathrm{val}(G)}{2}\leq \sum_{v\in G} \underline{\deg}(L\rvert_{\widehat{C}_t})(v)\leq \sum_{v\in G}\vartheta(v)+\frac{\mathrm{val}(G)}{2} ,\]
    where the multidegree $\underline{\deg}(L\rvert_{\widehat{C}_t})$ defines a tropical divisor supported on $\widehat{\Gamma}_t.$
\end{itemize}
We say that $\L$ is $\vartheta$-stable if the inequalities are strict. The stability condition is called \emph{non-degenerate} if every $\vartheta$-semi-stable line bundle is stable. Moreover $\vartheta$ is said \emph{small} if the trivial bundle is $\vartheta$-semi-stable.

If the curve $C\to S$ has a section, there exists a (small) non-degenerate stability condition (see \cite{kass2019stability} and references therein).

\medskip

We concentrate on the case $d=0$, and $\vartheta$ a small non-degenerate stability condition.
As explained in \cite[Section~4.1]{holmes2025logDR}, one can use $\vartheta$ to define a \emph{fine compactified Jacobian}
$\mathcal P_{\pi}^\vartheta\supset \Pic^{[0]}_{C/S}$
parametrizing $\vartheta$-stable line bundle (together with a rigidification along a smooth section) on quasi-stable models. The space $\mathcal P_{\pi}^\vartheta$ is a proper algebraic space over $S$. Furthermore, it admits a logarithimic \'etale map $\mathcal P_{\pi}^\vartheta\to\LogPic^0_{C/S}$ \cite[Remark~41]{holmes2025logDR}. 

\medskip

Let $(C\to S,\L)$ be as before, with $L$ of total degree zero along the fibers and $C\to S$ admitting at list a section, so that there exists a non-degenerate small stability condition $\vartheta$. 
We recall:

\begin{theo}\cite[Theorem~35]{holmes2025logDR}
There exists: a logarithmic modification $S_\L^{\vartheta}\to S,$ (induced by a subdivision of the tropicalization $\Sigma(S)$, see \cite[Section~2.4]{holmes2025logDR}, or also \cite[Section~1]{maulik2023logarithmic}) a quasi-stable model $C^\vartheta\to C\times_S S_\L^{\vartheta}$ and a sPL function $\alpha\in \mathrm{H}^0(S,\pi_*\bar{M}_{C^\vartheta}^{\gp})$ such that $L(\alpha)$ is $\vartheta$-stable.
    
\end{theo}

In other words, $\L(\alpha)$ defines a section $S_\L^{\vartheta}\to  \mathcal P_{\pi}^\vartheta$ lifting the section $S\to\LogPic^0_{C/S}$ defined by $\L$.

\begin{expl}
    It is proved in \cite[Section~4]{holmes2025logDR} that for $S=\bar{\M}_{g,n}$ and $\L=\O(\sum_{i=1}^n a_i p_i):=\mathcal O(\bfw)$ of total degree $0,$ the cones of the subdivision $(\bar{\M}_{g,n}^\vartheta)_{\mathcal O(\bfw)}$ are indexed by the following discrete data:
\begin{itemize}
    \item a quasi-stable curve $\Gamma^{\vartheta}_{\bfw}\to\Gamma;$
    \item a $\vartheta$-stable divisor $D_{\bfw}$ on $\Gamma^{\vartheta}_{\bfw}$ (this is the tropical  divisor $\underline{\deg}(\mathcal O(\bfw)(\alpha^{\vartheta}))$ associated to the $\vartheta$-stable twists);
    \item a piece-wise linear function $\alpha^{\vartheta}$ (vanishing on the vertex supporting the first marking) realizing the equivalence between $D_{\bfw}$ and $\underline{\deg}(\mathcal O(\bfw)).$
\end{itemize}
\end{expl}

\subsection{Moduli of log roots of a line bundle}

 Let $S$ an algebraic stack with logarithmic structure, $C\to S$ be a log smooth curve, and let $L$ a line bundle $C$ of degree $0$ along the fibers. Given $\delta$ we 
 denote by $ S( L^{1/\delta})$ the moduli space of log-$\delta$-roots of $ L$. This is constructed in \cite{holmes2023root} as the fiber product in the category of stacks over log schemes:
\bcd
 S( L^{1/\delta}) \ar[r]\ar[d]  &\LogPic_{C/S} \ar[d,"r\cdot"] \\
S  \ar[r," L"] & \LogPic_{C/S},
\ecd
where $\delta\cdot$ denotes the multiplication map. 

\medskip

 The stack $S( L^{1/\delta})$ comes equipped with a universal curve $C\times_S S( L^{1/\delta})$ and universal log-line bundle $\mathcal T$ such that
\[\mathcal T^{\otimes \delta}\equiv  L \in \LogPic_{C/S}( S( L^{1/\delta}).\]
 
\begin{prop}\cite[Corollary~3.12, Section~4.1.3]{holmes2023root}
   The moduli space  $S( L^{1/\delta})$  is a proper, finite, flat scheme over $S$ of degree $\delta^{2g};$ if $\delta$ is invertible on $S$ then is log \'etale (but not \'etale in general) over $S.$ 
\end{prop}
Over the locus where $C\to S$ is smooth, the map $S( L^{1/\delta})\to S $ is \'etale and for $L=\mathcal O_C$  $S(\mathcal O_C^{1/\delta})$ is a group scheme over $S.$
By considering the base change along a root stack $\widetilde{S}\to S$ which we now explain, we can regain these properties. 
\subsubsection{Root stack construction.}
 In the notation of \cite{chiodo2008stable}, $D_{i,(A|B)}$ the boundary divisor in $\mathfrak{M}_{g,n}$ of curves with one separating node where:  the genus splits as $i, g-i;$ the subsets $A, B \subseteq \left\{1,\dots,n\right\}$ prescribe the distribution of the markings on the two components and let $D_{\mathrm{irr}}$ the boundary divisor of curves with a non separating node. Consider the root-stack (see for example \cite[Section~2]{chiodo2008stable}:
\[\widetilde{\mathfrak{M}}_{g,n}=\mathfrak{M}_{g,n}\left[\sum_{i,A,B} D_{i,(A|B)}/\delta+D_{\mathrm{irr}}\slash \delta \right].\]
The latter  admits a modular interpretation: it is the moduli stack $\mathfrak{M}_{g,n}(\vec{\delta})$ of \emph{twisted} pre-stable curves in the sense of \cite{chiodo2008stable} with stabilizer of order $\delta$ at the nodes \cite[Theorem~4.5]{chiodo2008stable}.

\begin{rem}\label{rem:rootstackmonoids}
The root construction can  also be described in terms of
 logarithmic structures, as done in \cite[Section~2]{holmes2023root}.  At the level of monoids,  the upshot of the root  construction is the following: given closed point $s$, the monoid $\overline{M}_{\widetilde{\mathfrak{M}}_{g,n},s}$ is the saturation of the image of $\overline{M}_{\mathfrak{M}_{g,n},s}$ inside $(\mathbb Z^{|E(\Gamma_s)|}\oplus \mathbb Z^{|E(\Gamma_s)|})/(\ell_e-\delta\ell'_{e})$, i.e.
 $\overline{M}_{\widetilde{\mathfrak{M}}_{g,n},s}=\frac{1}{\delta} \overline{M}_{\mathfrak{M}_{g,n},s}.$

\medskip

    Holmes and Orecchia \cite{holmes2023root} consider a considerably more  efficient root stack which only extracts roots for those boundary components of $\mathfrak{M}_{g,n}$ corresponding to non separating nodes, namely for the components $\mathfrak{M}_{\Gamma}$ for which $b_1(\Gamma)\neq 0;$ even for those the roots are not extracted separately for each non separating edge. 
\end{rem}

There are three universal curves on $\widetilde{\mathfrak{M}}_{g,n}:$ 
\begin{itemize}
    \item the universal twisted curve $\mathfrak{C}^{1/\delta,\mathrm{tw}}$ with stacky nodes \cite{chiodo2008stable}; 
   \item  the corresponding coarse curve $\mathfrak{C}$, which is simply the pull-back from $\mathfrak{M}_{g,n}$ and it has singular total space due to  the base change; 
   \item the resolution $\widetilde{\mathfrak{C}}$ of the latter, whose fibers are obtained inserting a chain of $\delta-1$ unstable $\PP^1$ at each node.
\end{itemize}

A log smooth $C/S$ determines a log  morphism 
$S\to \mathfrak M_{g,n}$ where the latter is endowed with the natural logarithmic structure defined by the boundary divisor.
Then we define 
\[\widetilde{S}:=S\times_{\mathfrak M_{g,n}}\widetilde{\mathfrak M}_{g,n},\]
where the fiber product is in the category of  fs logarithmic algebraic stacks. 

By pull-back, on  $\widetilde{S}$ we have a twisted curve $C^{1/\delta},$ its coarse moduli space $C=C\times_S\widetilde{S}$ and the resolution of the latter, which we denote by $\widetilde{C}.$

If $\Gamma$ and $\widetilde{\Gamma}$ denote the tropicalization of $C_s,$ $\widetilde{C}_s$ respectively on some geometric point $s\in\widetilde{S},$ then $\widetilde{\Gamma}$ is obtained from $\Gamma$ subdividing each edge $e\in E(\Gamma)$ is $\delta$ edges $e_\delta\in E(\widetilde{\Gamma})$ of length $\ell(e_\delta)=\frac{\ell(e)}{\delta}.$ 
We have the following:
\begin{theo}\cite[Theorem~4.3]{holmes2023root},\cite[Section~3.1]{chiodo2008stable}\label{thm:logrotsafterbasechange}
\begin{itemize}
    \item There is a strict logarithmic morphism $\widetilde{S}(L^{1/\delta}):=S(L^{1/\delta})\times_S\widetilde{S}\to \widetilde{S}$, where the fiber product is in the category of fs logarithmic schemes.
    \item The underlying stack of $\widetilde{S}(\mathcal O^{1/\delta})$ is a finite flat group scheme over $\widetilde{S},$ \'etale if $\delta$ is invertible on $S$.
    \item The underlying stack of $\widetilde{S}( L^{1/\delta})$ is a finite \'etale torsor under $\widetilde{S}(\mathcal O^{1/r})$ over $\widetilde{S},$ for $\delta$ is invertible on $S$.
\end{itemize}
\end{theo}
Chiodo \cite{chiodo2008stable} interprets $\widetilde{S}( L^{1/\delta})$ as moduli of roots of a line bundle on the \emph{twisted curve} $C^{1/\delta}.$ Since we will not need this point of view, we will refrain from explaining it.

Furthermore, the following is true: on the universal curve $\widetilde{C}$ over $\widetilde{S}( L^{1/\delta})$ 
 the universal log root $\mathcal T$ can be represented by an honest line bundle, which we still denote by $\mathcal T\in \Pic_{\widetilde{C}/\widetilde{S}( L^{1/\delta})}$. We refer the reader to \cite[Lemma~2.2.5]{chiodo2008towards} or \cite[Section~2.2]{chiodo2024double}. The line bundle $\mathcal T$ is only a $\delta$-root  in the  logarithmic sense, namely we only know that fiberwise on $\widetilde{C}\to\widetilde{S}( L^{1/\delta})$

 \begin{equation}\label{eq:isolinebundles}
    \mathcal T^{\otimes \delta}=L_{\widetilde{C}}(\alpha),
\end{equation}

where $\alpha$ is some strict piecewise linear function on the universal tropical curve $\widetilde{\Gamma}.$ 

In  \cite[Section~3.6.3]{BCK}, following  \cite[Section~4,5]{abreu2023moduli} we explicitly constructed a tropical divisor $\mathrm{T}$ on the universal tropical curve $\widetilde{\Gamma}\to \widetilde{S}( L^{1/\delta})$ and a sPL function realizing the isomorphism above. Furthermore we can choose $\alpha$ to take value $0$ on each vertex 
     $v\in \sfV(\Gamma)\subseteq \sfV(\widetilde{\Gamma}).$
Then $\mathrm{T}=\underline{\deg}\mathcal T$.

\subsubsection{Stratification of moduli of roots}
\subsubsection{}
Let us first consider the case \medskip $S=\mathfrak{M}_{g,n}.$

The (stacky) cones of the tropicalization of $\Sigma(\widetilde{\mathfrak{M}}_{g,n})$  are the same as those of  $\Sigma(\mathfrak{M}_{g,n})$, but the lattices giving the integral structure have been modified. They are still indexed by dual graphs $\Gamma\in\GGG_{g,n+m}$, but to emphasize the change of lattice, we denote them by $\widetilde{\sigma}_{\Gamma}$. The stratum $S_{\Gamma}$ indexed by $\Gamma$ has closure $\overline{S}_{\Gamma}$ with normalization 
\[\widetilde{\mathfrak{M}}_{\Gamma}=\prod_v\widetilde{\mathfrak{M}}_{g(v),h(v)}\slash \Aut(\Gamma)\]
where $\widetilde{\mathfrak{M}}_{g(v),h(v)}$ is the algebraic stack of $\delta$-twisted stable curves with $\delta$-orbifold structure also along the
$h(v)$  marking corresponding to half edges in $\Gamma$ (i.e. nodes to be).

This relation between the  cone stacks $\Sigma(\mathfrak{M}_{g,n})$  and $\Sigma(\widetilde{\mathfrak M}_{g,n})$ implies  that the piece-wise polynomial with $\mathbb Q$ coefficients (as defined and  studied in \cite{molcho2023hodge,holmes2025logDR}) on these  two  cone  stacks coincide.

\subsubsection{}
 By \cite[Lemma~4.2]{holmes2023root}, the morphism  $\widetilde{\mathfrak{M}}_{g,n}(\frac{1}{\delta})\to  \widetilde{\mathfrak{M}}_{g,n}$ is strict. Therefore, a first possibility is to consider the log stratification pulled back from  $\Sigma(\widetilde{\mathfrak M}_{g,n})$. Instead, we refine further by keeping track of the multidegree of the universal root representative $\T$ constructed in \cite[Section~3.3.5]{BCK} and recalled above.

Let $\widetilde{\mathfrak{M}}_\Gamma(\frac{1}{\delta})$ the
 the stratum lying over $\widetilde{\mathfrak{M}}_\Gamma$:
\[\widetilde{\mathfrak{M}}_\Gamma(\frac{1}{\delta})=\widetilde{\mathfrak{M}}_{\Gamma}\times_{\widetilde{\mathfrak{M}}} \widetilde{\mathfrak{M}}(\frac{1}{\delta})\to\widetilde{\mathfrak{M}}_{\Gamma}.\]
Over such stratum we can consider the multidegree $\mathrm{T}=\underline{\operatorname{deg}}(\mathcal T)$. The multidegree $\mathrm{T}$ is a $\delta$-torsion tropical divisor on the tropical curve $\widetilde{\mathtt{C}}_{\widetilde{\sigma}_{\Gamma}}$ (supported at the vertices).

\medskip

Since the degree of line bundles is locally constant on the families of curves $\widetilde{\mathfrak{C}}_v$, distinct divisors $\mathrm{T}$ correspond to different open and closed component of $\widetilde{\mathfrak{M}}_\Gamma(\frac{1}{\delta})$. We thus refine the stratification, indexing new strata $\widetilde{\mathfrak{M}}_{\Gamma,\mathrm{T}}(\frac{1}{\delta})$ by the choice of a graph $\Gamma$ together with this $\delta$-torsion tropical divisor $\mathrm{T}.$
\medskip

This suggests that we can define a finer cone stack $\Sigma(\frac{1}{\delta})$ endowed with a scrict  morphism from  $\mathfrak{t}\colon \widetilde{\mathfrak{M}}(\frac{1}{\delta})\to \Sigma(\frac{1}{\delta})$ detecting the finer stratification discussed above. The cones of such cone stack are indexed by the following discrete data:

\begin{defi}
\label{defi-delta-torsion-graph}
($\delta$-torsion graph)
\begin{itemize}
    \item A $\delta$-torsion graph $(\Gamma,\mathtt{T})$ is the data of a graph $\Gamma\in\GGG_{g,n+m}$ together with the choice of a linear equivalence class of $\delta$-torsion divisor 
    \item A representative of $(\Gamma,\mathtt{T})$ is the choice of a divisor $\mathrm{T}$ in the class $\mathtt{T}$ supported at the vertices of the subdivided graph $\widetilde{\Gamma}$. Since it is always possible to make such a choice compatibly, we do not distinguish between the representative $\mathrm{T}$ and its class  $\mathtt{T}.$ 
    \item The set of $\delta$-torsion graphs is denoted by $\GGG_{g,n+m}(\frac{1}{\delta})$.
\end{itemize}
\end{defi}

The set of $\delta$-torsion graphs $\GGG_{g,n+m}(\frac{1}{\delta})$ is also provided with the edge contraction operation. 
Under edge contraction the multidegree supported at two vertices which  get identified is simply given by  the  sum.

For each  $(\Gamma,\mathtt{T})$ we  have a  cone $\widetilde{\sigma}_{\Gamma,\mathtt{T}}\cong \mathbb R^{E(\Gamma)}_{\geq 0}$ and $\frac{1}{\delta}$ refined integral structure.

Taking the colimits on $\GGG_{g,n+m}(\frac{1}{\delta})$ with face maps induced by edge contractions and self-maps determined  by automorphisms, we obtain 
a cone stack $\Sigma(\frac{1}{\delta})$. This should be thought as a \emph{tropical moduli space} of $\delta$-roots of the trivial line bundle.



The (normalization of the closure of) logarithmic strata of
$\widetilde{\mathfrak{M}}(\frac{1}{\delta})$ with respect to the strict map $\tfk\colon \widetilde{\mathfrak{M}}(\frac{1}{\delta})\to \Sigma(\frac{1}{\delta}) $ are precisely the $\widetilde{\mathfrak{M}}_{\Gamma,\mathtt{T}}(\frac{1}{\delta})$. Unlike the case of curves or orbifold curves, the strata do not split as a product of moduli over vertices, because we cannot recover a torsion line bundle from its pull-back to the normalization. However, the forgetful map $\nu_{\Gamma,\mathtt{T}}\colon\widetilde{\mathfrak{M}}_{\Gamma,\mathtt{T}}(\frac{1}{\delta})\to\widetilde{\mathfrak{M}}_\Gamma$ is \'etale, finite  \cite[Theorem~4.3]{holmes2023root} of degree $\delta^{2g-b_1(\Gamma)}$. 

There are exactly $\delta^{b_1(\Gamma)}$ linear equivalence classes of $\delta$-torsion tropical divisors, so that summing over  the choices of $\mathtt{T}$, the total degree is $\delta^{2g}$ as expected.

\begin{rem}
    Furthermore, following \cite[Section~6]{holmes2025logDR}, $\tfk$ determines a pull-back map  from the ring  $\mathrm{sPP}(\Sigma(\frac{1}{\delta}))$ of strict piece-wise polynomials to the Chow cohomology of $\widetilde{\mathfrak{M}}(\frac{1}{\delta}).$ 
\end{rem}

\subsubsection{}

More generally, let $C\to S$ be a log smooth curve over a log scheme, $S\to\Sigma$ a strict morphism and $S_{\sigma}\to S
$ a log stratum. We denote by $\Gamma_{\sigma}\to S_{\sigma}$ the tropical curve over this stratum. We remark that the dual graph is constant along the locally closed strata. 

The space of log roots $S( L^{1/\delta})$ has a refined logarithmic stratification with strata $S( L^{1/\delta})_{\sigma,\mathtt{T}}$ where $\mathrm{T}$ is one of the 
$\delta^{b_1(\Gamma_{\sigma})}$ tropical roots of $\underline{\deg}(L)$ constructed as sketched before. This can be thought as being induced by a strict map to the cone stack $\Sigma(\frac{1}{\delta})$ obtained as colimit of $\sigma_{\mathtt{T}}$ along face maps induced by edge contractions.

The morphism of Artin stacks $\Sigma(\frac{1}{\delta})\to\Sigma$ is a non separated \'etale cover.

\begin{expl}
If we take $S=(\bar{\M}^{\vartheta})_{\mathcal O(\bfw)}$ and $L$ the $\vartheta$-stable twist of $\mathcal O(\bfw)$ on the quasi-stable curve $C^\vartheta$ then the discrete data indexing the logarithmic strata of $S( L^{1/\delta})$ is the following:
\begin{itemize}
    \item a quasi-stable graph $\Gamma^{\vartheta}_{\bfw}\to\Gamma;$
    \item a $\vartheta$-stable divisor $D_{\bfw}$;
    \item a piece-wise linear function $\alpha^{\vartheta}$ realizing the equivalence between $D_{\bfw}$ and $\underline{\deg}(\mathcal O(\bfw)).$
    \item a divisor $\mathrm{T}$ supported on $\widetilde{\Gamma}^{\vartheta}_{\bfw}$ obtained sub-dividing each edge $\delta$ times such that $\delta\mathrm{T}$ is linearly equivalent to $D_{\bfw}$.
\end{itemize}
\end{expl}

\subsection{Virtual stratification of moduli of maps}
\label{sec-stratification-moduli-of-roots}

Let $\M(X)$ be the moduli space of stable maps for some fixed discrete data $g,n,\beta.$ Then $\M(X)$ is equipped with a strict morphism to the moduli space $\mathfrak{M}=\mathfrak{M}_{g,n}$ of pre-stable curves as well as a stabilization morphism  to $\bar{\M}_{g,n}$ as soon as $2g-2+n>0,$ which is always the case for us. The latter is not strict, but is combinatorially flat.

\medskip

Given: a subdivision $\mathfrak{M}^\diamond\to\mathfrak{M};$ a more general logarithmic modification $\widetilde{\mathfrak{M}}\to \mathfrak{M},$ for example induced by a root construction; the space of log roots of the trivial bundle $\mathfrak{M}(\frac{1}{\delta}),$ we can consider the pull-back of $\M(X)\to\mathfrak{M},$ which we accordingly denote by $\M(X)^\diamond,$ $\widetilde{\M}(X),$ $\widetilde{\M}(X)(\frac{1}{\delta}).$

The usual virtual pull-back techniques \cite{manolache2012virtual} endow any such space  with a virtual class. Furthermore, by Costello's virtual push-forward formula, \cite{costello2006higher,manolache2012virtual}, the push-forward of such class to $\M(X)$ is a multiple of $\vir{\M(X)}$ determined by the degree of the underlying map of pre-stable curves (this being either $1$ or $\delta^{2g}$ in the cases we are interested in).

\medskip

We continue this section discussing the virtual stratification of $\M(X)$ (or its variant) induced by the logarithmic stratification of $\mathfrak{M}$ (or its variants). 

\subsubsection{The space of maps $\M(X)$}

In the case of $\M(X)$, we have already discussed the virtual stratification after introducing $X$-stable graphs (Definition~\ref{defi-X-valued-graph}). We saw that
the virtual class of the stratum corresponding to $\Gamma$ is
defined by Gysin pull-back along the cartesian diagram

\begin{equation*}
\begin{tikzcd}
\M_{\Gamma}(X)\ar[r]\ar[d] &\M(X)\ar[d]\\
\mathfrak{ M}_{\Gamma} \ar[r, "j"] & \mathfrak{M}.
\end{tikzcd} 
\end{equation*}
Furthermore, it admits the  following splitting by vertex class:
\begin{equation}\label{eq-splitting}
    \vir{\M_\Gamma(X,\beta)}=\sum_{(\beta_v)}\Delta^!\left(\prod\vir{\M_v(X,\beta_v)}\right)=\vir{\M(X,\beta)}\cap p^*\mathrm{PD}([j_*\mathfrak{M}_{\Gamma}]),
\end{equation}

where $\Delta$ is the diagonal map $X^{\sfE(\Gamma)}\to X^{H\sf(\Gamma)}$, also admitting a Gysin pull-back.
The fact that these two classes coincide is one of the foundational aspects of GW-theory \cite{behrend1997gromov}.

\medskip

Furthermore, the class $\mathrm{PD}([j_*\mathfrak{M}_{\Gamma}])$ can be interpreted as a piecewise polynomial function through the morphism $\mathrm{sPP}(\Sigma(\mathfrak{M}))\to\mathrm{CH}^*(\M(X))$. In \cite[Section~6]{holmes2025logDR} the  piece-wise  polynomial corresponding to the boundary strata classes are explicitly computed.

\subsubsection{The space of maps $\M_K(X)$}
We saw that the strata 

\[\M_{K,\Gamma}(X)=\M_K(X)\times_{\mathfrak{M}}\mathfrak{ M}_{\Gamma},\]

are further disconnected by \emph{twisted diagonals}, i.e. denoting by

$$\M_{\Gamma,\widetilde{K},\varphi}(X,(\beta_v))=\bigsqcup_{(K_v):\bigcap K_v=\widetilde{K}} \M_{\Gamma,(K_v),\varphi}(X,(\beta_v))$$
where $\Gamma,\widetilde{K}=\cap K_v,\varphi$ is a monodromy graph with core $K,$
we have \[\vir{\M_{\Gamma,(K_v),\varphi}(X,(\beta_v))}=\Delta_{\widetilde{K},\varphi}^!(\prod\vir{\M_{v,K_v}(X)}),\]

and 
\[\vir{\M_{K,\Gamma}(X)}=\sum_{(\beta_v)}\sum_{(K_v):\bigcap K_v =\widetilde{K}}\sum_{\varphi}\Delta_{\widetilde{K},\varphi}^!(\prod\vir{\M_{v,K_v}(X,\beta_v)}).\]

\subsubsection{The space of maps $\M_K^{\diamond}(X)$}
Let $\M(X)^\diamond\to\M(X)$ be the logarithmic modification obtained pulling back $\mathfrak{M}^\diamond\to\mathfrak{M}$. First of all, since $\M(X)$ decomposes in open and closed components $\M_K(X),$ we simply get open and closed components 
$\M_K^{\diamond}(X)$ by pull-back, which again by Costello virtual push-forward formula \cite{costello2006higher} are virtually birational to $\M_K(X).$

\medskip 

The logarithmic stratification of $\mathfrak{M}^\diamond$ induces a virtual stratification on $\M_K^{\diamond}(X).$ Recall that strata of  $\mathfrak{M}^\diamond$ are indexed by cones $\sigma$ in the subdivision $\Sigma^\diamond.$

\medskip

    For geometrically meaningful subdivisions, such as those induced by a choice of stability condition $\vartheta$ and a $\vartheta$-stable twist of a given line bundle, one can explicitly name the discrete data parametrizing the cones of the subdivision. 
    
    For notational purposes, let's suppose to be in such a situation and denote by $\Gamma^\diamond$ such discrete data, indexing the cones $\sigma_{\Gamma^\diamond}$ of the subdivision $\Sigma^\diamond$. We denote the associated (cosed) stratum of $\mathfrak{M}^\diamond$ by $\mathfrak{M}^\diamond_{\Gamma^\diamond}$. We denote by $\sigma_{\Gamma}$ the smaller cone in $\Sigma(\mathfrak{M})$ containing the image of $\sigma_{\Gamma}$. 
    
    Then we have an associated (dominant, but possibly with positive dimensional fibers) morphism of strata $\mathfrak{M}^\diamond_{\Gamma^\diamond}\to \mathfrak{M}_{\Gamma} .$

    \medskip

The strata $\M_{\Gamma^\diamond,K}^{\diamond}(X)$ can be presented as fiber product in two different ways:

\begin{equation*}
\begin{tikzcd}
\M_{\Gamma^\diamond,K}^{\diamond}(X)\ar[r]\ar[d] &\M_K^\diamond(X)\ar[d]\\
\mathfrak{M}^\diamond_{\Gamma^\diamond} \ar[r] & \mathfrak{M}^\diamond
\end{tikzcd} ,\quad \text{ or }
\begin{tikzcd}
\M_{\Gamma^\diamond,K}^{\diamond}(X)\ar[r]\ar[d] &\M_{\Gamma,K}(X)\ar[d]\\
\mathfrak{M}^\diamond_{\Gamma^\diamond}\ar[r, "p_{\Gamma^\diamond}"] & \mathfrak{M}_{\Gamma}
\end{tikzcd}.
\end{equation*}

In both cases the bottom arrow are l.c.i morphisms are thus endowed with a Gysin pull-back. Moreover, functoriality of the latter and of the virtual pull-back ensure that the two classes coincide, i.e.

\[\vir{\M_{\Gamma^\diamond,K}^{\diamond}(X)}= p_{\Gamma^\diamond}^!(\vir{\M_{\Gamma,K}(X)})= \sum_{(\beta_v)}\sum_{(K_v):\bigcap K_v =\widetilde{K}}\sum_{\varphi} p_{\Gamma^\diamond}^!\vir{\M_{\Gamma,(K_v),\varphi}(X,(\beta_v))}\]

\subsubsection{The space of maps $\widetilde{\M}(X) (\frac{1}{\delta})$}

Passing to the stack parametrizing $\delta$-roots of the trivial line bundle in our two cases of interest recalled above we get that:
\begin{itemize}
    \item The strata of $\M(X)(\frac{1}{\delta})$ are indexed by stable $X$-graph together with a linear equivalence class of $\delta$-torsion divisor $\mathtt{T}$.  The usual colimit construction allows  us to define a cone stack $\Sigma(X)(\frac{1}{\delta})$, target  of the obvious tropicalization map $\tfk.$
    \item The strata of $\M_K(X)(\frac{1}{\delta})$ are indexed by the  monodromy graphs together with a linear equivalence class of $\delta$-torsion divisor $\mathtt{T}$ (in addition of the class decoration, the kernel $\widetilde{K}$ and global monodromy $\varphi$). 
    The associated cone stack is denoted  by $\Sigma_K(X)(\frac{1}{\delta})$.
\end{itemize}

For the stack of roots, Equation (\ref{eq-splitting}) becomes
\begin{equation}\label{eq:virtual compatibility}
     \vir{\widetilde{\M}_{\Gamma,\mathtt{T}}(X,\beta)} = \sum_{(\beta_v)} \vir{\widetilde{\M}_{\Gamma,\mathtt{T}}(X,(\beta_v))}=\vir{\widetilde{\M}(X)(\frac{1}{\delta})}\cap \mathrm{ft}^*\mathrm{PD}(j_*[\widetilde{\mathfrak{M}}_{\Gamma,\mathtt{T}}(\frac{1}{\delta})]),
\end{equation}
and the last class can still be interpreted as a piecewise polynomial function via the map to $\widetilde{\M}(X)(\frac{1}{\delta})\to \widetilde{\mathfrak{M}}(\frac{1}{\delta})\rightarrow\Sigma(\frac{1}{\delta})$.


Moreover, denoting by  $r\colon \widetilde{\M}_{\Gamma,\mathtt{T}}(X,(\beta_v))(\frac{1}{\delta})\to \widetilde{\M}_{\Gamma}(X,(\beta_v))$  simply defined by forgetting the torsion line-bundle, we get the following splitting
\[r_*\vir{\widetilde{\M}_{\Gamma,\mathtt{T}}(X,(\beta_v))}= \delta^{2g-b_1(\Gamma)} \vir{\widetilde{\M}_{\Gamma}(X,(\beta_v))}\]

\subsubsection{The space of maps $\M^\diamond_{K,\Gamma^\diamond}(X)(\frac{1}{\delta})$}

Proceeding as above, we can also describe the virtual classes of the
the boundary strata of $\M^\diamond_{K,\Gamma^\diamond}(X)(\frac{1}{\delta}),$ obtained looking at the logarithmic virtual strata of the pull-back of $\M_K(X)$ along 
$\mathfrak{M}^\diamond(\frac{1}{\delta})\to\mathfrak{M}^\diamond\to\mathfrak{M}.$ We stop writing $\widetilde{\mathfrak M}$ to denote the root stack along the boundary, as this is a particular case of logarithmic stratification.

A logarithmic stratum of $\mathfrak{M}^\diamond(\frac{1}{\delta})$, corresponding to a  cone of the tropicalization, is indexed by a pair $(\Gamma^\diamond,\mathtt{T})$ where $\Gamma^\diamond$ identifies a strata in $\mathfrak{M}^\diamond$ and $\mathtt{T}$ represents one of the $\delta^{b_1(\Gamma)}$ tropical roots of the trivial bundle.

Let $\mathfrak{M}_{\Gamma,\mathtt{T}}(\frac{1}{\delta})$ the strata over which $\mathfrak{M}_{\Gamma^\diamond,\mathtt{T}}^\diamond(\frac{1}{\delta})$ surjects, and let $p_{\Gamma^\diamond,\mathtt{T}}$ denotes the projection. Then:

\[\vir{\M^\diamond_{K,\Gamma^\diamond,\mathtt{T}}(X)(\frac{1}{\delta})}=p_{\Gamma^\diamond,\mathtt{T}}^!\vir{\M_{K,\Gamma,\mathtt{T}}(X)(\frac{1}{\delta})}\]

\medskip

We conclude this section introducing the correlator map for torsion line bundles, which gives us a way to further decompose $\M_K(X)(\frac{1}{\delta})$ (and its variant obtained via logarithmic modifications) as well as its strata in open and closed substack indexed by correlators.

\subsection{Correlators for torsion log-line bundles}
\label{sec-correlator-torsion line bundle}
In \cite{blomme2024correlated} we introduced a \emph{correlating map}  \[\M(Y|D^\pm)\xrightarrow{\kappa^\delta} \Alb(X).\] Identifying $\Pic^0(C)$ with the Albanese variety of $C$ (when this is smooth) one sees that
\[\kappa^\delta(f)=f_*\O(\sum\frac{a_i}{\delta}p_i),\]
where
$f_*\colon\Pic^0(C)\to\Alb(X)$ denotes the homomorphism dual, in the category of abelian varieties, to the pull-back.

In \cite{blomme2024correlated}, to extend the correlating map $\kappa^\delta$ to nodal curves, we used the evaluation map at the marked points as a way to go around the fact that the push-forward map $f_*$ and the meaning of $\O(\sum\frac{a_i}{\delta}p_i)$ was not quite clear for nodal curves. The present paper fixes that issue by working with spin stable maps and logarithmic Picard group.

We now define a \emph{seemingly different} \emph{correlating map} on the moduli space of spin stable map with value in $\Alb(X).$
 Proving that these two definition of correlator agree upon restriction to log moduli of (rubber or not) maps to $Y=\mathbb P_X(\mathcal O\oplus\L)|D)$ is the content of Section \ref{sec-link-correlators-line-bundle-stable-maps}.

\begin{defi}

    Given a point  in $\M(X)(\frac{1}{\delta})$, i.e. a  stable map $f:C\to X$ together with $\T$ a torsion log-line bundle on $C$, we  define the  \textit{correlator} $\theta_{f,\T}$  to be the map  $\Pic^0(X)[\delta]\to\mu_{\delta}$ induced  the log-Weil-pairing. Its value on  $\L\in\Pic^0(X)[\delta]$, is given by
    $$\theta_{f,\T}(\L)=W_X(f_*\T,\L)=W_C(\T,f^*\L).$$
    
    As the pairing is non-degenerate, the data is equivalent to the data of $f_*(\mathcal T)\in\Alb(X)[\delta]$, where  $f_*$  is the homorphism induced by the universal property of $\LogPic^0(C)$  (see Section~\ref{sec:logAJ}). We thus abuse notation and also denote $f_*\T$ with $\theta_{f,\T}.$

    The definition lifts in the obvious way to any logarithmic modification (induced by subdivision and or root construction on the underlying stack of pre-stable curves) $\M^\diamond(X)(\frac{1}{\delta}).$
\end{defi}

\begin{rem}
Recall that moduli of maps/spin maps are always considered as algebraic stack with log structure, the latter being pulled-back by the Artin  stack of pre-stable curves or of $\delta$-log roots of the trivial bundle, or any logarithmic modification of these.
\end{rem}

In order to simplify notation, we drop from the notation the superscript $\diamond$ indicating that a log modification happened. But everything we say below holds on any logarithmic model of 
$\M(\frac{1}{\delta}).$

\subsubsection{Properties of the correlator map}
The correlator function is of course locally constant. 
It follows that the moduli space of spin stable maps is a disjoint union of open and closed sub-stacks:
\[\M(X)(\frac{1}{\delta})=\bigsqcup_{\theta\in\Alb(X)[\delta]}\M^\theta(X)(\frac{1}{\delta}).\]

  We are interested in understanding the relation between the decomposition of $\M(X)(\frac{1}{\delta})$ in  $\M^\theta(X)(\frac{1}{\delta})$ and the decomposition of $\M(X)$ in  $\M_K(X)$. In fact, $K$ imposes restrictions on the values $\theta\in\Alb(X)[\delta]$ that can be achieved, since the pairing with line bundles restricting trivially has to be $0$. This is actually is the main reason for having introduced the components   $\M_K(X)$ in \ref{sec-refined-stratification-via-coverings}. 

  \medskip
   
   We denote by $\M_K^\theta(X)(\frac{1}{\delta})$ the restriction of $\M^\theta(X)(\frac{1}{\delta})$ to  $\M_K(X)(\frac{1}{\delta})$. Recall that for the $\delta$-torsion log-line bundles on a log-curve $C$, we have two short exact sequences, fitting as the line and second column of the following diagram:
   \bcd
\widetilde{K} \ar[rd,"f^*"]\ar[d] &  & & & \\
  \Pic^0(E) \ar[rd,"f^*"] & \mathrm{H}^1(\Gamma,\CC^*) \ar[d,hook] & & & \\
  0 \ar[r] & \Pic^{[0]}(C) \ar[r]\ar[d,two heads] & \LogPic^0(C) \ar[r,"\rm{trop}"]\ar[d,"f_*"] & \operatorname{TroPic}^0(C) \ar[d] \ar[r] & 0, \\
  & \bigoplus\Pic^0(C_v) & K^\perp\subset\Alb(E) \ar[r,"q"] & \Alb(E)/\widetilde{K}^\perp & \\
\ecd
where we deleted the $[\delta]$ to easen the notation. We further added the pull-back map $f^*$. Since elements of $\widetilde{K}$ restrict trivially componentwise, they factor through $\mathrm{H}^1(\Gamma,\CC^*)$. Completing the bottom right of the diagram is the content of the next lemma.

\begin{lem}
    Let $f:C\to X$ be a point of  $\M_K(X)$ for $K\subset\Pic^0(X)[\delta]$. We consider the map $f_*:\T\in\LogPic^0(C)[\delta]\mapsto\theta_{f,\T}\in\Alb(X)$. Then, we have the following:
    \begin{enumerate}
        \item We have $\operatorname{Im}f_*=K^\perp\subset\Alb(X)[\delta]$.
        \item We have $\rmH^1(\Gamma,\mu_\delta)\subset\ker f_*=(\mathrm{Im}f^*)^\perp$, where the orthogonal is for the log-Weil pairing.
        \item We have $f_*(\Pic^{[0]}(C)[\delta])=\widetilde{K}^\perp$. In particular, we get a map
        $$\operatorname{TroPic}^0(C)[\delta]\to\Alb(E)[\delta]/\widetilde{K}^\perp\simeq\operatorname{Hom}(\widetilde{K},\mu_\delta).$$
        \item The previous map coincides with the one induced by the monodromy $\varphi\colon\widetilde{K}\to\mathrm{H}^1(\Gamma,\mu_\delta).$
    \end{enumerate}
\end{lem}

\begin{proof}
    \begin{enumerate}
        \item Notice that if $f:C\to X$ belongs to $\M_K(X)$, $\theta_{f,\T}(\L)=0$ for any $\L\in K$, since $\L$ pulls back to $\O$ on $C$ by definition. In that case, $\theta_{f,\T}\in K^\perp$. Conversely, any element of $K^\perp$ is a character on $\Pic^0(X)[\delta]/K$, and thus on $\mathrm{Im}f^*\subseteq \Pic^{[0]}(C)$. We extend it to $\LogPic^0(C)[\delta]$. Using the non-degeneracy of the Weil pairing, it may be represented by some $\T$.
        
        \item As $f^*$ and $f_*$ are adjoints of each other for the Weil pairings, for every $\L\in\Pic^0(X)[\delta]$, we have
    $$\theta_{f,\T}(\L)=W_X(f_*\T,\L)=W_C(\T,f^*\L),$$
    yielding that the kernel is actually $(\mathrm{Im}\ f^*)^\perp$.
    
        Concerning the inclusion, the pull-back of elements of $\Pic^0(X)[\delta]$ are honest line bundles on $C$. Hence, $\mathrm{Im}f^*=f^*(\Pic^0(X)[\delta])$ is a subgroup of $\Pic^{[0]}(C)[\delta]$. Therefore, the orthogonal of the image of $f^*$ contains $(\Pic^{[0]}(C)[\delta])^\perp=\rmH^1(\Gamma,\mu_\delta)$.
        
        \item We rather compute the orthogonal of $f_*(\Pic^{[0]}(C)[\delta])$. We have that $\U\in f_*(\Pic^{[0]}(C)[\delta])^\perp$ if and only if for any $\L\in\Pic^{[0]}(C)[\delta]$, one has
        $$W_E(\U,f_*\L)=W_C(f^*\U,\L)=0.$$
        Equivalently, this means that $f^*\U\in \Pic^{[0]}(C)[\delta]^\perp=\mathrm{H}^1(\Gamma,\mu_\delta)$. In other words, $\U$ restricts trivially to any component, which is the definition of $\widetilde{K}$. Therefore, $f_*(\Pic^{[0]}(C)[\delta])=\widetilde{K}^\perp$. To get the map, take the image by $f_*$ of any lift of $\mathtt{T}\in\operatorname{TroPic}^0(C)[\delta]$. The image is well-defined up to $f_*(\Pic^{[0]}(C)[\delta])=\widetilde{K}^\perp$, so that the image modulo the latter is well-defined.
        
        \item To see it coincides with the dual of $\varphi$, we just unravel the definition. Let $\T$ be a torsion log-line bundle with tropicalization $\mathtt{T}$. The correlator is the data of the homomorphism $W_C(\T,f^*(-)) $. Via adjunction, looking at the correlator modulo $\widetilde{K}^\perp$ is the same as looking at $W_C(\T,f^*(\cdot)) $ restricted to $\L\in \widetilde{K}$.

        Now, for  $\L\in \widetilde{K}$, $W(\T,f^*\L)=\widetilde{W}(\mathtt{T},f^*\L)$, where $\widetilde{W}$ is the induced pairing between $H^1(\Gamma,\ZZ_\delta)$ and $\TroPic(C)[\delta]$ (See Proposition \ref{prop:logweil}). This is precisely $\varphi(\L)(\mathtt{T})$, where $\mathtt{T}$ is interpreted as an element in $H_1(\Gamma,\ZZ_\delta)$.
    \end{enumerate}
\end{proof}

In particular, the monodromy $\varphi$ determines the possible correlators for the stratum corresponding to $\mathtt{T}$: it is precisely the $\widetilde{K}^\perp$-torsor $\varphi^\intercal(\mathtt{T})$, where $\varphi^\intercal\colon \rmH_1(\Gamma,\ZZ_\delta)\to\Alb(E)[\delta]/\widetilde{K}^\perp$.

\subsubsection{Link between correlator and stratification}

In Section \ref{sec-stratification-moduli-of-roots}, we described a stratification of $\M(X)(\frac{1}{\delta})$ and $\M_K(X)(\frac{1}{\delta})$ and so that the forgetful map forgetting the spin structure restricted to one of these strata has degree $\delta^{2g-b_1(\Gamma)}$.
The correlator function shows that the above moduli spaces are disjoint union of open and closed sub-stacks, and we have an interaction between $\theta$ and $\mathtt{T}$ in the following sense:
\begin{itemize}
    \item for given $\theta$,  whether $f_*\T$ can take value $\theta$ depends on tropical divisor $\mathtt{T}=\underline{\deg}(\T);$ 
    \item for $\T$ a  spin structure on a stratum with fixed tropical torsion $\mathtt{T}$ the value of the correlator $f_*\T=:\theta$ is not constant.
\end{itemize}
We explain the relation more explicitly, restricting to log-line bundles with correlator $0$.

 Recall that for the log-Weil pairing on $\LogPic(C)[\delta]$, $\rmH^1(\Gamma,\mu_\delta)$ and $\Pic^{[0]}(C)[\delta]$ are orthogonal to each other (see Proposition~\ref{prop:logweil}), inducing a non-degenerate pairing $\widetilde{W}$ between $ \rmH^1(\Gamma,\mu_\delta)$ and $\TroPic(C)[\delta]\simeq \rmH_1(\Gamma,\ZZ_\delta)$. We use this to describe the image of the stratum $\M_{\Gamma,\widetilde{K},\varphi}^0(X)(\frac{1}{\delta})$ under the tropicalization map.

\begin{prop}
\label{prop-link-trop-correlator}
    Fix a boundary strata of $\M_K(X)$ indexed by $(\Gamma,\widetilde{K},\varphi)$ with core $K$.\footnote{ The dual graph is $\Gamma$,  $K$ (resp. $\widetilde{K}$) is the subgroup of torsion line bundles which pull-back to the trivial  one (resp. componentwise trivial) and monodromy is given by $\varphi:\widetilde{K}/K\hookrightarrow \rmH^1(\Gamma,\mu_\delta)$.} In particular, we have $f^*(\widetilde{K})\subset \rmH^1(\Gamma,\mu_\delta)$. Then, the tropicalization of log-line bundles with $0$-correlator is the orthogonal of $f^*(\widetilde{K})$ for the pairing $\rmH^1(\Gamma,\mu_\delta)\otimes\TroPic(C)[\delta]\to\mu_\delta$:
    $$\left(\Trop(\ker f_*)\right)^{\perp_{\widetilde{W}}} = f^*(\widetilde{K}) = \mathrm{Im}f^*\cap \rmH^1(\Gamma,\mu_\delta)\subset \rmH^1(\Gamma,\mu_\delta).$$
\end{prop}

\begin{proof}
 Let $\mathtt{T}=\Trop(\T)$ where $\T$ has correlator $0$. In particular, it means that for any $\L\in \widetilde{K}$, we have
 $$0=\theta_{f,\T}(\L)=W_C(f^*\L,\T)=\widetilde{W}(f^*\L,\mathtt{T}),$$
 so that $\mathtt{T}$ has to lie in the orthogonal of $f^*(\widetilde{K})$ for $\widetilde{W}$. Conversely, let $\mathtt{T}\in f^*(\widetilde{K})^{\perp_{\widetilde{W}}}$ and $\T$ be any lift of $\mathtt{T}$. To prove that $\mathtt{T}$ is the image of a correlator $0$ log-line bundle, we look for $\U\in \ker\Trop=\Pic^{[0]}(C)$ such that $\T-\U$ has correlator $0$. This means that for any $\L\in\Pic^0(X)[\delta]$ we have
 $$W_C(f^*\L,\T)=W_C(f^*\L,\U).$$
 As $f^*\L$ and $\U$ both belong to $\Pic^{[0]}(C)[\delta]$, we can quotient by $\rmH^1(\Gamma,\mu_\delta)$ on the right hand side, which amounts to restrict to components: we look for $(\U_v)$ such that
 $$W_C(f^*\L,\T)=\sum_v W_{C_v}(f^*_v\L,\U_v).$$
 If such $(\U_v)$ exist, this also means, unsurprisingly, that the choice of $\U$ only matters modulo $\rmH^1(\Gamma,\mu_\delta)$, whose elements are always orthogonal to $\mathrm{Im}f^*\subset \Pic^{[0]}(C)[\delta]=(\rmH^1(\Gamma,\mu_\delta))^\perp$ anyway.
 The assumption ensures that when all $f^*_v\L$ are zero, which is equivalent to saying that $\L\in\widetilde{K}$, so is the left-hand side. Therefore, we may quotient by the kernel of $(f_v^*)_v$ and use the non-degeneracy of $\bigoplus W_{C_v}$ to conclude.
\end{proof}

Putting together the strata description and Proposition~\ref{prop-link-trop-correlator}, we obtain the  following: 

\begin{coro}\label{coro-tropicalization-corr0-strata}
Fix a graph $\Gamma$,  a subgroup $\widetilde{K}$ and a global monodromy $\varphi\colon \widetilde{K}\otimes \rmH_1(\Gamma,\ZZ_\delta)\to\mu_\delta.$ The stratum  $\M_{\Gamma,\widetilde{K},\varphi}^0(X)(\frac{1}{\delta})$ parametrizes correlator $0$ spin maps with the chosen discrete  data. Then, we have the following:
    \begin{enumerate}[label=(\roman*)]
        \item The stratum sits at the boundary of $\M_K(X)(\frac{1}{\delta})$ where $K$ is the left kernel of $\varphi$:
        $$K=\{\L \text{ s.t. }\forall\gamma,\ \varphi(\L,\gamma)=0\}\subset \widetilde{K}.$$
        \item The image of the tropicalization map is the right kernel of $\varphi$:
        $$\TTT=\{\gamma \text{ s.t. }\forall\L\in \widetilde{K},\ \varphi(\L,\gamma)=0\} \subset \rmH_1(\Gamma,\ZZ_\delta).$$
    \end{enumerate}
\end{coro}

\begin{proof}
    The first statement has already been proven. By definition, the kernel is the subgroup of line bundles that restrict trivially on components that also happen to have trivial monodromy around any loop. The description of the tropicalization results from Proposition \ref{prop-link-trop-correlator}.
\end{proof}

The previous corollary describes the possible tropicalizations of a correlator $0$ log-line bundle, which index the boundary strata of the moduli of spin stable maps: namely the stratum $\M_{\Gamma,\widetilde{K},\varphi}^0(X)(\frac{1}{\delta})$ further decomposes into  components indexed by  $\mathtt{T}$ for $\mathtt{T}\in\TTT$. 

\medskip

We  can now compute the degree of the forgetful map forgetting the spin structure to a given stratum with fixed tropicalization.

\begin{lem}\label{lem-degree-proj-corr0}
    The kernel of the tropicalization restricted to $0$-correlator log-line bundles has cardinality $\delta^{2g-b_1(\Gamma)-2q(X)}|\widetilde{K}|$. In particular, each element of $\TTT$ has precisely $\delta^{2g-b_1(\Gamma)-2q(X)}|\widetilde{K}|$ log-line bundle with correlator $0$ in its preimage.
\end{lem}

\begin{proof}
Assume $f\colon C\to  X$ is in $\M_K(X)$. The image of the correlator map is $K^\perp$. Therefore, the subgroup of log-line bundles $\T\in\LogPic(C)[\delta]$ with $0$ correlator has cardinality
$$|\ker f_*| = \frac{\delta^{2g}}{|K^\perp|}.$$
By Corollary \ref{coro-tropicalization-corr0-strata}, the tropicalization map surjects $\ker f_*$ onto $\TTT$, so to conclude we only need to compute $|\TTT|$.

By Proposition \ref{prop-link-trop-correlator}, $\TTT$ is the orthogonal of $f^*\widetilde{K}$ for the pairing  $\widetilde{W}$ induced by the log Weil pairing. Thus, we have $|\TTT||f^*\widetilde{K}|=\delta^{b_1(\Gamma)}$.

The kernel of $f^*$ being $K$, we have that $f^*\widetilde{K}$ is isomorphic to $\widetilde{K}/K$. By non-degeneracy of the Weil pairing, the cardinality of $\TTT$ is thus $\delta^{b_1(\Gamma)}/|\widetilde{K}/K|$. Finally, the cardinality of $\ker f_*\cap \Pic^{[0]}(C)[\delta]$ is the quotient of $|\ker f_*|$ by $|\TTT|$:
    $$\frac{\delta^{2g}/|K^\perp|}{\delta^{b_1(\Gamma)}/|\widetilde{K}/K|} = \frac{\delta^{2g}|\widetilde{K}|}{|K^\perp||K|\delta^{b_1(\Gamma)}} = \delta^{2g-b_1(\Gamma)-2q(X)}|\widetilde{K}|.$$
\end{proof}

\subsubsection{Stratification of $\M_K^0(X)(\frac{1}{\delta})$}
\label{sec-stratification-MK0delta}
We now have the necessary information to describe the tropicalization of $\M_K^0(X)(\frac{1}{\delta})$ as  a  sub cone stack of the tropicalization of $\M_K(X)(\frac{1}{\delta})$.

\medskip

Recall that the tropicalization of $\M_K(X)(\frac{1}{\delta})$ is the cone $\Sigma_K(X)(\frac{1}{\delta})$ whose cones are indexed by graphs $\Gamma$ with the following decorations:
\begin{itemize}
    \item a class function $(\beta_v)$ (Definition \ref{defi-X-valued-graph});
    \item a global monodromy $\varphi$ with kernel $\widetilde{K}$ and core $K$ (Definition \ref{defi-group-decorated-graph});
    \item an equivalence class of $\delta$-torsion divisor $\mathtt{T}$ (Definition \ref{defi-delta-torsion-graph}).
\end{itemize}
Notice we have the compatibility relation between $\varphi$ and $\mathtt{T}$ from Proposition \ref{prop-link-trop-correlator}. Furthermore, the projection to $\Sigma_K(X)$ forgetting the tropical torsion divisor class $\mathtt{T}$ maps $\sigma_{\Gamma,\widetilde{K},\varphi,\mathtt{T}}$ to $\sigma_{\Gamma,\widetilde{K},\varphi}$ and using Corollary \ref{coro-tropicalization-corr0-strata}, the map between associated strata has degree $\delta^{2g-b_1(\Gamma)}$.

\medskip

Now, the tropicalization $\Sigma_K^0(X)(\frac{1}{\delta})$ of $\M_K^0(X)(\frac{1}{\delta})$ is a sub-cone stack of the previous one. Cones are indexed by graphs $\Gamma$ with the following decorations:
\begin{itemize}
    \item a class function $(\beta_v)$ (Definition \ref{defi-X-valued-graph});
    \item a global monodromy $\varphi$ with kernel $\widetilde{K}$ and core $K$ (Definition \ref{defi-group-decorated-graph});
    \item an equivalence class of $\delta$-torsion divisor $\mathtt{T}$ (Definition \ref{defi-delta-torsion-graph}) satisfying that $\mathtt{T}\in\TTT$, the right kernel of $\varphi$: $\varphi(-,\mathtt{T})=1$.
\end{itemize}
Furthermore, the projection to $\Sigma_K(X)$ maps $\sigma_{\Gamma,\widetilde{K},\varphi,\mathtt{T}}$ isomorphically to $\sigma_{\Gamma,\widetilde{K},\varphi}$, but the map between associated strata has now degree $\delta^{2g-b_1(\Gamma)-2q(X)}|\widetilde{K}|$.

%% file: sec-correlate-DRv2.tex
\section{(Log) Double ramification cycle and its refinement}
\label{sec-correlated-DR-where-evth-comes-to-place}
Our plan is to apply the universal DR-formula to  some suitable logarithmic model of the moduli space of spin stable maps, after restricting to fixed correlator $\theta$ and then, pushing forward to $\M(X),$ identifying the tautological cycle coming from Pixton's formula with $\epsilon_*R\M^\theta(Y|D^\pm).$

\medskip



We will start by recalling the results of \cite{bae2023pixton} and \cite{holmes2025logDR} and in particular recalling  Pixton's formula both in terms of decorated strata classes, and in terms of  piecewise polynomial functions on the tropicalization.

\medskip

Furthermore, to obtain the identification above, we recall the relation between (log) spin and usual DR cycles (see also \cite{chiodo2024double}) and compare the two definition of \emph{correlator map}.

\subsection{Universal double ramification cycle}
\label{sec-uniDR}

Let $\operatorname{Div}_{g,\mathbf{a}}$ Marcus-Wise stack \cite{marcuswiselog} over fs log schemes whose $S$ points parametrize: a log smooth curve $\pi\colon C\to S;$ a section $\alpha\in \rmH^0(S,\pi_*\bar{M}_C^\gp/\bar{M}_S^\gp)$, i.e. $\alpha$ is the data of a strict piece-wise linear function on $\Gamma_s$ with fixed slope $(a_i)=\mathbf{a}$ along the markings,  well-defined up to translation by a constant and compatible with edge contraction. We assume $\sum a_i=0.$
   
By \cite[Theorem~4.2.4]{marcuswiselog}, $\operatorname{Div}_{g,\mathbf{a}}$  is actually representable by a locally of finite type algebraic stack with log structure; in fact it is logarithmically \'etale on $\mathfrak{M}_{g,n}^{\mathrm{log}}.$

\medskip

Let $C/S$ be a log smooth curve over a log algebraic stack  $S$ and let $\L$ a line bundle on $C$ whose restriction to each fiber has  total degree $0$; this data defines a morphism from $S$ to the universal Picard stack $ \mathfrak{Pic}_{g,n,0}.$ 

We call \emph{double ramification locus} $\operatorname{DRL}(\L)$  the following fiber product in the category of  log schemes (\cite[Definition~4.5.2]{marcuswiselog}):
\bcd 
\operatorname{DRL}(\L) \ar[r]\ar[d,"\epsilon"] & \operatorname{Div}_{g,\mathbf{a}}\ar[d,"\mathrm{aj}"] \\
S \ar[r,"\varphi_{\L}"] & \mathfrak{Pic}_{g,n,0}
\ecd
where $\mathrm{aj}$ is defined by sending $\alpha\mapsto \mathcal O_C(\alpha).$  
To be precise, the right vertical morphism takes value in the relative Picard space $\mathfrak{Pic}^{\mathrm{rel}}_{g,n,0}$ parametrizing isomorphism classes of line bundles. This can however be lifted by pull-back to a morphism with values in the Picard stack as explained in \cite[Definition 31]{bae2023pixton}. 
By a slight abuse of notation we keep denoting by $\operatorname{Div}_{g, \mathbf{a}} \xrightarrow{\mathrm{aj}}
\mathfrak{Pic}_{g,n,0}$ the lifted map.

As explained in \cite{bae2023pixton}, we obtain the same double ramification locus and class considering the fiber product

\bcd 
\operatorname{DRL}(\L) \ar[r]\ar[d,"\epsilon"] & \operatorname{Div}_{g, 0}\ar[d,"\mathrm{aj}"] \\
S \ar[r,"\varphi_{\L(-\bfw)}"] & \mathfrak{Pic}_{g,n,0}.
\ecd
This point of view is useful when we want to lift the class to the logarithmic Chow, as done below.

\medskip

As explained in \cite[Section 4.6]{marcuswiselog}, the above fiber product is proper over $S$ and it is naturally endowed with an obstruction theory induced by pull-back from the obstruction theory for the Abel-Jacobi map $\mathrm{aj}.$ 

In particular, if $S$ has a fundamental class or a virtual fundamental class, by virtual pull-back we get a
 class $\vir{\operatorname{DRL}(\L)}$ whose push-forward to $S$ is the so-called DR-cycle \[\operatorname{DRL}(\L)\in\operatorname{CH}_{\dim S-g}(S).\]
 We remark that, implicit in the notation, the class lives in the expected codimension $g;$ $\dim$ should be replaced by $\mathrm{vdim}$ when we work with virtual classes.

The main result of \cite{bae2023pixton} furthermore tells us that 
this class can be computed via the universal DR-cycle formula, i.e.
$$\epsilon_*\vir{\operatorname{DRL}(\L)} = \varphi^*_{\L}\DR\cap [S]^{(\mathrm{vir})},$$
where $\DR\in \operatorname{CH}^{\mathrm{op},g}( \mathfrak{Pic}_{g,n,0}) $  admits a universal  Pixton type formula, i.e. can be written in terms of decorated strata classes, thus providing an explicit expression as tautological cycle 
\cite[Section~0.3.5]{bae2023pixton}. For completeness, we recall the formula.




\subsubsection{DR-cycle in terms of decorated strata classes}
Denote by $\mathfrak{Pic}_{g,n+m}$ the Picard stack, $\pi:\mathfrak{C}_{g,n+m}\to\mathfrak{Pic}_{g,n+m}$ the universal curve, endowed with sections coming from marked points, and $\mathfrak{L}\to\mathfrak{C}_{g,n+m}$ the universal line bundle.
\begin{itemize}[leftmargin=0.4cm]
    \item A \emph{degree decorated} prestable graph is a prestable graph together with a degree decoration function $\deg:\sfV\to\ZZ$. The sum of vertex degrees is the total degree. Let $\GGG_{g,m+n}(d)$ be the set of degree decorated prestable graphs of total degree $d$. If $\Gamma\in \GGG_{g,n+m}(d)$ is a degree decorated prestable graph, we have a morphism
    $$j_\Gamma\colon\mathfrak{Pic}_\Gamma\to\mathfrak{Pic}_{g,n+m},$$
    where $\mathfrak{Pic}_\Gamma$ is the substack of curves with dual graph $\Gamma$ together with a line bundle $L$ whose restriction to the components indexed by $v$ has degree $\deg(v)$. We have the vertex restriction map
    $$\mathfrak{Pic}_\Gamma\to\prod_v \mathfrak{Pic}_{g(v),n(v),\deg(v)},$$
    whose fibers are $\mathbb{G}_m^{b_1(\Gamma)}$-torsors.
    
    \item Let $r$ be a positive integer. If $\Gamma$ is a \emph{degree decorated} prestable graph\footnote{A degree decoration on a pre-stable graph is the same as a graph  together with a divisor $D$ supported at the vertices.} with set of half-edges $\sfH$, a $r$-\textit{weighting} on $\Gamma$, as defined in \cite[Definition~4]{janda2020double}, is a function on the set of half-edges $w:\sfH\to \{0,\dots,r-1\}\simeq\ZZ_r$ such that
        \begin{enumerate}
            \item For every edge $e=(h,h')$ we have $w(h)+w(h')\equiv 0\mod r$.
            \item For every vertex $v$ we have $\sum_{h\vdash v}w(h)\equiv\deg v\mod r$.
            \item For every leg $i$ the weight of the unique adjacent half edge is $a_i$.
        \end{enumerate}
        The set of $r$-weightings on $\Gamma$ is denoted by $W_{\Gamma,r}(\bfw)$. It is a set of cardinality $r^{b_1(\Gamma)}$.
    \item We consider the following cohomology classes on $\mathfrak{Pic}_{g,n+m}$:
        \begin{enumerate}
            \item $\psi_i$ is the first Chern class of the cotangent line bundle at the $i$-th marked point  $q_i$;
            \item $\xi_i=c_1(q_i^*\mathfrak{L})$ 
            \item $\xi=c_1(\mathfrak{L})$ is the first chern class of the universal line bundle, and $\eta=\pi_*(\xi^2)$.
        \end{enumerate}
        In particular, pulling back by the vertex restriction map, such classes also make sense over strata $\mathfrak{Pic}_\Gamma$ for $\Gamma\in G_{g,n+m}(d)$. More precisely, we care about the classes $\psi_h$ for $h\in\sfH$ and $\eta_v$ where $v\in\sfV$.
    \item If $\Gamma\in G_{g,n+m}(d)$ is a degree decorated graph, we call a \textit{cohomological decoration} the choice of the following:
    	\begin{enumerate}
    	\item for each leg $i$ a monomial $\xi_i^a\psi_i^b$;
    	\item for each half-edge $h$ a monomial $\psi_h^a$;
    	\item for each vertex a monomial $\eta^a$.
    	\end{enumerate}
        Given a choice of cohomological decoration $\gamma$, we can consider the class $j_{\Gamma\ast}[\gamma]$ obtained by proper push-forward.
\end{itemize}

For $\Gamma\in G_{g,n+m}(d)$ a degree decorated graph and $w\in W_{\Gamma,r}$ a $r$-weighting. We consider the following class of mixed dimension:
\begin{equation}\label{eq:decoration}
    \gamma_\Gamma(w)=\prod_{i=1}^n e^{\frac{1}{2}a_i^2\psi_i+a_i\xi_i}\prod_{v\in \sfV(\Gamma)}e^{-\frac{1}{2}\eta_v}\prod_{e=(h,h')\in\sfE(\Gamma)} \frac{1-e^{-\frac{w(h)w(h')}{2}(\psi_h+\psi_{h'})}}{(\psi_h+\psi_{h'})}.
\end{equation}
By \cite[Proposition 6]{bae2023pixton},  averaging over all $r$-weightings, one obtains a class whose coefficients are polynomials in $r$ provided $r$ is big enough:
$$\gamma_\Gamma^r=\frac{1}{r^{b_1(\Gamma)}}\sum_{w\in W_{\Gamma,r}(\bfw)}\gamma_\Gamma(w).$$
We take the constant coefficient in $r$ and call the obtained class $\gamma_\Gamma^0$ and consider
$$\sum_{\Gamma\in G_{g,n+m}(d)}\frac{1}{|\operatorname{Aut}(\Gamma)|}j_{\Gamma\ast}[\gamma_\Gamma^0],$$
where $j_{\Gamma\ast}$ is the push-forward of the class Poincar\'e dual to the cohomology class $\gamma_\Gamma^0$. Let $P_{g,\bfw}$ be codimension $g$ term of the above class.

\begin{theo}\cite[Theorem 7]{bae2023pixton}
We have
$$\DR_{g,\bfw}^\mathrm{op} = P_{g,\bfw} \in \operatorname{CH}^g_\mathrm{op}(\mathfrak{Pic}_{g,n+m}).$$    
\end{theo}


\paragraph{\bf{DR with target variety}}
Consider the moduli space of stable maps $\M_{g,n+m}(X,\beta)$, shortly denoted by $\M(X)$, and let $L\in\Pic(X)$ be a line bundle with $c_1(L)\cdot\beta=\sum a_i$. In particular, we have a map $\GGG_{g,n+m}(X,\beta)\to\GGG_{g,n+m}(c_1(L)\cdot\beta)$, transforming the class decoration into a degree decoration where the degree function is $\deg(v)=c_1(L)\cdot\beta_v$. We have the following map to the Picard stack:
$$\Phi_{f^*L}\colon (f:C\to X)\longmapsto (C,f^*L)\in\mathfrak{Pic}_{g,n+m,c_1(L)\cdot\beta}.$$
Applying the universal DR cycle formula \cite{bae2023pixton} recovers the main result in \cite{janda2020double}, expressing the DR-cycle with target defined with rubber maps in terms of decorated strata classes
\[\epsilon_*\vir{R\M(Y|D^\pm)} = \DR^{\mathrm{op}}_{g,\bfw}(\Phi_{f^*L})\cap\vir{\M(X)} \in \operatorname{CH}_{\mathrm{vdim-g}}(\M(X)).\]
In this case, we recover a finite sum since the vertex degree needs to be equal to $c_1(L)\cdot\beta_v$. The terms $j_{\Gamma\ast}[\gamma]$ are to be understood as follows:
$$j_{\Gamma\ast}[\gamma] = j_{\Gamma\ast}\left[\Delta^!(\prod\vir{\M_v}\cap\gamma)\right] = j_{\Gamma\ast}\left[\vir{\M_\Gamma}\cap\Delta^*\gamma\right].$$

\medskip

    In the particular case where $L$ has degree $0$, the formula simplifies since we have on one hand that $\xi_i=0$, and on the other hand that $\eta_v=0$, so that there are no middle term in the strata decoration of (\ref{eq:decoration}):
    $$\DR_{g,\bfw}(X,\beta,L) = \left.\sum_{\Gamma}\frac{1}{|\Aut(\Gamma)|}\frac{1}{r^{b_1(\Gamma)}}\sum_{w\in W_{\Gamma,r}(\bfw)}j_{\Gamma\ast}\left[ \prod e^{\frac{1}{2}a_i^2\psi_i}\prod_{e=(h,h')}\frac{1-e^{-\frac{ww'}{2}(\psi+\psi')}}{\psi+\psi'} \right] \right|_{r=0}.$$

It is proved in \cite{janda2020double} that the class above coincide with 
\[\epsilon_*\vir{R\M(Y|D^\pm)},\]
where $R\M(Y|D^\pm)$ is the moduli space of rubber maps to $Y=\mathbb P_X(\mathcal O\oplus L)$ with divisorial log structure $D^\pm$ given by the fiberwise zero and infinity section. It is explained in \cite[Section~5]{marcuswiselog} that the moduli space $R\M(Y|D^\pm)=\mathrm{Rub}_{g,\bfw}\times_{\mathfrak{P}ic_{g,n,0}}\M(X)$ where  $\mathrm{Rub}_{g,\bfw}\to \mathrm{Div}_{g,\bfw}$ is a proper, birational logarithmic modification. In particular, we have a proper, virtually birational morphism
\[R\M(Y|D^\pm)\to\operatorname{DRL}(\M(X),L).\]
\paragraph{\bf{Spin DR with target variety}}

Let  $\widetilde{\M}(X)(\frac{1}{\delta})$ be the logarithmic modification  of the Spin moduli space of stable maps where the universal $\delta$ (logarithmic  root) of a the trivial line bundle   can be represented by a honest line bundle $\mathcal T$  satisfying
\[\mathcal T^{\otimes\delta}\cong \mathcal O(\alpha),\]
for $\alpha$ a cone-wise linear function over $\widetilde{\mathcal C}\to \widetilde{\M}(X)(\frac{1}{\delta}).$ We choose $\alpha$ to take value zero at each vertex of the non subdivided curve (see \cite[Section~3.6]{BCK}, \cite[Section~6]{holmes2023root}. Then, we have  a morphism

\[ \Phi_\T \colon \widetilde{\M}(X)(\frac{1}{\delta})\longrightarrow \mathfrak Pic_{g,m+n,0} .\]
Applying the universal DR  formula once again, we get a cycle class  \[\DR_{g,\bfw/\delta}(X,\beta,\mathcal T)\in \operatorname{CH}_*(\widetilde{\M}(X)(\frac{1}{\delta}),\mathbb Q)\cong \operatorname{CH}_*(\M(X)(\frac{1}{\delta}),\mathbb Q).\] 
Such class is the degree $g$ component of the constant term of the polynomial in $r$ defined by Pixton's formula upon replacing: 
\begin{itemize}
    \item the set of $X$-valued stable graph with the set  of $\delta$-torsion decorated $X$-valued stable graphs (see Definition~\ref{defi-delta-torsion-graph}). In particular, the set of $r$-\emph{weighting} $W_{\Gamma,r}$ mod $r$ changes depending on the multidegree $\mathcal T$;
    \item the weights $a_i$ are replaced by $a_i/\delta$, the line bundle $F^*\L$ is replaced by $\mathcal T$ and in the second factor of (\ref{eq:decoration}), the product is over the vertices of $\widetilde{\Gamma}$.
    
    We notice that while $F^*\L$ has multidegree $0$ if we have chosen $L\in\Pic^0(X),$ but the universal root $\T$ on $\widetilde{\M}(X)(\frac{1}{\delta})$ does not, so the second factor in the decoration is not trivial.
    \item the decorated strata terms become
    \[j_{\Gamma,\mathrm{T}\ast}[\gamma]= j_{\Gamma,\mathrm{T}\ast}\left[ \vir{\widetilde{\M}_{\Gamma,(\beta_v),\mathrm{T}}}\cap \rho^*\Delta^*\gamma \right]\]
\end{itemize}

\begin{rem}
    Notice that for $L$ the trivial line bundle and $\mathcal T$ its universal $\delta$-root, the morphism $\Phi_{\mathcal O}$ and $\Phi_{\mathcal T}$ from the moduli space of maps/spin map to the Picard stack factor through the moduli space of curves. In particular, we can think as the DR class being obtained pulling back the DR/spin DR from the moduli of curved and then capping with the virtual class.
\end{rem}

\subsubsection{Comparison between DR and Spin DR}\label{sec:spintonormalDR}
Let us denote 
\[\operatorname{Div}^{\Diamond}_{g, \bfw}:= \operatorname{Div}_{g,\bfw}\times_{\Pic_{g,n,0}} \mathcal J_{\mathcal C/\mathcal M(X)},\]

where $\mathcal J_{\mathcal C/\mathcal M(X)}$ denotes the open sub-stack of the universal Picard stack parametrizing multi-degree zero line bundles. After rigidification, this is a separated (non proper) smooth algebraic space over $\M(X)$. 
We can take the fiber product
\[\mathrm{DR}(\M(X),L):=\operatorname{Div}^{\Diamond}_{g, \bfw}\times_{\mathcal J_{\mathcal C/\mathcal M(X)}} \M(X)\]
where $\Phi_{\mathcal L}\colon\M(X)\to J_{\mathcal C/\mathcal M(X)}$ is induced by any $\mathcal L$ with trivial $c_1(\mathcal L)$, e.g the trivial line bundle.
Then $\mathrm{DR}(\M(X))$ is again induced with a virtual class coming by Gysin pull-back along the Abel-Jacobi map and
its push-forward to $\M(X)$ recovers the DR class (See also \cite{holmes2021extending}).  In what follows we omit $L$ from the notation; below we will restrict our attention to the case $\L$ trivial.

Similarly, the class of the Spin double ramification cycle can be obtained as 

\[\mathrm{DR}(\widetilde{\M}(X)(\frac{1}{\delta})):=\operatorname{Div}^{\Diamond}_{g, \frac{\bfw}{\delta}}\times_{\widetilde{\mathcal J}_{\mathcal C/\mathcal M(X)}(\frac{1}{\delta})}\widetilde{\M}(X)(\frac{1}{\delta}).\]

Where  $\widetilde{\mathcal J}_{\mathcal C/\mathcal M(X)}(\frac{1}{\delta})$  comes equipped with get a universal line bundle $\mathcal T_{\mathcal J}$ such that:
\[\mathcal T_{\mathcal J}^{\otimes\delta}=\mathcal L_{\mathcal J}(\alpha_{\mathcal T})\]
on the destabilization $\widetilde{ \mathcal C}\to\mathcal C$. Once again. we remark that $\alpha_{\mathcal T}$ can be chosen to take value $0$ on the vertices of the tropicalization of $\mathcal C.$

\medskip

We  remark that $\widetilde{\mathcal J}_{\mathcal C/\mathcal M(X)}(\frac{1}{\delta})$  can also be interpreted as the moduli stack of $\delta$-torsion line bundle on the $\delta$-stable (orbifold) curves, which is in fact Chiodo's original point of view \cite{chiodo2008stable,chiodo2008towards}.

The above algebraic stacks with log structure fit  into the following  commutative diagram:

\bcd
\mathrm{DRL}(\widetilde{\M}(X)(\frac{1}{\delta}))\ar[r]\ar[d] &\mathrm{Div}^{\Diamond}_{g,\frac{\bfw}{\delta}}\ar[r]\ar[d,"\rho"] & \widetilde{\mathcal J}_{\mathcal C/\mathcal{M}(X)}(\frac{1}{\delta})\ar[d]\\
\mathrm{DRL}(\M(X))\ar[r] &\mathrm{Div}^{\Diamond}_{g,\bfw}\ar[r] & \mathcal J_{\mathcal C/\mathcal M(X)},
\ecd
where the left vertical map is defined by pull-back from the middle vertical map, which is in turn defined as follow.
Given a point $S$ point of $\mathrm{Div}^{\Diamond}_{g,\frac{\bfw}{\delta}}$, we have $\alpha\in\rmH^0(\widetilde{C},\bar{M}^{\gp}_{\widetilde{C}}\slash \bar{M}_S^{\gp})$ such that $\mathcal T\cong \mathcal L_{\mathcal J}(\alpha)$ and
taking the $\delta$-power we get an isomorphism of  bundles:
\[\mathcal T^{\otimes\delta}\cong \mathcal L_{\mathcal J}(\alpha_{\mathrm{T}})\cong \mathcal O_{\widetilde{C}}(\delta\alpha)\]
where $\alpha_{\mathrm{T}}$ is the piece-wise linear function on $\widetilde{\Gamma}$ discussed above. 
Then $\delta\alpha-\alpha_{\mathrm{T}}\in\rmH^0(\widetilde{C},\bar{M}^{\gp}_{\widetilde{C}}\slash\bar{M}_S^{\gp})$ has fixed slope $a_i$ along the unbounded edges corresponding to markings carrying log structure (this is because $\alpha_{\mathrm{T}}$ is really a strict PL function on the bounded part of the tropical curve) and we have an isomorphism 
\[\mathcal L_{\mathcal J}\cong\mathcal O_{\widetilde{C}}(\delta\alpha-\alpha_{\mathrm{T}}).\]
Since the universal line bundle over the Jacobian has multi-degree zero on each geometric fiber, the above isomorphism implies that
 on all exceptional component, i.e. for each component of $\widetilde{C}\setminus C$ the sum of the outgoing slopes of $\delta\alpha-\alpha_{\mathrm{T}}$ at the corresponding $v\in V(\widetilde{\Gamma})$ is zero and thus
$\delta\alpha-\alpha_{\mathrm{T}}=\rm st^*\beta$ for some $\beta\in \rmH^0(C,\bar{M}^{\gp}_{C}\slash\bar{M}_S^{\gp}).$
In conclusion, the middle vertical morphism is given by

\[(\widetilde{C}\to C,\mathcal T,\alpha)\mapsto (C,\beta).\]

Notice that its restriction to the rubber spaces (these are called double ramification loci in \cite[Section~3]{holmes2025dr} where the same comparison argument is presented) is proper.

\medskip

We furthermore see that on the locus of smooth curves the vertical middle map is, up to forgetting the $\mu_{\delta}$-gerbe structure coming from the automorphism of the root $\mathcal O_C(\sum\frac{a_i}{\delta})$, an isomorphism. Since both source and target of $\rho$ are logarithmically smooth and  irreducible their fundamental classes are the natural ones and we get, omitting a factor of $\frac{1}{\delta}$,
\[\rho_*[\mathrm{Div}^{\Diamond}_{g,\frac{\bfw}{\delta}}]=[\mathrm{Div}^{\Diamond}_{g,\bfw}].\]

Since the virtual classes on $\mathrm{DRL}(\widetilde{\M}(X)(\frac{1}{\delta}))$ and $\mathrm{DRL}(\M(X))$ are induced by Gysing pull-back by the latter, it follows that

\[\rho_*[\mathrm{DRL}(\widetilde{\M}(X)(\frac{1}{\delta}))]^{vir}=[\mathrm{DRL}(\M(X)]^{vir}.\]

In particular, since the map $\widetilde{\M}(X)(\frac{1}{\delta})\to \widetilde{\M}(X)$ is \'etale, projecting the class associated to a strata indexed by $(\Gamma,\mathrm{T})$ we obtain the class of the image multiplied by the degree of the projection, which is $\delta^{2g-b_1(\Gamma)}$. This provides a second  (not obvious) expression for the DR-cycle with target variety:
$$\DR_{g,\bfw}(X,\beta,L) = \sum_{\Gamma,\mathtt{T}}\frac{\delta^{2g-b_1(\Gamma)}}{|\operatorname{Aut}(\Gamma)|}j_{\Gamma\ast}\left[ \prod e^{\frac{1}{2}(\frac{a_i}{\delta})^2\psi_i}\prod e^{-\frac{1}{2}\eta_v}\prod\frac{1-e^{-\frac{ww'}{2}(\psi+\psi')}}{\psi+\psi'} \right].$$

\subsubsection{DR-cycle in terms piecewise polynomial functions}
As in \cite[Section~6]{holmes2023root,holmes2025logDR}, the Pixton formula also has an expression in terms of strict piecewise polynomial functions.
In particular, both classes $\DR_{g,\bfw}(X,\beta,L)$ and $\DR_{g,\bfw/\delta}(X,\beta,\mathcal T)$ can be written in terms of  strict piecewise polynomial functions.
We now explain how to suitably modify the formula given in \cite[Section~6]{holmes2023root} to get the formula for the case of Spin DR cycle with target varieties.

\medskip

The relation between the formula given in the previous section and this one follows directly from the translation between piecewise polynomials on the tropicalization and decorated strata classes carefully explained in \cite[Section~6]{holmes2025logDR}.

\medskip

Recall  that $\Sigma(\frac{1}{\delta})$ is the   cone stack interpreted as  tropical moduli space of tropical curves together with an equivalence class of $\delta$-torsion divisor. Its cones are indexed by $(\Gamma,\mathtt{T})$, where $\Gamma$ is a prestable graph and $\mathtt{T}$ goes over equivalence classes of $\delta$-torsion divisors on $\widetilde{\Gamma}$.

We can interpret strict piece-wise polynomial on $\Sigma(\frac{1}{\delta})$ as operational Chow classes on the moduli space of spin stable maps via the morphism
\[\widetilde{\M}(X)(\frac{1}{\delta}) \to \widetilde{\mathfrak{M}}(\frac{1}{\delta})\xrightarrow{\widetilde{t}_{\delta}}\Sigma(\frac{1}{\delta}). \]

The first arrow is simply the forgetful morphism, and $\widetilde{t}_{\delta}$ is a morphism of (in fact smooth) algebraic stack identifying the cone stack with its Artin fan \cite{cavalieri2020moduli}.\footnote{For the reader who is not familiar with Artin fans and their  cohomology, one can think this is a generalization of the  correspondence for toric variety between strict piece-wise polynomials on the fan and toric strata.}

We write down the operational class $\DR^\mathrm{spin}\in\mathrm{CH}^*(\widetilde{\mathfrak{M}}(\frac{1}{\delta}))$ which we can express as the product of a strict piece-wise polynomial in $\mathrm{sPP}(\Sigma(\frac{1}{\delta}))$ and a tautological class. 

The spin DR cycle with target variety $X$ in the Chow group of $\widetilde{\M}(X)(\frac{1}{\delta})$ can then be written as the cap product of $\DR^{\mathrm{spin}}$ with the virtual class $\vir{\widetilde{\M}(X)(\frac{1}{\delta})}$.

 \begin{itemize}[leftmargin=0.5cm]
     \item We define a first piecewise polynomial function as follows. On each cone $\widetilde{\sigma}_{\Gamma,\mathtt{T}}$, we consider the function \cite[Section 1.7.3]{holmes2025logDR}
     \[\operatorname{Cont}^r_{\Gamma,\mathtt{T}}=\sum_{w\in W_{\Gamma,r}(\bfw/\delta)} \frac{1}{r^{b_1(\Gamma)}}\prod_{e=(h,h')\in \sfE(\Gamma)}\operatorname{exp}\left(\frac{w(h)\cdot w(h')}{2}l_e\right)\in\QQ[[l_e\; :\;e\in\sfE(\Gamma)]]\]
     where the sum runs over the $r$-weightings. As before, this is a polynomial in $r$ for $r$ sufficiently large. Then, following the notation of \cite{holmes2025logDR}, we denote by $\operatorname{Cont}_{\Gamma,\mathtt{T}}$ the value at $r=0$ of such polynomial, which is then a formal series in the variables $l_e$. The assumption on the choice of $\mathrm T$ ensures that all these cone functions glue together to define $\mathfrak{P}\in\mathrm{sPP}(\Sigma(\frac{1}{\delta}))$, the strict piecewise polynomial with
     \[\mathfrak P\rvert_{\widetilde{\sigma}_{\Gamma,\mathtt{T}}}=\operatorname{Cont}_{\Gamma,\mathtt{T}}\]
     \item We also consider $\mathfrak{L}\in \mathrm{sPP}(\Sigma(\frac{1}{\delta}))$  the function that restricted to  $\widetilde{\sigma}_{\Gamma,\mathtt{T}}$ is defined by:

     \[\mathfrak{L}\rvert_{\widetilde{\sigma}_{D,\Gamma}}=\frac{1}{\delta^2}\sum_{v\in \sfV(\widetilde{\Gamma})}\alpha_{\mathrm{T}}(v)\underline{\mathrm{deg}}(\mathcal O(\alpha_{\mathrm{T}}))(v)\]
     where we have chosen $\alpha_{\mathrm{T}}$ is such a way that $\alpha_{\mathrm{T}}(v)=0$ for all $v\in \sfV(\Gamma)\subset\sfV(\widetilde{\Gamma}),$ which is always doable according to \cite[Section 6]{holmes2023root}. As the multidegree is divisible by $\delta$, the product under the sum belongs to the right monoid.

     \item Finally, let us consider the tautological  class
     \[\eta=-\sum_{i=1}^n\frac{a_i^2}{\delta^2}\psi_i.\]
    \end{itemize}


 Then we have an operational class in $\mathrm{CH}^*(\widetilde{\mathfrak{M}}(\frac{1}{\delta}))$ given by:

 \[{\bf{P}}_{g,\bfw/\delta}=\operatorname{exp}(-\frac{1}{2}(\eta+\widetilde{t}_{\delta}^*\mathfrak{L}))\cdot \widetilde{t}_{\delta}^*\mathfrak{P}.\]

 It is proved in \cite[Section~6]{holmes2023root} that the degree $g$ part ${\bf{P}}_{g,\bfw/\delta}^g$ of this operational class is precisely the pull-back of the universal DR class via
 \[\Phi_\T \colon \widetilde{\mathfrak{M}}(\frac{1}{\delta})\longrightarrow\mathfrak{Pic}_{g,n+m,0}.\]

The usual compatibility of the virtual class for moduli of maps  $\vir{\widetilde{\M}(X)(\frac{1}{\delta})}$ with the boundary stratification of $\widetilde{\mathfrak{M}}(\frac{1}{\delta})$ (see \eqref{eq:virtual compatibility}),
gives the following identity
\[\mathrm{ft}^*{\bf{P}}_{g,\bfw/\delta}^g\cap\vir{\widetilde{\M}(X)(\frac{1}{\delta})} = \DR_{g,\bfw/\delta}(X,\beta,\mathcal T)\]

\subsection{Correlated Pixton formula}\label{sec:correlatedPixton}

We now write and prove the correlated Pixton formula expressing the part of the DR-cycle with target variety having $0$-correlator.

We consider again the space of log-$\delta$-roots of $f^*L$, which become true line bundles after base change to a suitable root stack, along with the morphism
\[ \Phi_\T \colon \widetilde{\M}(X)(\frac{1}{\delta})\longrightarrow \mathfrak Pic_{g,m+n,0} .\]
To obtain the correlated DR formula, we restrict to the stratum of log-line bundles with correlator $\theta=0$, denoted by $\widetilde{\M}^0(X)(\frac{1}{\delta})$. We there apply the universal DR formula and push to $\widetilde{\M}(X)$. To do so, we need to know the stratification of this space. Section \ref{sec-stratification-MK0delta} provides an explicit description, provided we restrict to a given $\M_K(X)$.

\medskip

To express the formula, we provide some piecewise polynomial functions on the following cone stacks:
\begin{itemize}
    \item the tropicalization $\Sigma_K(X)$ of $\M_K(X)$, whose cones are indexed by $(\Gamma,\widetilde{K},\varphi)$;
    \item the tropicalization $\Sigma_K^0(X)(\frac{1}{\delta})$ of $\M_K^0(X)(\frac{1}{\delta})$, whose cones are indexed by $(\Gamma,\widetilde{K},\varphi,\mathtt{T})$ with $\mathtt{T}$ in the right-kernel of $\varphi$.
\end{itemize}
We have a map $\Sigma_K^0(X)(\frac{1}{\delta})\to\Sigma(\frac{1}{\delta})$, so that we can pull-back piecewise polynomial functions on $\Sigma(\frac{1}{\delta})$. We consider the following ones:
\begin{itemize}
    \item We define the piecewise polynomial function $\mathfrak{P}$ on the cone indexed by $(\Gamma,\widetilde{K},\varphi,\mathtt{T})$ function by the formula
    $$\mathfrak{P}|_{\widetilde{\sigma}_{\Gamma,\widetilde{K},\varphi,\mathtt{T}}} = \left.\sum_{w\in W_{\widetilde{\Gamma},r}(\bfw/\delta)}\frac{1}{r^{b_1(\Gamma)}}\prod_{e=(h,h')} e^{\frac{w(h)w(h')}{2}l_e}\right|_{r=0},$$
    where the sum is over weightings with divergence $\mathrm T$, the preferred divisor in the class $\mathtt{T}$, and slopes at infinity $\bfw/\delta$. The expression is a polynomial in $r$ for $r$ big enough, the $|_{r=0}$ means we take the constant coefficient.
    
    \item We define the piecewise linear function $\mathfrak{L}$ on the cone indexed by $(\Gamma,\widetilde{K},\varphi,\mathtt{T})$ function by the formula
    $$\mathfrak{L}|_{\widetilde{\sigma}_{\Gamma,\widetilde{K},\varphi,\mathtt{T}}} = \frac{1}{\delta^2}\sum_{v\in\sfV(\widetilde{\Gamma})}\alpha_{\mathrm{T}}(v)\cdot\delta \mathrm T(v),$$
    where $\delta \mathrm T(v)$ is actually the degree of $\O(\alpha_{\mathrm{T}})$ at $v$.
    
    \item We consider the piecewise polynomial function $e^{-\frac{1}{2}\mathfrak{L}}\cdot\mathfrak{P}$. Its expression may seem to depend on the choice of representatives $\mathrm T$ as both factor do, but \cite[Lemma 6.4]{holmes2023root} proves it is actually not the case, and it only depends on $\mathtt{T}$.
\end{itemize}

We have a map of fans $\Sigma_K^0(X)(\frac{1}{\delta})\to\Sigma_K(X)$ forgetting the tropical torsion information. Thus, we can push this last piecewise polynomial functions using the degree of the projection between strata: we define a piecewise polynomial function on the cone $\sigma_{\Gamma,\widetilde{K},\varphi}$ of $\Sigma_K(X)$ summing over the cones $\sigma_{\Gamma,\widetilde{K},\varphi,\mathtt{T}}$ of $\Sigma_K^0(X)(\frac{1}{\delta})$ for all $\mathtt{T}$ in the right kernel of $\varphi$. We get
$$\mathfrak{DR}_K|_{\sigma_{\Gamma,\widetilde{K},\varphi}} = \delta^{2g}\frac{1}{\delta^{b_1(\Gamma)}|\widetilde{K}^\perp|}\sum_{\mathtt{T}\in\TTT} e^{-\frac{1}{2}\mathfrak{L}|_{\widetilde{\sigma}(\Gamma,\widetilde{K},\varphi,\mathtt{T})}}\mathfrak{P}|_{\widetilde{\sigma}(\Gamma,\widetilde{K},\varphi,\mathtt{T})}.$$

We also provide a formula for the general correlated DR-cycle, with values in the group algebra $\QQ[\Alb(X)]$, such that the $(\theta)$-coefficient for $\theta\in E[\delta]$ provides the value for correlator $\theta$. To this end, for $\varphi\colon \rmH_1(\Gamma,\ZZ_\delta)\to \Alb(X)[\delta]/\widetilde{K}^\perp$ and $\mathtt{T}\in\rmH_1(\Gamma,\ZZ_\delta)$, we set $T_{\varphi(\mathtt{T})}=\frac{1}{|\widetilde{K}^\perp|}\sum_{\theta\in\varphi(\mathtt{T})}(\theta)$ to be the group algebra average of the elements in the $\widetilde{K}^\perp$-torsor $\varphi(\mathtt{T})$. We then set
$$\underline{\mathfrak{DR}}_K|_{\sigma_{\Gamma,\widetilde{K},\varphi}} = \delta^{2g-b_1(\Gamma)}\sum_{\mathtt{T}} e^{-\frac{1}{2}\mathfrak{L}|_{\widetilde{\sigma}(\Gamma,\widetilde{K},\varphi,\mathtt{T})}}\mathfrak{P}|_{\widetilde{\sigma}(\Gamma,\widetilde{K},\varphi,\mathtt{T})} \cdot T_{\varphi(\mathtt{T})}.$$
We recover the previous expression taking the $(0)$-coefficient, for which we only have to sum for $\mathtt{T}\in\TTT$ and not all torsion divisor classes $\mathtt{T}$.

\begin{lem}
The function $\mathfrak{DR}_K$ and $\underline{\mathfrak{DR}}_K$ on $\Sigma_K(X)$ is well-defined in the sense that the expressions given on each cone correctly glue together.
\end{lem}

\begin{proof}
The piecewise polynomial function is just the push-forward from a piecewise polynomial function on the tropicalization of the spin space, and is therefore well-defined. The factor $\delta^{2g-b_1(\Gamma)}\cdot T_{\varphi(\mathtt{T})}$ is the group algebra element whose $(\theta)$-coefficient is the degree when restricted to the component with correlator $\theta$.

It can also be checked by hand that the functions defined on each cone glue together. To do so, it suffices to see the definition on each cone are compatible under edge contraction. Noting happens when a contracted edge has distinct extremities, so one only needs to consider the case of a contracted loop. If $q\colon\Gamma\to\Gamma'=\Gamma/e$ is a loop contraction, all tropical divisors differing by a multiple of the loop are mapped to the same image by $q$.

We already explained what happens for $\widetilde{K}$ and $\varphi$: $\widetilde{K}'$ is the orthogonal of the loop $e$ in $\widetilde{K}$ for $\varphi$ and $\varphi'$ is the induced pairing between $\widetilde{K}'$ and $\rmH_1(\Gamma,\ZZ_\delta)$, or in other words the composition $\varphi'\colon\rmH_1(\Gamma,\ZZ_\delta)\to\Alb(X)[\delta]/\widetilde{K}^\perp\to\Alb(X)[\delta]/(\widetilde{K}')^\perp$. The degree equality then just comes from
$$T_{\varphi'(\mathtt{T}')} = \sum_{\mathtt{T}\colon q(\mathtt{T})=\mathtt{T}'}T_{\varphi(\mathtt{T})}.$$
Concretely, it means that the correlator classes were sorted out according to their value on the loop, and the contraction of the latter forgets the sorting, putting them back together.
\end{proof}

We now use this piecewise polynomial function to express the correlator DR-cycle.

\begin{theo}\label{theo-correlated-DR-formula}
As a strict piecewise polynomial function, the correlator $0$ part of the DR-cycle has the following expression:
    $$\DR_{g,\bfw}^0(X,\beta) = \sum_{K} e^{-\frac{1}{2}\sum(\frac{a_i}{\delta})^2\psi_i}\mathfrak{DR}_K,$$
    where each summand yields a class in $\M_K(X)$, and we take the degree $g$ part. The total correlated DR-cycle, with values in the group algebra $\QQ[\Alb(X)]$ whose $(\theta)$-coefficient gives the value for correlator $\theta$ is given by
    $$\underline{\DR}_{g,\bfw}(X,\beta) = \sum_{K} e^{-\frac{1}{2}\sum(\frac{a_i}{\delta})^2\psi_i}\underline{\mathfrak{DR}}_K,$$
    where each summand yields a class in $\M_K(X)$, and we take the degree $g$ part.
\end{theo}

\begin{proof}
    The moduli space of stable maps splits as the disjoint union of the $\M_K(X)$. Therefore, we may restrict to a $\M_K(X)$ for a fixed choice of $K$ and later sum over the possible choices of $K$.
    
    By construction, the refined boundary strata of $\M_K(X)$ are indexed by monodromy decorated graphs $(\Gamma,\widetilde{K},\varphi)$, where
    \begin{itemize}
    \item $\Gamma$ is a $X$-valued stable graph,
    \item $\widetilde{K}$ is a subgroup of $\Pic^0(X)$ containing $K$,
    \item $\varphi\colon\widetilde{K}/K\hookrightarrow \rmH^1(\Gamma,\mu_\delta)$ is an injective group morphism.
    \end{itemize}

Given a stratum $(\Gamma,\widetilde{K},\varphi)$, the strata of $\M_K(X)(\frac{1}{\delta})$ over the latter are indexed by the additional choice of a tropical torsion divisor, of which there are $\delta^{b_1(\Gamma)}$. The projection restricted to each of them has degree $\delta^{2g-b_1(\Gamma)}$.

Restricting to the correlator $\theta=0$, by Corollary \ref{coro-tropicalization-corr0-strata}, only strata with $\mathtt{T}\in\TTT$ contribute. Therefore, strata are described by $(\Gamma,\widetilde{K},\varphi,\mathtt{T})$, where $\mathtt{T}$ is a $\delta$-torsion tropical divisor in the left kernel $\TTT$ of $\varphi$. By Lemma \ref{lem-degree-proj-corr0}, the degree of the projection is now $\delta^{2g-b_1(\Gamma)}/|\widetilde{K}^\perp|$. Putting all this information together, we recover the formula. Alternatively, we use the group algebra valued polynomial function which describes the possible correlators for each stratum.
\end{proof}

\begin{rem}
    For the divisor $\mathrm T$, one may either take the one provided by the construction extending the log-line-bundle $\T$ as a true line bundle, or take any $\mathrm T$ in the linear equivalence class. Indeed, according to \cite[Lemma 6.4]{holmes2023root}, the picewise polynomial function does not depend on the specific choice of $\mathrm T$.
\end{rem}

Due to deformation invariance, in cases where the mapping class group of $X$ acts transitively on torsion elements of $\Alb(X)$ with a given order, the $(\theta)$-coefficient of $\underline{\DR}_{g,\bfw}(X,\beta)$ only depends on the order of $\theta$. Under this assumption, the formula for $0$-correlator and all possible values of $\delta$ are actually sufficient to recover the classes for non-zero correlators. This uses what we call \textit{unrefinement} relations, that relate correlation at level $\delta$, to the correlation at another level $\delta'$ dividing $\delta$. This is the case when $X$ is an elliptic curve for instance.

\subsection{Comparison of correlators of stable relative maps and log-line bundles}
\label{sec-link-correlators-line-bundle-stable-maps}
So far, we have been using the homomorphism
\[F_*\colon  \LogPic^0(\mathcal{C}\slash \M(X)({\frac{1}{\delta}}))\to\Alb(X)\times \M(X)({\frac{1}{\delta}}),\]
to define a \textit{correlating} map from the moduli space of maps with spin structure to the Albanese variety
\begin{align*}
  \M(X)({\frac{1}{\delta}})&\longrightarrow\Alb(X)\times \M(X)({\frac{1}{\delta}}),\\
   [f,\mathcal T]&\longmapsto f_*\mathcal T
\end{align*} 
where $f_*$ simply denotes the restriction of $F_*$ to a point of the moduli space. Since $\mathcal T$ is $\delta$-torsion, the image actually lies in the discrete $\delta$-torsion subgroup $\Alb(X)[\delta];$ and its value is what we defined to be the correlator $\theta$ of the $[f,\mathcal T]$.
Correlators being locally constant, there is a decomposition of $\M(X)({\frac{1}{\delta}})$ into open and closed substacks
\[ \M(X)({\frac{1}{\delta}})=\bigsqcup_{\theta}  \M^{\theta}(X)({\frac{1}{\delta}}).\]
This induces a refinement of the spin DR-cycle and thus of its push-forward along the forgetful morphism, which is the usual DR-cycle.

\medskip

On the other hand, as we already recalled, by \cite{janda2020double} and \cite{marcuswiselog}, we have
\[\DR_{g,\bfw}(X,\beta,L)=\epsilon_*\vir{R\M(Y|D^\pm)}=\epsilon_*\vir{\mathrm{DRL}(\M(X))}.\]

In \cite{blomme2024correlated} we defined a refinement of such virtual class using the a priori different definition of the correlating map, which does not require talking about logarithmic Picard group and log Abel-Jacobi map. We considered

\begin{align*}
    \kappa^\delta\colon R\M(Y|D^\pm)&\longrightarrow\Alb(X) ,\\
    [(f\colon (C,\underline{p},\underline{q})\to Y^{\text{rub}}]&\longmapsto\sum\frac{w_i}{\delta}a_X\circ f(q_i)
\end{align*}
where $a_X$ denote the Albanese map. We saw that $\kappa^\delta$ takes value in the discrete subgroup pf $\delta$-torsion elements of the Albanese.
So this also determines a decomposition of $R\M(Y|D^\pm)$ into open and closed substacks
\[R\M(Y|D^\pm)=\bigsqcup_{\theta} R\M^\theta\]
 and thus a refinement 
 \[\DR_{g,\bfw}(X,\beta,L)=\sum \epsilon_*\vir{R\M^\theta}.\]

We now prove that these two refinements do coincide.

\begin{prop}
    The correlator $\kappa^\delta$ from \cite{blomme2024correlated} and from the present paper $f_*\T$ agree on the DR-locus.
\end{prop}

\begin{proof}
Let us denote by $R\M(\mathcal T)$ the virtual birational model of $\mathrm{DRL}(\widetilde{\M}(X)(\frac{1}{\delta}))$ obtained taking the fiber product with $\mathrm{Rub}_{g,\bfw}$ and projecting to $R\M(Y|D^\pm).$ As showed in Section~\ref{sec:spintonormalDR} this comes endowed with a virtual class, which pushes forward to the usual virtual class on $R\M(Y|D^\pm).$
The key observation is now that when we consider the homomorphism 
\[F_*\colon  \LogPic^0(\mathcal{C}\slash R\M(\mathcal T)\to\Alb(X)\times R\M(\mathcal T),\]
we have that
\[F_*\mathcal T=F_*\mathcal O_C\left(\sum_{i=1}^n \frac{w_i}{\delta} q_i\right).\]
since the isomorphism  of line bundles $\mathcal T\cong\mathcal O_{\widetilde{C}}(\alpha)$ implies that $\mathcal O_C(\sum_{i=1}^n \frac{w_i}{\delta} q_i)$ differ by $\alpha^v\in\rmH^0(\widetilde{C},\bar{M^v}^{\gp}_{\widetilde{C}})$ where $M^v_{\widetilde{C}}$ denotes the vertical part of the log structure (alias, in the tropicalization of $\widetilde{C}$ with respect to the vertical structure  there are only bounded edges).\footnote{We also remark one more time that whenever we talk about the logarithmic Picard group we always considered a log curve only with its vertical log structure}.

On the other hand, $\mathcal O_C(\sum_{i=1}^n \frac{w_i}{\delta} q_i)$ is in the image of the log Abel Jacobi  map defined in Section~\ref{sec:logabeljacobi}. Indeed, independently from the choice of strict section $\sigma_0\in\mathcal C$ we have 
\[ \mathcal O_C(\sum_{i=1}^n \frac{w_i}{\delta} q_i)=\bigotimes \operatorname{logAJ}( q_i-\sigma_0)^{w_i/\delta}.\]
Then remembering that $\LogPic^0(\mathcal C\slash R\M(\mathcal T))$ with $\mathrm{logAJ}$ satisfied the Albanese universal property (see Proposition~\ref{prop:logalb}), i.e.
 \[a_X\circ F=\operatorname{logAJ}\circ F_*\]
we have that 
\[F_*\mathcal T= \sum\frac{w_i}{\delta}a_X\circ f(q_i)= \kappa^{\delta}( [F\colon (C,\underline{p},\underline{q})\to Y^{\text{rub}} ]) . \]
\end{proof}

In other words, the two maps  $R\M(\mathcal T)\to\Alb(X)[\delta]$ defined by $[f, \mathcal T]\mapsto f_*\mathcal T$ and $\kappa^\delta\circ \mathrm{ft}$
for $\mathrm{ft}\colon  R\M(\mathcal T)\to R\M$ coincide and consequently so do the decomposition of the induced decomposition of these moduli spaces and their virtual classes as open and closed substacks.

\subsection{Lifting to logarithmic DR cycle}\label{sec:loglift}
For the usual double ramification cycle on $\bar{\M}_{g,n},$ the definition provided by Holmes  \cite{holmes2021extending} via resolutions of the indeterminacy of the Abel-Jacobi map   naturally yields DR classes on logarithmic blow-ups of $\M^\diamond$, and thus a class $\mathrm{logDR}(\mathcal O(\bfw))\in\mathrm{logChow} (\bar{\M}_{g,n})$ which push-forward to cycle considered above.

\begin{rem}
We also refer the reader to \cite{kass2019stability} for an alternative approach to the extension of DR-cycles using an extended Brill–Noether
locus on a suitable compactified universal Jacobians (also constructed in \cite{kass2019stability}). 
\end{rem}

By \cite{holmes2025logDR}, it is possible to find an explicit logarithmic modification of $\bar{\M}_{g,n},$ and an explicit twist of the line bundle $\mathcal O(\bfw)$ on a quasi-stable curve over the latter, such that the class
$\mathrm{logDR}(\mathcal O(\bfw))$ can be computed via the UniDR cycle recalled above. 

Such an explicit logarithmic modification is constructed using small non degenerate stability conditions, i.e. we can take 
$\M^\diamond=(\bar{\M}_{g,n}^\vartheta)_{\mathcal O(\bfw)}$ and consider the morphism to the universal Picard stack induced by the $\vartheta-$stable twist 
$\mathcal O(\bfw)(\alpha^{\vartheta}).$ Then, the pull-back of the operational DR-class recalled above capped with the fundamental class is a representative of 
$\mathrm{logDR}(\mathcal O(\bfw)).$

The class admits the following expression:

\[\mathrm{logDR}(\mathcal O(\bfw))=\sum_{(\Gamma^{\vartheta}_{\bfw},D_{\bfw},\alpha^{\vartheta})} [\M_{\Gamma^{\vartheta}_{\bfw},D_{\bfw},\alpha^{\vartheta}}]\cap\operatorname{Dec}(\underline{\deg}\mathcal O(\bfw)(\alpha^{\vartheta})),\]

where $(\Gamma^{\vartheta}_{\bfw},D_{\bfw},\alpha^{\vartheta})$ indexes the logarithmic strata, and we use the shorthand notation $\operatorname{Dec}(\underline{\deg}(\mathcal O(\bfw)(\alpha^{\vartheta}))$ to denote the polynomial with $\psi,\eta_v$ classes whose explicit shape only depends on the quasi-stable curve and the twist of the line bundle.

\medskip

Similarly, the Spin DR class admits a lift to $\mathrm{logDR}(\mathcal T)\in\mathrm{logCH} (\bar{\M}_{g,n}(\frac{1}{\delta}))$. Furthermore, as proved in \cite[Proposition~5.6]{BCK}, constructing the moduli space of log roots and the representable twist $\mathcal T$  as in \cite[Section~3.6.3]{BCK} starting from the quasi-stable curve $C^\vartheta\to(\bar{\M}_{g,n}^\vartheta)_{\mathcal O(\bfw)}$ equipped with the $\vartheta$-stable twist $\mathcal O(\bfw)(\alpha^{\vartheta})$ we obtain an explicit  representative for $\mathrm{logDR}(\mathcal T).$

\medskip

Finally, as proved in \cite[Proposition~5.15]{BCK} it is sill the case that pushing forward along the forgetful morphism $\M^\diamond(\frac{1}{\delta})\xrightarrow{\rho}\M^\diamond,$ for some logarithmic modification $\M^\diamond$  we get an identification:

\[\rho_*\mathrm{logDR}(\mathcal T)=\mathrm{logDR}(\mathcal O(\bfw)).\]

\medskip

Let $\M(X)$ be a moduli space of stable maps, and assume that the stabilization morphism $\mathrm{ft}\colon\M(X)\to\bar{\M}_{g,n}$ is an isomorphism on the source curve of the map. This is for example the case when $X$ is an elliptic curve, which is the case we need for the application in \cite{blomme2025multiple}, or more generally an abelian variety of any dimension.

As explained at the beginning of Section~\ref{sec-stratification-moduli-of-roots}, the virtual class $\vir{\M(X)}$ lifts to a class in $\mathrm{logCH}(\M(X)).$ We give the following definition:

\begin{defi}
    Under the assumption above we define the log DR cycle with target variety to be \[\mathrm{logDR}(\M(X),\mathcal O(\bfw))=\mathrm{ft}^*\mathrm{logDR}(\mathcal O(\bfw))\cap \vir{\M(X)}\in \mathrm{logCH}(\M(X))\]
    and similarly for the spin log DR cycle with target $\mathrm{logDR}(\M(X)(\frac{1}{\delta},\mathcal T).$
\end{defi}
\begin{rem}
In \cite{holmes2022intersection}, the author show that it is possible to define a class $\mathrm{logDR}\in\mathrm{logCH}^*(\mathfrak{P}ic_{g,n})$ one can thus define more in general for any map 

\[\Phi_{\mathcal L}\colon\M(X)\to \mathfrak{P}ic_{g,n}\]
 a log DR class with target variety as the pull-back of the latter. The question then becomes how to compute such class explicitly. This requires adapting the almost-twistability criteria of \cite{holmes2022intersection,holmes2025logDR} developed for the case of log curves and line bundles over log smooth stack with log structure, to the virtually smooth setting. 
We will return to this question in later projects to not add too much to this already long paper.
\end{rem}

The correlated DR class can be lifted to a class in the logarithmic Chow ring. We leave to the reader to carry out the straightforward adaptation of Section~\ref{sec:correlatedPixton} starting from Formula~(19) in \cite{holmes2025logDR} and using the discussion  carried out in Section~\ref{sec-various-moduli-of-curves-including-spin} about how the virtual stratification of maps and spin maps behave with respect to logarithmic modifications.

%% file: sec-explicit-computation.tex
\section{Explicit computations}

We now use the formula to compute some correlator $0$ invariants when the base $X$ is an elliptic curve $E$. We recover the result from \cite{blomme2024correlated} for point insertions and compute its refinement by inserting a $\lambda$-class.

\subsection{Point insertions}

We consider the correlated Gromov-Witten invariants with the following data: $X=E$ is an elliptic curve and $Y=E\times\PP^1$ the total space of the trivial bundle, $\beta=d[E]$, $g=1$, $m=1$ interior marked point and $n$ log-points with associated weights $a_i$. We then also set $a_0=0$. We also choose a common divisor $\delta$ of the weights $a_i$ to consider the level $\delta$ refinement. Insertions are chosen to be points for every point except for the first log point, whose weight is $a_1$. The non-correlated invariant is denoted as follows:
$$N_d(\bfw) = \gen{\pt_0,1_{a_1},\pt_{a_2},\dots,\pt_{a_n}}^{Y/D^\pm}_{1,d[E]},$$
and the correlator $0$ counterpart for level $\delta$ is denoted by:
$$N^0_d(\bfw) = \gen{\pt_0,1_{a_1},\pt_{a_2},\dots,\pt_{a_n}}^{Y/D^\pm,0}_{1,d[E]}.$$

\medskip

To provide explicit formulas for these invariants, we use the following arithmetic functions:
\begin{itemize}
    \item the second Jordan function $J_2(n)$ giving the number of order $n$ elements in $\ZZ_n^2$, having the explicit expression $J_2(n)=n^2\prod_{p|n}\left(1-\frac{1}{p^2}\right)$;
    \item the sum of divisors function $\sigma(n)=\sum_{k|n}k$;
    \item the deformed sum of divisors functions $\overline{\sigma}^d(n)=\sigma(n/d)$, with value $0$ if $d$ does not divide $n$.
\end{itemize}

Explicit formulas are given in the following statement. The uncorrelated case has already been considered in \cite{blomme2022floor} or \cite{blomme2024bielliptic} and actually stems from an elementary computation using coverings of an elliptic curve. Refining this approach to the correlated case, in \cite{blomme2024correlated} we proved an explicit formula for the correlated case. Below, we recover both formulas using the (correlated) DR-formula.

\begin{prop}
    We have the following identity:
$$N_d(\bfw) = a_1^2d^{n-1}\sigma(d).$$
\end{prop}

\begin{proof}
Since we have a unique interior marked point $p_0$ we can reduce the computation of the invariants to the an intersection of the DR cycle with the remaining constraints in the homology of maps to $E$ (see for example \cite{maulik2006topological}):
    \[N_d(\bfw) = \DR_{1,\bfw}(E,d)\cap\prod_{i\neq 1}\ev_i^*(\pt),\]
    where $i\neq 1$ means $i=0,2,3,\dots,n$. We use Pixton formula for the DR-cycle. As we are dealing with genus $1$ curves mapping to an elliptic curve with point insertions, we have the following restrictions on the graphs that appear in the formula:
    \begin{itemize}[label=$\bullet$]
        \item Rational curves maps to $E$ with degree $0$. Therefore, since the genus is $1$ and degree $d\neq 0$, the only graphs that may appear are trees, with a unique genus $1$ vertex carrying degree $d$. Other vertices have genus $0$ and degree $0$.
        \item Since we have point insertions at every marked point but one, and because a degree $0$ curve may not pass through two points at the same time, a genus $0$ component may carry at most one of the markings $0,2,\dots,n$. Due to the stability condition, each of its components needs to carry at least three special points, and thus this component can only be a single $\mathbb P^1$ which also carries the $0$ marking.
        \item In the end, we only have two kinds of graphs:
            \begin{itemize}[label=$\triangleright$]
            \item The graph $\Gamma_\mathrm{sm}$ with a unique vertex, corresponding to the smooth strata with a unique genus $1$ vertex. The decoration provided by the DR-formula is given by $\sum \frac{a_i^2}{2}\psi_i$.
            \item The graph $\Gamma_i$ for $i=0,2,\dots,n$ where the graph has two vertices: a genus $0$ one carrying markings $1$ and $i$, and a genus $1$ vertex with the other markings. There is a unique $r$-weighting, so the decoration provided by the Pixton formula is
            $$\left.\sum_w\frac{w(h)w(h')}{2}\right|_{r=0} = -\frac{1}{2}(a_i+a_1)^2.$$
            \end{itemize}
    \end{itemize}
    We need the following three Gromov-Witten invariants, computed by repeated use of the divisor and dilation equation:
    \begin{align*}
    	\int_{\vir{\M_{1,n}(E,d)}}\prod_{i=1}^n\ev_i^*(\pt) = &  \sigma(d)\cdot d^{n-1},\\
        \int_{\vir{\M_{1,n+1}(E,d)}}\psi_i\prod_{i\neq 1}\ev_i^*(\pt) = &  \sigma(d)\cdot d^{n-1},\\
        \int_{\vir{\M_{1,n+1}(E,d)}}\psi_1\prod_{i\neq 1}\ev_i^*(\pt) = &  \sigma(d)\cdot d^{n-1}\cdot n.\\
    \end{align*}
   
    Putting together the results, we get the following sum. Inside the bracket, the four terms correspond respectively to $\Gamma_\mathrm{sm}$ and decoration $\psi_1$, $\Gamma_\mathrm{sm}$ and decorations $\psi_i$ with $i\geqslant 2$, $\Gamma_0$ and $\Gamma_i$ with $i\geqslant 2$.
    \begin{align*}
        N_d(\bfw) = & \sigma(d)\cdot d^{n-1}\cdot\left(n\frac{a_1^2}{2} + \sum_2^n \frac{a_i^2}{2} - \frac{a_1^2}{2}-\sum_2^n \frac{(a_1+a_i)^2}{2} \right) \\
        = & \sigma(d)\cdot d^{n-1}\cdot \left( -\sum_2^n a_1a_i\right) \\
        = & \sigma(d)\cdot d^{n-1}\cdot a_1^2 \text{ since }a_1=-\sum_2^n a_i.\\
    \end{align*}
    \end{proof}

We now get to the correlated version.

\begin{theo}\label{theo-computation-points}
We have the following identity
$$N^0_d(\bfw) = \left(\frac{a_1}{\delta}\right)^2d^{n-1}\sum_{\omega|\delta}J_2(\omega)\overline{\sigma}^\omega(d) = a_1^2\cdot d^{n-1}\cdot \sum_{l|d}\frac{d}{l}\cdot \frac{\gcd(l,\delta)^2}{\delta^2}.$$
\end{theo}

    \begin{proof}
    We use the correlated DR-formula. To shorten notation, we denote by $\M(E)$ the moduli space $\M_{1,n+1}(E,d)$, with a subscript $K$ when we restrict to a component where line bundle from $K$ pull-back trivially. As in the non correlated case, the graphs associated to the intersected strata all have genus $0$ (i.e.  are trees) for the degree considerations made above. In particular there is no tropical part nor tropical divisors or choice of $\widetilde{K}$ and $\varphi$. In particular, the strata decoration takes the same form as in the standard DR-formula with target. In other words, restricting to the component $\M_{1,n+1,K}(E,d)$ associated to a given subgroup $K$ of $\Pic^0(E)[\delta]$, we have
    $$\DR_{1,\bfw}^0(E,d)|_{\M_K(E)} \cap\prod_{i\neq 1}\ev_i^*(\pt) = |K|\cdot\DR_{1,\bfw/\delta}(E,d)|_{\M_K(E)}\cap\prod_{i\neq 1}\ev_i^*(\pt).$$
    Therefore, the formula takes the following form:
    $$N^0_d(\bfw) = \sum_{K}|K|\cdot\DR_{1,\bfw/\delta}(E,d)_{\M_K(E)}\cap\prod_{i\neq 1}\ev_i^*(\pt),$$
    where the sum is over subgroups of $\Pic^0(E)[\delta]\simeq\ZZ_\delta^2$.

    \medskip
    
    Let $H$ be a subgroup of $\Pic^0(E)[\delta]$ and let $\pi:E_H\to H$ be the associated $H$-covering. In particular we have that $\pi^*(\pt)=|H|\pt$ where $\pt$ is the Poincar\'e dual class to a point in $E$ or $E_H$. Furthermore, the only class pushing forward to $d[E]$ is $\frac{d}{|H|}[E_H]$. From the description of $\M^H(E)$ as moduli space of $H$-rubber maps to the  covering (see Section~\ref{sec-H-covering}), we have the following relation
    \begin{align*}
         \sum_{K\supset H}\DR_{1,\bfw/\delta}(E,d)|_{\M_K(E)}\cap\prod_{i\neq 1}\ev_i^*(\pt) = & \DR_{1,\bfw/\delta}(E,d)|_{\M^H(E)}\cap\prod_{i\neq 1}\ev_i^*(\pt)  \\
        = & \frac{1}{|H|}\cdot \DR_{1,\bfw/\delta}(E_H,d/|H|)\cap \prod_{i\neq 1}\ev_i^*(\pi^*\pt) \\
        = & \frac{|H|^n}{|H|}\left(\frac{d}{|H|}\right)^{n-1}\left(\frac{a_1}{\delta}\right)^2\sigma\left(\frac{d}{|H|}\right) \\
        = & \left(\frac{a_1}{\delta}\right)^2 d^{n-1}\sigma\left(\frac{d}{|H|}\right).
    \end{align*}
    Now, to compute the correlated invariant, we replace $|K|$ by $\sum_{\L\in K}1$ and switch the order of the two sums:
    \begin{align*}
        N^0_d(\bfw) = & \sum_K\sum_{\L\in K} \DR_{1,\bfw/\delta}|_{\M_K(E)}\cap\prod_{i\neq 1}\ev_i^*(\pt) \\
        = & \sum_\L \sum_{K\supset\gen{\L}} \DR_{1,\bfw/\delta}|_{\M_K(E)}\cap\prod_{i\neq 1}\ev_i^*(\pt) \\
        = & \sum_\L \left(\frac{a_1}{\delta}\right)^2d^{n-1}\sigma\left(\frac{d}{|\gen{\L}|}\right),
    \end{align*}
    where we denote by $\gen{\L}$ the subgroup of $\Pic^0(E)$, which we can identify with $E$, generated by $\L$. As $\Pic^0(E)[\delta]\simeq\ZZ_\delta^2$, we have $J_2(\omega)$ elements of order $\omega$. Grouping together the $\L$ according to the value of $|\gen{\L}|$, we get
    $$N^0_d(\bfw) = \left(\frac{a_1}{\delta}\right)^2d^{n-1}\sum_{\omega|\delta}J_2(\omega)\sigma\left(\frac{d}{\omega}\right),$$
    which is the expected result. To get the second expession, we use that $\sum_{\omega|\delta}J_2(\omega) = \delta^2$:
$$\sum_{\omega|\delta}J_2(\omega)\sigma\left(\frac{d}{\omega}\right) = \sum_{\omega|\delta}J_2(\omega)\sum_{k|d/\omega}\frac{a}{k\omega} = \sum_{l|d}\frac{d}{l}\sum_{\omega|l,\delta}J_2(\omega) = \sum_{l|d}\frac{d}{l}\cdot\gcd(l,\delta)^2.$$
\end{proof}

The recovering of all other correlated invariants using decorrelation relations is already handled in \cite{blomme2024correlated}.

\subsection{Point insertions with a $\lambda$-class}

We now increase the genus from $1$ to $g$, and to match the virtual dimensions we insert an additional $\lambda$-class constraint $\lambda_{g-1}$, which is the $(g-1)$th-chern class of the Hodge bundle:
$$N_{g,d}(\bfw) = (-1)^{g-1}\gen{\lambda_{g-1};\pt_0,1_{a_1},\pt_{a_2},\dots,\pt_{a_n}}^{Y/D^\pm}_{1,d[E]},$$
and the correlator $0$ counterpart for level $\delta$ is denoted by:
$$N^0_{g,d}(\bfw) = (-1)^{g-1}\gen{\lambda_{g-1};\pt_0,1_{a_1},\pt_{a_2},\dots,\pt_{a_n}}^{Y/D^\pm,0}_{1,d[E]}.$$
In particular, for $g=1$ we recover the previous invariants. The top chern class $\lambda_g$ is already shown to vanish on $\M_{g,n}(E,d)$ for $d\neq 0$ due to the existence of a non-vanishing holomorphic form on $E$ which can be pulled-back to yield a section of the Hodge bundle, see \cite[Lemma 6.4]{oberdieck2023quantum} or \cite[Lemma 3.9]{blomme2024bielliptic}.

\subsubsection{Invariants in the non-correlated case}
We consider the following generating series of Gromov-Witten invariants:
$$R_{d}(\bfw) = \sum_{g=1}^\infty N_{g,d}(\bfw) u^{2g-2+n} \in\QQ[[u]].$$

These invariants have been computed in \cite[Section 3.4]{blomme2024bielliptic} using tropical geometry techniques from \cite{bousseau2019tropical}, as well as in \cite{oberdieck2023quantum} using the DR-cycle formula. In this situation, similarly to the genus $1$ case with point insertions, it is sufficient to use Haine's formula for compact type curve for the following reason:
\begin{itemize}
    \item The $\lambda$-class $\lambda_{g-1}$ vanishes whenever the strata graph has genus bigger than $2$.
    \item Given a strata associated to a genus $1$ graph, following \cite[Prop 3.1]{bousseau2021floor} the $\lambda$-class splits over the vertices, all inheriting a top $\lambda$-class $\lambda_{g_v}$. At least one vertex $v$ as non-zero degree and positive genus so its $\lambda$-class thus vanishes for the previously mentioned reason.
    \item Consequently, all curves that meet the constraints are compact type. Furthermore, the $\lambda$-class splits over the vertices, all inheriting a top $\lambda$-class $\lambda_{g_v}$ except one inheriting $\lambda_{g_v-1}$. This unique vertex is thus the only one having non-zero degree. In fact, it also must carry all the genus. In the end, the sum is over the same graphs as in the point insertion case, but with different decorations.
\end{itemize}

To express the result, we denote by $[n]$ the $q$-analog of $n$: $[n]=q^{n/2}-q^{-n/2}$, which is a Laurent polynomial in $q^{1/2}$, where $q=e^{iu}$.

\begin{prop}\cite[Prop 3.12]{blomme2024bielliptic}
    Through the change of variable $q=e^{iu}$, we have the following identity
    $$R_d(\bfw)=(-i)^na_1^2\sum_{d=kl}\left(\frac{d}{k}\right)^{n-1}\prod_1^n\frac{[ka_i]}{a_i}.$$
\end{prop}

In particular, the generating series of these generating series has the following expression:
$$\R(\bfw) = \sum_{d=1}^\infty R_d(\bfw)y^d = (-i)^na_1^2 \sum_{k,l}l^{n-1}\left(\prod_1^n\frac{[ka_i]}{a_i}\right)y^{kl} \in\QQ[[u,y]].$$
Expanding $q=e^{iu}$, coefficients of this double generating series are the invariants depending on the degree $d$ and the genus $g$. In \cite{oberdieck2023quantum}, G.~Oberdieck and A.~Pixton provide a different expression for these invariants.

\begin{prop}\cite[Theorem 6.10]{oberdieck2023quantum}
    We have the following identity:
    $$\sum_{d=1}^\infty N_{g,d}(\bfw)y^d
    = \frac{a_1^2}{a_1\cdots a_n}\sum_{S\subset\{1,\cdots,n\}}(-1)^{|S|}a_S^{2g-2+n}\frac{(-1)^{n+g-1}}{(n+2g-2)!}D^{n-1}G_{2g}(y),$$
    where $G_{2g}(y)$ is the Eisenstein series, $D$ is the differential operator $y\frac{\mathrm{d}}{\mathrm{d}y}$ and $a_S=\sum_{i\in S} a_i$.
\end{prop}

\begin{rem}
    Actually, the result from \cite{oberdieck2023quantum} deals with the cap product with the DR-cycle so that it's $D^{n-2}$ instead of $D^{n-1}$ appearing. So to get the statement above, one just has to multiply by the degree, which amounts to apply the operator $y\frac{\mathrm{d}}{\mathrm{d}y}$.
\end{rem}

\begin{lem}
    The expressions from \cite{oberdieck2023quantum} and \cite{blomme2024bielliptic} are compatible. The common coefficient is given by
    $$N_{g,d}(\bfw) = \frac{a_1^2}{a_1\cdots a_n}\left(\sum_{S\subset\{1,\cdots,n\}}(-1)^{|S|}a_S^{2g-2+n}\frac{(-1)^{n+g-1}}{(n+2g-2)!}\right)\cdot d^{n-1}\sigma_{2g-1}(d),$$
    where $\sigma_{2g-1}(d) = \sum_{k|d}\left(\frac{d}{k}\right)^{2g-1}$ is the $(2g-1)$-power sum of divisors function.
\end{lem}

\begin{proof}
    We explain how to pass from one presentation to another. To do so, we expand
    \begin{align*}
        \prod_1^n\frac{[ka_i]}{a_i} = & \frac{1}{a_1\dots a_n}\prod_1^n (e^{iuka_i/2}-e^{-iuka_i/2}) \\
        = & \frac{1}{a_1\dots a_n}\sum_{S\subset\{1,\dots,n\}} (-1)^{n-|S|}e^{iuk(a_S-a_{S^c})/2},
    \end{align*}
    where $S^c$ is the complement of $S$. Since $a_S=-a_{S^c}$, we get
    \begin{align*}
        \prod_1^n\frac{[ka_i]}{a_i} = & \frac{1}{a_1\dots a_n}\sum_S (-1)^{n-|S|}e^{iuka_S} \\
        = & \frac{1}{a_1\dots a_n}\sum_S (-1)^{n+|S|}\sum_j \frac{(iuka_S)^j}{j!}.
    \end{align*}
    Since before expansion the function is a product of $\sin(ka_i u/2)$, it vanishes till order $n$ and all the non zero terms have the parity of $n$. 
   
    Thus, we may write the exponent $j$ as $n+2g-2$ for $g\geqslant 1$:
    \begin{align*}
        \prod_1^n\frac{[ka_i]}{a_i} = & \frac{1}{a_1\dots a_n}\sum_S \sum_{g=1}^\infty (-1)^{n+|S|}\frac{(iuka_S)^{n+2g-2}}{(n+2g-2)!}.
    \end{align*}
    Furthermore, we have the expansion $D^{n-1}G_{2g}(y) = \sum_{l,k} (kl)^{n-1}k^{2g-1}y^{kl}$. With this knowledge, we can now factor the generating series in $g$ of the expression from \cite{oberdieck2023quantum} after expanding $D^{n-1}G_{2g}(y)$:
    \begin{align*}
         & \frac{a_1^2}{a_1\cdots a_n}\sum_{g=1}^\infty\sum_S (-1)^{|S|}a_S^{2g-2+n}\frac{(-1)^{n+g-1}}{(n+2g-2)!}D^{n-1}G_{2g}(y)u^{n+2g-2} \\
        = & \frac{a_1^2}{a_1\cdots a_n}\sum_S\sum_{g=1}^\infty\sum_{k,l}(-1)^{|S|}a_S^{n+2g-2}\frac{(-1)^{n+g-1}}{(n+2g-2)!}l^{n-1}k^{n+2g-2}u^{n+2g-2}y^{kl} \\
        = & (-i)^n a_1^2\sum_{k,l} l^{n-1}\frac{1}{a_1\dots a_n}\sum_S\sum_{g=1}^\infty (-1)^{n+|S|} \frac{(ikua_S)^{n+2g-2}}{(n+2g-2)!}y^{kl} \\
        = & (-i)^na_1^2 \sum_{l,k}l^{n-1}\left(\prod_1^n\frac{[ka_i]}{a_i}\right)y^{kl}.
    \end{align*}
    To get the coefficient formula, one just replaces $D^{n-1}G_{2g}(y)$ by its $y^d$-coefficient.
\end{proof}

\subsubsection{Invariants in the correlated case} Given a common divisor $\delta$ of $a_1,\dots,a_n$, we now care about correlator $0$ for level $\delta$-refinement, considering the following generating series:
$$R_d^0(\bfw) = \sum_{g=1}^\infty N^0_{g,d}(\bfw) u^{2g-2+n} \in\QQ[[u]],$$
along with the double generating series
$$\R^0(\bfw) = \sum_{d=1}^\infty R_d^0(\bfw) y^d = \sum_{d,g\geqslant 1} N^0_{g,d}(\bfw) u^{n+2g-2}y^d \in\QQ[[u,y]].$$

\begin{theo}\label{theo-computation-lambda}
    Through the change of variable $q=e^{iu}$, the generating series of correlated invariants with correlator $0$ has the following expression:
    $$\R^0(\bfw) = (-i)^n\left(\frac{a_1}{\delta}\right)^2\sum_{\omega|\delta}J_2(\omega)\sum_{k,l} (\omega l)^{n-1}\left(\prod_1^n\frac{[ka_i]}{a_i}\right)y^{\omega kl} .$$
    The coefficient is given by the following explicit formula
    $$N^0_{g,d}(\bfw) = \frac{a_1^2}{a_1\cdots a_n}\left(\sum_{S\subset\{1,\cdots,n\}}(-1)^{|S|}a_S^{2g-2+n}\frac{(-1)^{n+g-1}}{(n+2g-2)!}\right)\cdot d^{n-1}\sum_{k|d}\left(\frac{d}{k}\right)^{2g-1}\frac{\gcd(k,\delta)^2}{\delta^2}.$$
\end{theo}

\begin{proof}
    As in the case of point insertions, but this time also due to the vanishing properties of the $\lambda$-classes recalled above, only  strata associated to compact type curves contribute non trivially to the cap product with DR-cycle: again, no tropical part, no tropical torsion, no $\varphi$ and no choice of $\widetilde{K}$ which has to be equal to $K$. Therefore, as in the point case, we have that the $0$-correlator DR-cycle restricted to the component of the moduli space indexed by $K$ is given by:
    $$\DR^0_{g,\bfw}(E,d)|_{\M_K(E)}\cap\lambda_{g-1}\prod_{i\neq 1}\ev_i^*(\pt) = \delta^{2g-2}|K|\cdot \DR_{g,\bfw/\delta}(E,d)|_{\M_K(E)}\cap\lambda_{g-1}\prod_{i\neq 1}\ev_i^*(\pt).$$
    We thus have the following:
    $$N^0_{g,d}(\bfw) = \sum_K \delta^{2g-2}|K|\DR_{g,\bfw/\delta}(E,d)|_{\M_K(E)}\cap \lambda_{g-1}\prod_{i\neq 1}\ev_i^*(\pt). $$
    To ease notation, we denote the cap product by $f_K(g,d)$, so that
    $$N^0_{g,d}(\bfw)=\sum_K\delta^{2g-2}|K|f_K(g,d).$$
    Let $H$ be a subgroup of $\Pic^0(E)[\delta]$. The computation for the sub-stack $\M_{g,n+1}^H(E,d)$  (which we recalled in the previous proof) gives the following identity:
    \begin{align*}
        \sum_{K\supset H}f_K(g,d) = & \sum_{K\supset H} \DR_{g,\bfw/\delta}(E,d)|_{\M_K(E)}\cap\lambda_{g-1}\prod_{i\neq 1}\ev_i^*(\pt) \\
        = & \frac{1}{|H|}\vir{\M_{g,n+1}(E_H,d/|H|)}\cap \lambda_{g-1}\prod_{i\neq 1}\ev_i^*(\pi^*\pt) \\
        = & \frac{|H|^n}{|H|}N_{g,d/|H|}(\bfw/\delta) = |H|^{n-1}N_{g,d/|H|}(\bfw/\delta).
    \end{align*}
    We can now perform the generating series in $g$ and $d$ and use the same trick as in the case of point insertions to trade the $|K|$ for a sum over $\L\in\Pic^0(E)[\delta]$:
    \begin{align*}
        \sum_{d,g}N^0_{g,d}(\bfw)u^{n+2g-2}y^d = & \sum_{d,g}\left(\sum_K\delta^{2g-2}|K|f_K(g,d)\right)u^{n+2g-2}y^d \\
        = & \sum_{d,g}\delta^{2g-2}\left(\sum_K\sum_{\L\in K}f_K(g,d)\right)u^{n+2g-2}y^d \\
        = & \sum_{d,g}\delta^{2g-2}\left(\sum_\L \sum_{K\supset\gen{\L}}f_K(g,d)\right)u^{n+2g-2}y^d \\
        = & \sum_\L\sum_{d,g}\delta^{2g-2}|\gen{\L}|^{n-1}N_{g,d/|\gen{\L}|}(\bfw/\delta)u^{n+2g-2}y^d.
    \end{align*}
    Grouping together the $\L$ with the same order, we get
    \begin{align*}
        \sum_{d,g}N^0_{g,d}(\bfw)u^{n+2g-2}y^d = & \sum_{\omega|\delta} J_2(\omega)\sum_{d,g}\delta^{2g-2}\omega^{n-1}N_{g,d/\omega}(\bfw/\delta)u^{n+2g-2}y^d. \\
    \end{align*}
    In particular, in the second sum, we may only sum over the values of $d$ divisible by $\omega$, since otherwise the coefficient $N_{g,d/\omega}(\bfw/\delta)$ is $0$. We thus write $d=\omega d'$:
    \begin{align*}
        \sum_{d,g}N^0_{g,d}(\bfw)u^{n+2g-2}y^d = & \sum_{\omega|\delta} J_2(\omega)\frac{1}{\delta^n}\sum_{d',g}\omega^{n-1}N_{g,d'}(\bfw/\delta)(\delta u)^{n+2g-2}(y^\omega)^{d'} \\
        = & (-i)^n\left(\frac{a_1}{\delta}\right)^2\sum_{\omega|\delta}J_2(\omega)\frac{1}{\delta^n}\sum_{k,l}(\omega l)^{n-1}\left(\prod_1^n\frac{[\delta\cdot ka_i/\delta]}{a_i/\delta}\right)(y^\omega)^{kl} \\
        = & (-i)^n\left(\frac{a_1}{\delta}\right)^2\sum_{\omega|\delta}J_2(\omega)\sum_{k,l}(\omega l)^{n-1}\left(\prod_1^n\frac{[ka_i]}{a_i}\right)y^{\omega kl},
    \end{align*}
    where  we used that replacing $u$ by $\delta u$ has the effect or replacing $q$ by $q^\delta$ and thus $[n]$ by $[\delta n]$. For the explicit expression, we just take a fixed coefficient. Alternatively, performing the same trick we have
    \begin{align*}
        N^0_{g,d}(\bfw) = & \sum_K \delta^{2g-2}|K|f_K(g,d) = \delta^{2g-2}\sum_K\sum{\L\in K}f_K(g,d) = \delta^{2g-2}\sum_\L\sum_{K\ni\L}f_K(g,d) \\
        = & \delta^{2g-2}\sum_{\omega|\delta}J_2(\omega) \omega^{n-1} N_{g,d/\omega}(\bfw/\delta) \\
        = & \delta^{2g-2}\sum_{\omega|\delta}J_2(\omega) \omega^{n-1}\frac{(a_1/\delta)^2}{a_1\cdots a_n}\delta^n\left(\sum_S (-1)^{|S|}\left(\frac{a_S}{\delta}\right)^{2g-2+n}\frac{(-1)^{n+g-1}}{(n+2g-2)!}\right) \\
         & \cdot\left(\frac{d}{\omega}\right)^{n-1}\sum_{k|d/\omega}\left(\frac{d}{k\omega}\right)^{2g-1} \\
        = & \frac{a_1^2}{a_1\cdots a_n}\left(\sum_S (-1)^{|S|}a_S^{2g-2+n}\frac{(-1)^{n+g-1}}{(n+2g-2)!}\right)\cdot d^{n-1}\sum_{\omega|\delta}J_2(\omega)\sum_{k|d/\omega}\left(\frac{d}{k\omega}\right)^{2g-1} \\
        = & \frac{a_1^2}{a_1\cdots a_n}\left(\sum_S (-1)^{|S|}a_S^{2g-2+n}\frac{(-1)^{n+g-1}}{(n+2g-2)!}\right)\cdot d^{n-1}\sum_{l|d}\left(\frac{d}{l}\right)^{2g-1}\cdot\frac{\gcd(l,\delta)^2}{\delta^2}.
    \end{align*}
\end{proof}